\documentclass[12pt]{amsart}

\usepackage[]{cmbright} 
\usepackage{bm,amsfonts,amsmath,amssymb,dsfont,mathrsfs,bbm} 
\usepackage{graphicx} 

\makeatletter
\def\section{\@startsection{section}{1}\z@{.9\linespacing\@plus\linespacing}%
  {.7\linespacing} {\fontsize{13}{14}\selectfont\bfseries\centering}}
\def\paragraph{\@startsection{paragraph}{4}%
  \z@{0.3em}{-.5em}%
  {$\bullet$ \ \normalfont\itshape}}

\@addtoreset{equation}{section}
\makeatother

\usepackage{color}

\definecolor{gr}{rgb}   {0.,   0.69,   0.23 }
\definecolor{bl}{rgb}   {0.,   0.5,   1. }
\definecolor{cy}{rgb}   {0.,   0.57,   0.67 }
\definecolor{mg}{rgb}   {0.85,  0.,    0.85}
\definecolor{marron}{rgb}  {0.6, 0.40, 0.1} 
\definecolor{or}{rgb}   {0.9,  0.5,   0.}

\definecolor{webred}{rgb}{0.75,0,0}
\definecolor{webgreen}{rgb}{0,0.75,0}
\usepackage[citecolor=webgreen,colorlinks=true,linkcolor=webred]{hyperref}

\newtheorem{theorem}{Theorem}[section]
\newtheorem{proposition}[theorem]{Proposition}
\newtheorem{lemma}[theorem]{Lemma}
\newtheorem{corollary}[theorem]{Corollary}

\theoremstyle{definition}
\newtheorem{definition}[theorem]{Definition}

\theoremstyle{remark}
\newtheorem{remark}[theorem]{Remark}
\newtheorem{example}[theorem]{Example}

\newcommand{\R}{\mathbb{R}}

\newcommand{\cA}{\mathcal{A}}
\newcommand{\cB}{\mathcal{B}}

\newcommand{\cD}{\mathcal{D}}

\newcommand{\cF}{\mathcal{F}}

\newcommand{\cI}{\mathcal{I}}
\newcommand{\cO}{\mathcal{O}}

\newcommand{\cS}{\mathcal{S}}

\newcommand{\cU}{\mathcal{U}}
\newcommand{\cV}{\mathcal{V}}
\newcommand{\cW}{\mathcal{W}}

\newcommand{\gA}{\mathfrak{A}}
\newcommand{\gC}{\mathfrak{C}}
\newcommand{\gD}{\mathfrak{D}}
\newcommand{\gE}{\mathfrak{E}}
\newcommand{\gF}{\mathfrak{F}}
\newcommand{\gJ}{\mathfrak{J}}
\newcommand{\gO}{\mathfrak{O}}
\newcommand{\gP}{\mathfrak{P}}
\newcommand{\gS}{\mathfrak{S}}
\newcommand{\gT}{\mathfrak{T}}
\newcommand{\gV}{\mathfrak{V}}

\newcommand{\ogD}{\overline\gD}
\newcommand{\ogP}{\overline\gP}

\newcommand{\dx}{\mathbb{X}}
\newcommand{\dS}{\mathbb{S}}

\newcommand{\spm}{\,\widehat{\!\psi\!}\,}
\newcommand{\dec}{{\bp}}
\newcommand{\tbB}{\widetilde{\bB}}
\newcommand{\bA}{{\boldsymbol{\mathsf{A}}}}
\newcommand{\bB}{{\boldsymbol{\mathsf{B}}}}

\newcommand{\bd}{{\boldsymbol{\mathsf{d}}}}
\newcommand{\be}{{\boldsymbol{\mathsf{e}}}}
\newcommand{\bff}{{\boldsymbol{\mathsf{f}}}}
\newcommand{\bp}{{\boldsymbol{\mathsf{p}}}}

\newcommand{\bt}{{\boldsymbol{\mathsf{t}}}}
\newcommand{\bu}{{\boldsymbol{\mathsf{u}}}}
\newcommand{\bv}{{\boldsymbol{\mathsf{v}}}}

\newcommand{\bx}{{\boldsymbol{\mathsf{x}}}}
\newcommand{\by}{{\boldsymbol{\mathsf{y}}}}
\newcommand{\bz}{{\boldsymbol{\mathsf{z}}}}

\newcommand{\bZ}{{\boldsymbol{\mathsf{Z}}}}

\newcommand{\bfz}{{\boldsymbol{0}}}

\newcommand{\sA}{\mathscr{A}}
\newcommand{\sC}{\mathscr{C}}
\newcommand{\sE}{\mathscr{E}}

\newcommand{\sP}{\mathscr{P}}

\newcommand{\sZ}{\mathscr{Z}}

\newcommand{\rd}{\mathrm{d}}
\newcommand{\re}{\mathrm{e}}
\newcommand{\ri}{i}
\newcommand{\rD}{\mathrm{D}}
\newcommand{\rG}{\mathrm{G}}
\newcommand{\rJ}{\mathrm{J}}

\newcommand{\uA}{\underline{\bA}}
\newcommand{\uB}{\underline{\bB}}

\newcommand{\udiffeo}{\ee\underline{\me\mathrm U\me}\ee} 
\newcommand{\uPsi}{\,\underline{\!\Psi\!}\,}
\newcommand{\upsi}{\underline{\psi\!}\,}

\newcommand{\ess}{\mathrm{ess}}
\newcommand{\loc}{\mathrm{loc}}
\newcommand{\ee}{\hskip 0.15ex}
\newcommand{\me}{\hskip -0.15ex}

\newcommand{\tronc}{\chi_{j}} 
\newcommand{\troncg}{\chi_{(\bx,r)}}
\newcommand{\troncp}{\xi_{(\bx,r)}}
\newcommand{\troncz}{\xi_{(\bx_0,r_0)}}

\newcommand\dom{\operatorname{Dom}}
\newcommand\supp{\operatorname{supp}}
\newcommand\curl{\operatorname{curl}}

\newcommand\Id{\operatorname{\mathbb{I}}}
\renewcommand\Re{\operatorname{Re}}

\newcommand{\OP}{H} 
\newcommand{\DG}[1]{\mathcal{H}(#1)} 
\newcommand{\En}{E} 
\newcommand{\seE}{\mathscr{E}^*} 
\newcommand{\pot}{\widetilde{\bA}{}} 

\newcommand{\diffeo}{\mathrm U}

\newcommand{\diffeoT}{\mathrm T}

\newcommand{\dir}{\boldsymbol{\tau}}

\title[Ground energy of the magnetic Laplacian in polyhedral bodies]{Ground energy of the magnetic Laplacian\\ in polyhedral bodies}
\author{Virginie Bonnaillie-No\"el, Monique Dauge, Nicolas Popoff}

\begin{document}

\begin{abstract}
The asymptotic behavior of the first eigenvalues of magnetic Laplacian operators with large magnetic fields and Neumann realization in polyhedral domains is characterized by a hierarchy of model problems. We investigate properties of the model problems (continuity, semi-continuity, existence of generalized eigenfunctions). We prove estimates for the remainders of our asymptotic formula. Lower bounds are obtained with the help of a classical IMS partition based on adequate coverings of the polyhedral domain, whereas upper bounds are established by a novel construction of quasimodes, qualified as sitting or sliding according to spectral properties of local model problems.
\end{abstract}

\today
\maketitle

{\parskip 0.5pt
\tableofcontents
}

\section{Introduction. Main results}
\label{sec:intro}

The Schr\"odinger operator with magnetic field (also called magnetic Laplacian) takes the form
\[
   (-i\nabla+\bA)^2
\]
where $\bA$ is a given vector field that will be assumed to be regular. When set on a domain $\Omega$ of $\R^n$ ($n=2$ or $3$) and completed by natural boundary conditions (Neumann), this operator is denoted by $\OP(\bA,\Omega)$. If $\Omega$ is bounded with Lipschitz boundary, the form domain of $\OP(\bA,\Omega)$ is the standard Sobolev space $H^1(\Omega)$ and $\OP(\bA,\Omega)$ is positive self-adjoint with compact resolvent. The ground state of $\OP(\bA,\Omega)$ is the eigenpair $(\lambda,\psi)$ 
\begin{equation}
\label{eq:pb1}
   \begin{cases}
   (-i\nabla+\bA)^2\psi=\lambda\psi\ \ &\mbox{in}\ \ \Omega, \\
   (-i\partial_n+\mathbf{n}\cdot\bA)\psi=0\ \ &\mbox{on}\ \ \partial\Omega,
   \end{cases}
\end{equation}
associated with the lowest eigenvalue $\lambda$. 
If $\Omega$ is simply connected, its eigenvalues only depend on the magnetic field $\bB$ defined as
\begin{equation}
\label{eq:B}
   \bB = \curl\bA.
\end{equation}
The eigenvectors corresponding to two different instances of $\bA$ for the same $\bB$ are deduced from each other by a {\em gauge transform}. 

Introducing a (small) parameter $h>0$ and setting
\[
   \OP_{h}(\bA,\Omega) = (-ih\nabla+\bA)^2
   \quad\mbox{with Neumann b.c. on $\partial\Omega$},
\]
we get the relation
\begin{equation}
\label{eq:h2}
   \OP_{h}(\bA,\Omega) = h^2 \OP\Big(\frac{\bA}{h},\Omega\Big)
\end{equation}
linking the problem with large magnetic field to the semiclassical limit $h\to0$. We denote by $\lambda_h(\bB,\Omega)$ (or $\lambda_h$ if no confusion is possible) the smallest eigenvalue of $\OP_{h}(\bA,\Omega)$ and by $\psi_h$ an associated eigenvector, so that
\begin{equation}
\label{eq:pbh}
   \begin{cases}
   (-ih\nabla+\bA)^2\psi_h=\lambda_h\psi_h\ \ &\mbox{in}\ \ \Omega \,,\\
   (-ih\partial_n+\mathbf{n}\cdot\bA)\psi_h=0\ \ &\mbox{on}\ \ \partial\Omega\,.
   \end{cases}
\end{equation}
The behavior of $\lambda_h$ as $h\to0$ clearly provide equivalent information about the lowest eigenvalue of $\OP(\breve\bA,\Omega)$ when $\breve\bB$ is large, especially in the parametric case when $\breve\bB = B\ee\bB$ where the real number $B$ tends to $+\infty$ and $\bB$ is a chosen reference magnetic field.

From now on, we consider that $\bB$ is fixed. We assume that it is smooth and does not vanish\footnote{Should $\bB$ cancel, the situation would be very different, leading to another type of asymptotics \cite{HeMo96,DoRa13}.} on $\overline\Omega$. 
The question of the semiclassical behavior of $\lambda_h$ has been considered in many papers for a variety of domains, with constant or variable magnetic fields: Smooth domains \cite{LuPan99-2, HeMo01, FouHe06,Ara07, Ray09} and polygons \cite{Bo05, BonDau06, BoDauMaVial07} in dimension $n=2$, and mainly smooth domains \cite{LuPan00, HeMo02, HeMo04, Ray09-3d, FouHe10} in dimension $n=3$. Until now, three-dimensional non-smooth domains are only addressed in two particular configurations---rectangular cuboids \cite{Pan02} and lenses \cite{PoTh, PoRay12}, with special orientation of the (constant) magnetic field. We give more detail about the state of the art in Section \ref{sec:art}.

\subsection{Asymptotic formulas with remainders}
Let us briefly describe our main results in the three-dimensional setting.

Each point $\bx$ in the closure of a polyhedral domain $\Omega$ is associated with a dilation invariant, tangent open set $\Pi_\bx$, according to the following cases:
\begin{enumerate}
\item If $\bx$ is an interior point, $\Pi_\bx=\R^3$,
\item If $\bx$ belongs to a {\em face $\bff$} (i.e., a connected component of the smooth part of $\partial\Omega$), $\Pi_\bx$ is a half-space,
\item If $\bx$ belongs to an {\em edge} $\be$, $\Pi_\bx$ is an infinite wedge,
\item If $\bx$ is a {\em vertex} $\bv$, $\Pi_\bx$ is an infinite polyhedral cone.
\end{enumerate}
Let $\bB_\bx$ be the magnetic field frozen at $\bx$. Let $\En(\bB_\bx \ee,\Pi_\bx)$ be the bottom of the spectrum (ground energy) of the tangent operator $\OP(\bA_\bx \ee,\Pi_\bx)$ where $\bA_\bx$ is the linear approximation of $\bA$ at $\bx$, so that
\[
   \curl\bA_\bx = \bB_\bx \,.
\] 
We introduce the quantity
\begin{equation}
\label{eq:s}
   \sE(\bB \ee,\Omega) := \inf_{\bx\in\overline\Omega} \En(\bB_\bx \ee,\Pi_\bx).
\end{equation}
In this paper, we prove that this quantity provides the value of the limit of $\lambda_h/h$ as $h\to0$ with some control of the convergence rate as $h\to0$, namely
\begin{equation}
\label{eq:conv}
   - C h^{5/4} \le  \lambda_h(\bB \ee,\Omega) - h\ee \sE(\bB \ee,\Omega) \le
    C h^{5/4} ,
\end{equation}
where the constant $C$ is bounded by the norm of $\bA$ in $W^{2,\infty}(\Omega)$, as proved in Theorems~\ref{T:generalUB} and \ref{T:generalLB}. We can also control the constant $C$ by the magnetic field $\bB$ as established in Corollaries~\ref{co:T:generalUBB} and \ref{co:T:generalUB}.
With the point of view of large magnetic fields in the parametric case $\breve\bB = B\ee\bB$,  \eqref{eq:conv} yields obviously
\begin{equation}\label{eq.asympB54}
   - C B^{3/4} \le  \lambda(\breve\bB \ee,\Omega) - B\ee \sE(\bB \ee,\Omega) \le C B^{3/4},\quad\mbox{ as } B\to+\infty.
\end{equation}
Note that $B\ee \sE(\bB \ee,\Omega)=\ee \sE(\breve\bB \ee,\Omega)$ by homogeneity (see Lemma~\ref{lem.dilatation}).

If the magnetic potential is more regular $\bA\in W^{3,\infty}(\overline\Omega)$, we establish in Theorem~\ref{T:sUB} a better upper bound: 
\begin{equation}
\label{eq:conv2}
   - C h^{5/4} \le  \lambda_h(\bB \ee,\Omega) - h\ee \sE(\bB \ee,\Omega) \le
    C h^{4/3} ,
\end{equation}
where the constant $C$ is bounded by the norm of $\bA$ in $W^{3,\infty}(\Omega)$ and can be controlled by the magnetic field $\bB$ as mentioned in Corollary~\ref{co:T:generalUBB2}. 
Like for \eqref{eq.asympB54}, we deduce the asymptotics for large magnetic fields $\breve\bB = B\ee\bB$: 
\begin{equation}\label{eq.asympB43}
   - C B^{3/4} \le  \lambda(\breve\bB \ee,\Omega) - B\ee \sE(\bB \ee,\Omega) \le C B^{2/3},
   \quad\mbox{ as } B\to+\infty.
\end{equation}
These results are new in this generality. In view of \cite[Theorem 1.1]{HeMo04} (in the smooth three-dimensional case) the upper bound in \eqref{eq.asympB43} is optimal. The lower bound coincides with the one obtained in the smooth case in dimensions $2$ and $3$ when no further assumptions are done. In the literature, improvements of the convergence rates are possible in certain cases when one knows more on $\sE(\bB \ee,\Omega)$, in particular whether the infimum is attained in some special points. 

Our result does not need such extra assumptions, but our proofs have to distinguish cases whether the local ground energies $\En(\bB_\bx \ee,\Pi_\bx)$ are attained or not, and we have to understand the behavior of the function $\bx\mapsto\En(\bB_\bx \ee,\Pi_\bx)$ when $\bx$ spans the different regions of $\overline\Omega$. We have proved very general continuity and semi-continuity properties as described now.

Let $\gF$ be the set of faces $\bff$, $\gE$ the set of edges $\be$ and $\gV$ the set of vertices of $\Omega$. They form a partition of the closure of $\Omega$, called stratification
\begin{equation}
\label{eq:stratif}
   \overline\Omega = \Omega \cup \big(\bigcup_{\bff\in\gF}\ \bff\ \big)
   \cup \big(\bigcup_{\be\in\gE}\ \be\ \big)
   \cup \big(\bigcup_{\bv\in\gV}\ \bv\ \big) .
\end{equation}
The sets $\Omega$, $\bff$, $\be$ and $\bv$ are called the strata of $\overline\Omega$, compare with \cite{MazyaPlamenevskii77} and \cite[Ch. 9]{NazarovPlamenevskii94}. We denote them by $\bt$ and their set by $\gT$. For each stratum $\bt$, let us denote by $\Lambda_\bt$ the function
\begin{equation}
\label{eq:Lams}
   \Lambda_\bt : \bt\ni \bx \mapsto \En(\bB_\bx \ee,\Pi_\bx).
\end{equation}
We will show the following facts
\begin{enumerate}
\item[a)] The function $\bx\mapsto \En(\bB_\bx \ee,\Pi_\bx)$ is lower semi-continuous on $\overline\Omega$.
\item[b)] For each stratum $\bt\in\gT$, the function $\Lambda_\bt$ is continuous on $\bt$ and can be continuously extended to the closure $\bar\bt$ of $\bt$. Moreover, for each $\bx_0\in\bar\bt$, $\Lambda_\bt(\bx_0)$ is the bottom of the spectrum $\En(\bB_\dx \ee,\Pi_\dx)$ of a tangent magnetic operator $\OP(\bA_\dx,\Pi_\dx)$ associated with a {\em singular chain} $\dx$ originating at $\bx_0$. 
\end{enumerate}
As a consequence, the infimum determining the limit $\sE(\bB,\Omega)$ in \eqref{eq:s} is a minimum
\begin{equation}
\label{eq:s,min}
   \sE(\bB \ee,\Omega) = \min_{\bx\in\overline\Omega} \,\En(\bB_\bx \ee,\Pi_\bx) \,.
\end{equation}

\subsection{Contents of the paper}
In Section \ref{sec:art} we place our results in the framework of existing literature for dimensions $2$ and $3$. In Section \ref{sec:chain} we introduce the wider class of corner domains, alongside with their tangent cones and singular chains. We particularize these notions in the case of three-dimensional polyhedral domains. In Section \ref{sec:tax} we introduce and classify magnetic model problems on tangent cones (taxonomy) and extract from the literature related facts. We show that to each point $x\in\overline\Omega$ is associated a singular chain $\dx$ originating at $\bx$ for which the tangent operator $\OP(\bA_\dx,\Pi_\dx)$ possesses {\em admissible generalized eigenvectors} with energy $\En(\bB_\bx \ee,\Pi_\bx)$. In Section \ref{sec:cont} we prove the semi-continuity and continuity properties of the functions $\bx\mapsto \En(\bB_\bx \ee,\Pi_\bx)$ on $\overline\Omega$ and its strata. In Section \ref{sec:up} we prove the upper bounds $\lambda_h(\bB \ee,\Omega) \le h\ee \sE(\bB \ee,\Omega) + C h^\kappa$, with $\kappa={5/4}$ or $\kappa=4/3$ according to the regularity of $\bA$, by a construction of quasimodes based on admissible generalized eigenvectors for tangent problems. In Section \ref{sec:low} we prove the lower bound $h\ee \sE(\bB \ee,\Omega) - C h^{5/4} \le \lambda_h(\bB \ee,\Omega)$ by a classical IMS formula. 

\subsection{Notations} 
For a generic (unbounded) self-adjoint operator $L$
we denote by $\dom(L)$ its domain and $\gS(L)$ its spectrum. 
Likewise the domain of a quadratic form $q$ is denoted by $\dom(q)$.

Domains as open simply connected subsets of $\R^n$ are denoted by $\cO$ if they are generic, $\Pi$ if they are invariant by dilatation (cones) and $\Omega$ if they are bounded.

In this paper, the quadratic forms of interest are those associated with magnetic Laplacians, namely, for a positive constant $h$, a smooth magnetic potential $\bA$, and a generic domain $\cO$
\begin{equation}
\label{D:fq}
   q_{h}[\bA,\cO](f):=
   \int_{\cO}(-ih\nabla+\bA)f\cdot \overline{(-ih\nabla+\bA)f}\;\rd\bx ,
\end{equation}
and its domain
\begin{equation}
\label{D:fqd}
   \dom(q_{h}[\bA,\cO]) = \{f\in L^2(\cO), \ (-ih\nabla+\bA)f\in L^2(\cO)\}\,.
\end{equation}
For a bounded domain $\Omega$, $\dom(q_{h}[\bA,\Omega])$ coincides with $H^1(\Omega)$. For $h=1$, we omit the index $h$, denoting the quadratic form by $q[\bA,\cO]$. 

In relation with changes of variables, we will also use the more general form with metric:
\begin{equation}
\label{D:fqG}
   q_{h}[\bA,\cO,\rG](\psi)=
   \int_{\cO}(-ih\nabla+\bA)\psi\cdot \rG \big(\,\overline{\!(-ih\nabla+\bA)\psi\!}\,\big)
   \; |\rG|^{-1/2} \,\rd\bx ,
\end{equation}
where $\rG$ is a smooth function with values in $3\times3$ positive symmetric matrices and $|\rG|=\det \rG$.  Its domain is (see \cite[\S5]{HeMo04} for more details)
\[
   \dom(q_{h}[\bA,\cO,\rG]) = \{\psi\in L^2_\rG(\cO), \ \rG^{1/2}(-ih\nabla+\bA)\psi\in L^2_\rG(\cO)\}\,,
\]
where $L^2_{\rG}(\cO)$ is the space of the square-integrable functions for the weight $|\rG|^{-1/2}$ and $\rG^{1/2}$ is the square root of the matrix $\rG$.

The domain of the magnetic Laplacian with Neumann boundary conditions on the  set $\cO$ is
\begin{multline}
\label{eq:dom}
   \dom\left(\OP_h(\bA \ee,\cO)\right) =
   \big\{\psi\in L^2(\cO), \quad \\
   (-ih\nabla+\bA)^2\psi\in L^2(\cO) \ \ \mbox{and}\ \ 
   (-ih\partial_n+\mathbf{n}\cdot\bA)\psi=0\ \mbox{on}\ \partial\cO \big\} \, .
\end{multline}

We will also use the space of the functions which are {\em locally}\footnote{Here $H^m_\loc(\overline{\cO})$ denotes for $m=0,1$ the space of functions  which are in $H^m(\cO\cap\cB)$ for any ball $\cB$.} in the domain of $\OP_h(\bA,\cO)$:
\begin{multline}
\label{D:domloc}
   \dom_{\,\loc} \left(\OP_h(\bA,\cO)\right) := 
   \{\psi\in H^1_\loc(\overline{\cO}), \\  
   (-ih\nabla+\bA)^2\psi\in H^0_\loc(\overline{\cO}) \ \ \mbox{and}\ \  
  (-ih\partial_n+\mathbf{n}\cdot\bA)\psi=0 \ \mbox{on} \ \partial\cO\}.
\end{multline}
When $h=1$, we omit the index $h$ in \eqref{eq:dom} and \eqref{D:domloc}.

\section{State of the art}
\label{sec:art}
 
Here we collect some results of the literature about the semiclassical limit for the first eigenvalue of the magnetic Laplacian depending on the geometry of the domain and the variation of the magnetic field. We briefly mention the case when the domain has no boundary, before reviewing in more detail what is known on bounded domains $\Omega\subset \R^n$ with Neumann boundary condition depending on the dimension $n\in \{2,3\}$. To keep this section relatively short, we only quote results related with our problematics, i.e., the asymptotic behavior of the ground energy with error estimates from above and from below.


\subsection{Without boundary or with Dirichlet conditions}
Here $M$ is either a compact Riemannian manifold without boundary or $\R^n$, and $\OP_{h}(\bA,M)$ is the magnetic Laplacian associated with the 1-form $\bA$. In this very general framework, the magnetic field is the 2-form $\bB=\curl\bA$. Then for each $\bx\in M$ the local energy at $\bx$ is the intensity  
$$b(\bx) := \tfrac12 \operatorname{\rm Tr}([\bB^*(\bx)\cdot\bB(\bx)]^{1/2})$$ 
and $\sE(\bB,M)=b_0:=\inf_{\bx\in M} b(\bx)$. 
It is proved by Helffer and Mohamed in \cite{HeMo96} that if $b_0$ is positive and under a condition at infinity if $M=\R^n$, then 
$$
   \exists C>0, \quad - Ch^{5/4} \leq
   \lambda_h(\bB \ee,M)-h\sE(\bB \ee,M) \leq Ch^{4/3}  \ .
$$
More precise results are proved in dimension $2$ when $b$ admits a unique positive non-degenerate minimum: A complete asymptotic expansion of the eigenvalues of $\OP_{h}(\bA,M)$ in powers of $\sqrt{h}$ has been obtained by Helffer and Kordyukov in \cite{HeKo11}, and improved in  by V\~u Ng\d{o}c and Raymond in \cite{RayVuN13} where it is proved that the sole integer powers of $h$ appear in the expansion. These results imply in particular that with these assumptions there holds 
$$|\lambda_{h}(\bB,M)-h\sE(\bB \ee,M) |\leq Ch^2.$$
The case of Dirichlet boundary condition is very close to the case without boundary.

\subsection{Neumann conditions in dimension 2}
In contrast, when Neumann boundary conditions are applied on the boundary, the local energy drops significantly as was established in \cite{SaGe63} by Saint-James and de Gennes as early as 1963.
In this review of the dimension $n=2$, we classify the domains in two categories: those with a regular boundary and those with a polygonal boundary.

\subsubsection{Regular domains}
Let $\Omega\subset \R^2$ be a regular domain and $B$ be a regular non-vanishing scalar magnetic field on $\overline\Omega$. To each $\bx\in\overline\Omega$ is associated a tangent problem. According to whether $\bx$ is an interior point or a boundary point, the tangent problem is the magnetic Laplacian on the plane $\R^2$ or the half-plane $\Pi_\bx$ tangent to $\Omega$ at $\bx$, with the constant magnetic field $B_\bx\equiv B(\bx)$. The associated spectral quantities $\En(B_\bx,\R^2)$ and $\En(B_\bx,\Pi_\bx)$ are respectively equal to $|B_\bx|$ and $|B_\bx|\Theta_{0}$ where $\Theta_{0}:=\En(1,\R^2_+)$ is a universal constant whose value is close to $0.59$ (see \cite{SaGe63}). With the quantities 
\begin{equation}
\label{D:betbprime}
   b=\inf_{\bx\in\overline \Omega} |B(\bx)| \quad \mbox{and} \quad  
   b'=\inf_{\bx\in \partial\Omega}|B(\bx)|, 
\end{equation}
we find 
$$
   \sE(B \ee,\Omega)=\min(b,b'\Theta_{0}) \ . 
$$
In this generality, the asymptotic limit
\begin{equation}
\label{Metaordre0}
   \lim_{h\to0}\frac{\lambda_h(B \ee,\Omega)}{h} = \sE(B,\Omega) 
\end{equation} 
is proved by Lu and Pan in \cite{LuPan99-2}. Improvements of this result depend on the geometry and the variation of the magnetic field as we describe now.

\paragraph{Constant magnetic field} 
If the magnetic field is constant and normalized to $1$, then $\sE(B \ee,\Omega)=\Theta_{0}$. The following estimate is proved by Helffer and Morame:
$$
  \exists C>0, \quad -Ch^{3/2} \leq \lambda_h(1 \ee,\Omega)-h\Theta_{0} \leq Ch^{3/2} \ ,
$$
for $h$ small enough \cite[\S 10]{HeMo01}, while the upper bound was already given by Bernoff and Sternberg \cite{BeSt98}. 
This result is improved in \cite[\S 11]{HeMo01} in which a two-term asymptotics is proved, showing that a remainder in $O(h^{3/2})$ is optimal. Under the additional assumption 
that the curvature of the boundary admits a unique and non-degenerate maximum, a complete 
expansion of $\lambda_h(1 \ee,\Omega)$ is provided by Fournais and Helffer \cite{FouHe06}.

\paragraph{Variable magnetic field}
Here we recall results from \cite[\S 9]{HeMo01} for variable magnetic fields (we use the notation \eqref{D:betbprime})
\begin{subequations}
\begin{eqnarray}
\label{eq:2.3a}
   \mbox{If $b<\Theta_{0}b'$,} \quad &
   \exists C>0, \quad & |\lambda_h(B \ee,\Omega)-hb| \leq Ch^{2} , \\[0.2ex]
\label{eq:2.3b}
   \mbox{If $b>\Theta_{0}b'$,} \quad &
   \exists C>0,  \quad 
   & -Ch^{5/4} \leq \lambda_h(B \ee,\Omega)-h\Theta_{0}b' \leq Ch^{3/2} , \\[0.2ex]
\label{eq:2.3c}
   \mbox{If $b=\Theta_{0}b'$,} \quad &
   \exists C>0,  \quad 
   & -Ch^{5/4}  \leq \lambda_h(B \ee,\Omega)-hb\leq Ch^{2},
\end{eqnarray}
\end{subequations}
Under non-degeneracy hypotheses, the optimality of the interior estimates \eqref{eq:2.3a} is a consequence of \cite{HeKo11}, whereas the eigenvalue asymptotics provided in \cite{Ray09,Ray13} yield that the upper bound in \eqref{eq:2.3b} is sharp.

\subsubsection{Polygonal domains}
\label{sec:2.1.2}
Let $\Omega$ be a curvilinear polygon and let $\gV$ be the (finite) set of its vertices. 
In this case, new model operators appear on infinite sectors $\Pi_\bx$ tangent to $\Omega$ at vertices $\bx\in\gV$. By homogeneity $\En(B_\bx \ee,\Pi_\bx)=|B(\bx)|\En(1 \ee,\Pi_\bx)$ and by rotation invariance, $\En(1 \ee,\Pi_\bx)$ only depends on the opening $\alpha(\bx)$ of the sector $\Pi_\bx$. Let $\cS_\alpha$ be a model sector of opening $\alpha\in(0,2\pi)$. Then
$$
   \sE(B \ee,\Omega) = \min\big(b,b'\Theta_{0},
   \min_{\bx\in\gV} |B(\bx)|\,\En(1 \ee,\cS_{\alpha(\bx)})\big) \ . 
$$
In \cite[\S 11]{Bo05}, it is proved that 
\[
   \exists C>0,  \quad 
   -Ch^{5/4} \leq \lambda_h(B \ee,\Omega)-h\sE(B \ee,\Omega) \leq Ch^{9/8}.
\]
This estimate can be improved under the assumption that 
\begin{equation}
\label{eq:sless}
   \sE(B \ee,\Omega) < \min (b,b'\Theta_{0}) ,
\end{equation}
which means that at least one of the corners makes the energy lower than in the regular case: The asymptotic expansions provided in \cite{BonDau06}  then yield the sharp estimates
\[
   \exists C>0, \quad |\lambda_h(B \ee,\Omega) - h \sE(B \ee,\Omega)|\leq Ch^{3/2} \ . 
\]

From \cite{Ja01,Bo05} follows that for all $\alpha\in(0,\frac\pi2]$ there holds
\begin{equation}
\label{eq:alpha}
   \En(1 \ee,\cS_{\alpha})<\Theta_{0}.
\end{equation}
Therefore condition \eqref{eq:sless} holds for constant magnetic fields as soon as there is an angle opening $\alpha_\bx\le\frac\pi2$. 
Finite element computations by Galerkin projection as presented in \cite{BoDauMaVial07} suggest that \eqref{eq:alpha} still holds for all $\alpha\in(0,\pi)$. Let us finally mention that if $\Omega$ has straight sides and $B$ is constant, the convergence of $\lambda_h(B \ee,\Omega)$ to $h \sE(B \ee,\Omega)$ is exponential: Their difference is bounded by $C\exp(-\beta h^{-1/2})$ for suitable positive constants $C$ and $\beta$ (see \cite{BonDau06}).

\subsection{Neumann conditions in dimension 3}
In dimension $n=3$ we still distinguish the regular and singular domains.

\subsubsection{Regular domains}
\label{SS:regulardomain}
Here $\Omega\subset \R^3$ is assumed to be regular. For a continuous magnetic field $\bB$ it is known (\cite{LuPan00} and \cite{HeMo02}) that \eqref{Metaordre0} holds. In that case 
$$
   \sE(\bB \ee,\Omega)=
   \min\big(\inf_{\bx\in\Omega}|\bB(\bx)|,
   \inf_{\bx\in\partial\Omega}|\bB(\bx)| \ee\sigma(\theta(\bx))\big),
$$ 
where $\theta(\bx)\in[0,\frac\pi2]$ denotes the unoriented angle between the magnetic field and the boundary at the point $\bx\in\partial\Omega$, and the quantity $\sigma(\theta)$ is the bottom of the spectrum of a model problem, see Section \ref{sec:tax}. Let us simply mention that $\sigma$ is increasing on $[0,\frac\pi2]$ and that $\sigma(0)=\Theta_0$, $\sigma(\pi/2)=1$.

\paragraph{Constant magnetic field}
Here the magnetic field $\bB$ is assumed to be constant and unitary. 
There exists a non-empty set $\Sigma$ of $\partial\Omega$ on which $\bB(\bx)$ is tangent to the boundary. In that case we have 
$$
   \sE(\bB \ee,\Omega)=\Theta_{0} \ . 
$$ 
Theorem 1.1 of \cite{HeMo04} states that
$$
  \exists C>0, \quad |\lambda_h(\bB \ee,\Omega)-h\Theta_{0}| \leq Ch^{4/3},  
$$
for $h$ small enough.
Under some extra assumptions on $\Sigma$, Theorem 1.2 of \cite{HeMo04} yields a two-term asymptotics for $\lambda_h(\bB \ee,\Omega)$ showing the optimality of the previous estimate.

\paragraph{Variable magnetic field}
Let $\bB$ be a smooth non-vanishing magnetic field. There holds \cite[Theorem 9.1.1]{FouHe10} 
$$
   \exists C>0, \quad - Ch^{5/4} \leq
   \lambda_h(\bB \ee,\Omega)-h\sE(\bB \ee,\Omega) \leq Ch^{5/4}  \ .
$$
The proof of this result was already sketched in \cite{LuPan00}. 
In \cite[Remark 6.2]{HeMo04}, the upper bound is improved to $O(h^{4/3})$.

Under the following two extra assumptions
\begin{itemize}
\item[\em a)] The inequality
$\inf_{\bx\in\partial\Omega} |\bB(\bx)|\, \sigma(\theta(\bx))
   < \inf_{\bx\in\Omega} |\bB(\bx)|$ holds,
\item[\em b)] The function $\bx\mapsto |\bB(\bx)| \ee\sigma(\theta(\bx))$ reaches its minimum at a point $\bx_{0}$ where $\bB$ is neither normal nor tangent to the boundary,
\end{itemize}
a three-term quasimode is constructed in \cite{Ray09-3d}, providing the sharp upper bound: 
$$ 
   \exists C>0, \quad \lambda_h(\bB \ee,\Omega)-h\sE(\bB \ee,\Omega) \leq C h^{3/2} \,.
$$

\subsubsection{Singular domains}
Until now, two examples of non-smooth domains have been addressed in the literature. In both cases, the magnetic field $\bB$ is assumed to be constant.
 
\paragraph{Rectangular cuboids}
The case where $\Omega$ is a rectangular cuboid (i.e., the product of three bounded intervals) is considered by Pan \cite{Pan02}:
The asymptotic limit \eqref{Metaordre0} holds for such a domain and there exists a vertex $\bv\in\gV$ such that $\sE(\bB \ee,\Omega)=\En(\bB \ee,\Pi_\bv)$.
Moreover, in the case where the magnetic field is tangent to a face but is not tangent to any edge, there holds   
$$
   \En(\bB \ee,\Pi_\bv) < 
   \inf_{\bx\in \overline{\Omega}\setminus \gV }\En(\bB \ee,\Pi_{\bx})
$$
and eigenfunctions associated to $\lambda_h(\bB \ee,\Omega)$ concentrate near corners as $h\to0$.

\paragraph{Lenses} 
The case where $\Omega$ has the shape of a lens is treated in \cite{PoTh} and \cite{PoRay12}. The domain $\Omega$ is supposed to have two faces separated by an edge $\be$ that is a regular loop contained in the plane $x_3=0$. The magnetic field considered is $\bB=(0,0,1)$.

It is proved in \cite{PoTh} that, 
if the opening angle $\alpha$ of the lens is constant and $\leq 0.38\pi$,
$$
   \inf_{\bx\in \be}\En(\bB,\Pi_{\bx}) < 
   \inf_{\bx\in \overline{\Omega}\setminus \be}\En(\bB,\Pi_{\bx})  
$$
and that the asymptotic limit \eqref{Metaordre0} holds with the following estimate: 
$$
   \exists C>0, \quad |\lambda_h(\bB \ee,\Omega)-h\sE(\bB \ee,\Omega)| 
   \leq C h^{5/4} \ .
$$
When the opening angle of the lens is variable and under some non-degeneracy hypotheses, a complete eigenvalue asymptotics is obtained in \cite{PoRay12} resulting into the optimal estimate
$$
   \exists C>0, \quad 
   |\lambda_h(\bB \ee,\Omega)-h\sE(\bB \ee,\Omega)| \leq C h^{3/2} \,.  
$$

\section{Polyhedral domains and their singular chains}
\label{sec:chain}

For the sake of completeness and for ease of further discussion, in the same spirit as in \cite[Section 2]{Dau88}, we introduce here a recursive definition of two intertwining classes of domains
\begin{itemize}
\item[a)] $\gP_n$, a class of infinite open cones in $\R^n$. 
\item[b)] 
$\gD(M)$, a class of bounded connected open subsets of a smooth manifold without boundary---actually, $M=\R^n$ or $M=\dS^n$, with $\dS^n$ the unit sphere of $\R^{n+1}$, 
\end{itemize}

\subsection{Domains and tangent cones}
\label{SS:tangent}
We call a \emph{cone} any open subset $\Pi$ of $\R^n$ satisfying 
\[
   \forall\rho>0 \ \ \mbox{and}\ \ \bx\in\Pi,\quad \rho \bx\in\Pi,
\]
and the \emph{section} of the cone $\Pi$ is its subset $\Pi\cap\dS^{n-1}$.
Note that $\dS^0=\{-1,1\}$.

{\bf Initialization}: $\gP_0$ has one element, $\{0\}$.
$\gD(\dS^0)$ is formed by all subsets of $\dS^0$.

{\bf Recurrence}: For $n\ge1$,
\begin{enumerate}
\item $\Pi\in\gP_n$ if and only if the section of $\Pi$ belongs to $\gD(\dS^{n-1})$,
\item
$\Omega\in\gD(M)$ if and only if for any $\bx\in\overline\Omega$, there exists a cone $\Pi_\bx\in\gP_n$ and a local $\sC^\infty$  diffeomorphism $\diffeo^\bx$ which maps a neighborhood $\cU_\bx$ of $\bx$ in $M$ onto a neighborhood $\cV_\bx$ of $\bfz$ in $\R^n$ and such that
\begin{equation}
\label{eq:diffeo}
   \diffeo^\bx(\cU_\bx\cap\Omega) = \cV_\bx\cap\Pi_\bx \quad\mbox{and}\quad
   \diffeo^\bx(\cU_\bx\cap\partial\Omega) = \cV_\bx\cap\partial\Pi_\bx .
\end{equation}
We assume without restriction that the differential of $\diffeo^\bx$ at the point $\bx$ is the identity matrix $\Id_n$. The cone $\Pi_\bx$ is said to be tangent to $\Omega$ at $\bx$. 
\end{enumerate}

{\bf Examples}:
\begin{itemize}
\item The elements of $\gP_1$ are $\R$, $\R_+$ and $\R_-$.
\item The elements of $\gD(\dS^1)$ are $\dS^1$ and all open intervals $\cI\subset\dS^1$ such that $\overline \cI\neq\dS^1$.
\item The elements of $\gP_2$ are $\R^2$ and all sectors with opening $\alpha\in(0,2\pi)$, including half-spaces ($\alpha=\pi$).
\item The elements of $\gD(\R^2)$ are curvilinear polygons with piecewise non-tangent smooth sides. Note that corner angles do not take the values $0$ or $2\pi$, and that $\gD(\R^2)$ includes smooth domains.
\end{itemize}

\begin{definition}
\label{def:redcone}
Let $\gO_n$ denote the group of orthogonal linear transformations of $\R^n$. 
\begin{itemize}
\item[a)] We say that a cone $\Pi$ is {\em equivalent} to another cone $\Pi'$ and denote
\[
   \Pi \equiv \Pi'
\]
if there exists $\udiffeo\in\gO_n$ such that $\udiffeo\Pi=\Pi'$. 
\smallskip
\item[b)] Let $\Pi\in\gP_n$. If $\Pi$ is equivalent to $\R^{n-d}\times\Gamma$ with $\Gamma\in\gP_d$ and $d$ is minimal for such an equivalence, $\Gamma$ is said to be {\em a minimal reduced cone} associated with $\Pi$. 
\smallskip
\item[c)] Let $\bx\in\overline\Omega$ and let $\Pi_\bx$ be its tangent cone. We denote by $d_0(\bx)$ the dimension of the minimal reduced cone associated with $\Pi_\bx$.
\end{itemize}
\end{definition}

\subsection{Recursive definition of singular chains}
A singular chain
$\dx=(\bx_0,\bx_1,\ldots,\bx_\nu)\in\gC(\Omega)$ (with a natural number $\nu$) is a finite collection of points according to the following recursive definition.

{\bf Initialization}: $\bx_0\in\overline\Omega$, 
\begin{itemize}
\item Let $C_{\bx_0}$ be the tangent cone to $\Omega$ at $\bx_0$ (here $C_{\bx_0}=\Pi_{\bx_0}$),
\item Let $\Gamma_{\bx_0}\in\gP_{d_0}$ be its minimal reduced cone: 
$C_{\bx_0}=\udiffeo^0(\R^{n-d_0}\times\Gamma_{\bx_0})$.
\item Alternative: \begin{itemize}
\item If $\nu=0$, stop here.
\item If $\nu>0$, then $d_0>0$ and let $\Omega_{\bx_0}\in\gD(\dS^{d_{0}-1})$ be the section of $\Gamma_{\bx_0}$
\end{itemize}
\end{itemize}

{\bf Recurrence}: $\bx_j\in\overline\Omega_{\bx_0,\ldots,\bx_{j-1}}\in\gD(\dS^{d_{j-1}-1})$.
If $d_{j-1}=1$, stop here ($\nu=j$). If not:
\begin{itemize}
\item Let $C_{\bx_0,\ldots,\bx_j}$ be the tangent cone to $\Omega_{\bx_0,\ldots,\bx_{j-1}}$ at $\bx_j$, 
\item Let $\Gamma_{\bx_0,\ldots,\bx_j}\in\gP_{d_j}$ be its minimal reduced cone: 
$C_{\bx_0,\ldots,\bx_j}=\udiffeo^j(\R^{d_{j-1}-1 -d_j}\times\Gamma_{\bx_0,\ldots,\bx_j})$.
\item Alternative: \begin{itemize}
\item If $j=\nu$, stop here.
\item If $j<\nu$, then $d_j>0$ and let $\Omega_{\bx_0,\ldots,\bx_j}\in\gD(\dS^{d_{j}-1})$ be the section of $\Gamma_{\bx_0,\ldots,\bx_j}$.
\end{itemize}
\end{itemize}

Note that $n\ge d_0>d_1>\ldots>d_\nu$. Hence $\nu\le n$. Note also that for $\nu=0$, we obtain the trivial one element chain $(\bx_0)$ for any $\bx_0\in\overline\Omega$.

While $\gC(\Omega)$ is the set of all singular chains, for any $\bx\in\overline\Omega$, we denote by $\gC_\bx(\Omega)$ the subset of chains originating at $\bx$, i.e., the set of chains $\dx=(\bx_0,\ldots,\bx_\nu)$ with $\bx_0=\bx$. Note that the one element chain $(\bx)$ belongs to $\gC_\bx(\Omega)$. We also set
\begin{equation}
\label{eq:gCx}
   \gC^*_\bx(\Omega) = \{\dx\in\gC_\bx(\Omega),\  \nu>0\} = \gC_\bx(\Omega)\setminus\{(\bx)\}.
\end{equation}

We set finally, with the notation $\langle \by\rangle$ for the vector space generated by $\by$,
\begin{equation}
\label{eq:PiX}
   \Pi_\dx = 
   \begin{cases}
     C_{\bx_0} =\Pi_{\bx_0} & \ \mbox{if} \ \  \nu=0,\\[0.5ex]
     \udiffeo^0 \big(\R^{n-d_0}\times\langle \bx_1\rangle \times 
     C_{\bx_0,\bx_1} \big)  &\ \mbox{if} \ \  \nu=1, \\[0.5ex]
     \udiffeo^0 \Big(\R^{n-d_0}\times\langle \bx_1\rangle \times  \ldots \times
     \udiffeo^{\nu-1} \big(\R^{d_{\nu-2}-1-d_{\nu-1}}\times\langle \bx_\nu\rangle \times 
     C_{\bx_0,\ldots,\bx_{\nu}}\big)
     \ldots \Big)  &\ \mbox{if} \ \  \nu\ge2.
   \end{cases}
\end{equation}
Note that if $d_\nu=0$, the cone $C_{\bx_0,\ldots,\bx_{\nu}}$ coincides with $\R^{d_{\nu-1}-1}$, leading to $\Pi_\dx=\R^n$.  

\begin{definition}
\label{def:chaineq}
Let $\dx=(\bx_0,\ldots,\bx_\nu)$ and $\dx'=(\bx'_0,\ldots,\bx'_{\nu'})$ be two chains in $\gC(\Omega)$. We say that $\dx$ is equivalent to $\dx'$ 
if $\bx_0=\bx'_0$ and $\Pi_\dx=\Pi_{\dx'}$.
\end{definition}

{\bf Special subsets} of $\overline\Omega$: For $d\in\{0,\ldots,n\}$, let
\begin{equation}
\label{eq:Ad}
   \gA_d(\Omega) = \{\bx\in\overline\Omega,\quad d_0(\bx)=d\}.
\end{equation}
The strata of $\overline\Omega$ are the connected components of $\gA_d(\Omega)$, for $d\in\{0,\ldots,n\}$. They are denoted by $\bt$ and their set by $\gT$.

Examples:
\begin{itemize}
\item $\gA_0(\Omega)$ coincides with $\Omega$.
\item $\gA_1(\Omega)$ is the subset of $\partial\Omega$ of the regular points of the boundary.
\item If $n=2$, $\gA_2(\Omega)$ is the set of corners.
\item If $n=3$, $\gA_2(\Omega)$ is the set of edge points.
\item If $n=3$, $\gA_3(\Omega)$ is the set of corners.
\end{itemize}

\subsection{Polyhedral domains} \label{sec:3.3} Polyhedral domains and polyhedral cones form subclasses of $\gD(M)$ and $\gP_n$, denoted by $\ogD(M)$ and $\ogP_n$, respectively:

\begin{itemize}
\item[a)] The cone $\Pi\in \gP_n$ is a polyhedral cone if its boundary is contained in a finite union of (hyper)surfaces. We write $\Pi\in\ogP_n$.
\item[b)]  The domain $\Omega\in\gD(M)$ is a polyhedral domain if all its tangent cones $\Pi_\bx$ are polyhedral. We write $\Omega\in\ogD(M)$.
\end{itemize}

This allows to make precise the definition of faces, edges and corners in dimension $3$, in connection with singular chains.
\begin{enumerate}
\item Interior point $\bx\in\Omega$. Only one chain in $\gC_\bx(\Omega)$: $\dx=(\bx)$.
\smallskip

\item The faces $\bff$ are the connected components of $\gA_1(\Omega)$. The set of faces is denoted by $\gF$. Let $\bx$ belong to a face. There are two chains in $\gC_\bx(\Omega)$:
\begin{enumerate}
\item $\dx = (\bx)$ with $\Pi_\dx=\Pi_\bx$, the tangent half-space. $\Pi_\dx\equiv\R^2\times\R_+$.
\item $\dx = (\bx,\bx_1)$ where $\bx_1=1$ is the only element in $\R_+\cap\dS^0$. 
Thus $\Pi_\dx=\R^3$. 
\end{enumerate}
\smallskip

\item The edges $\be$ are the connected components of $\gA_2(\Omega)$. The set of edges is denoted by $\gE$. Let $\bx$ belong to an edge. There are three types of chains in $\gC_\bx(\Omega)$:
\begin{enumerate}
\item $\dx = (\bx)$ with $\Pi_\dx=\Pi_\bx$, the tangent wedge (which is not a half-plane). The reduced cone of $\Pi_\bx$ is a sector $\Gamma_\bx$ the section of which is an interval $\cI_\bx\subset\dS^1$.
\item $\dx = (\bx,\bx_1)$ where $\bx_1\in \overline \cI_\bx$. 
\begin{enumerate}
\item If $\bx_1$ is interior to $\cI_\bx$, $\Pi_\dx=\R^3$. No further chain.
\item If $\bx_1$ is a boundary point of $\cI_\bx$, $\Pi_\dx$ is a half-space, containing one of the two faces $\partial^\pm\Pi_\bx$ of the wedge $\Pi_\bx$.
\end{enumerate}
\item $\dx=(\bx,\bx_1,\bx_2)$ where $\bx_1\in \partial \cI_\bx$, $\bx_2=1$ and $\Pi_\dx=\R^3$.
\end{enumerate}
To sum up, there are 4 equivalence classes in $\gC_\bx(\Omega)$ in the case of an edge point $\bx$:
\begin{itemize}
\item $\dx=(\bx)$
\item $\dx=(\bx,\bx_1^\pm)$ with $\{\bx_1^-,\bx_1^+\}=\partial \cI_\bx$
\item $\dx=(\bx,\bx_1^\circ)$ with $\bx_1^\circ$ any chosen point in $\cI_\bx$.
\end{itemize}
\smallskip

\item The corners $\bv$ are the connected components of $\gA_3(\Omega)$. The set of corners is denoted by $\gV$. There are four types of chains in $\gC_\bx(\Omega)$:
\begin{enumerate}
\item $\dx = (\bx)$ with $\Pi_\dx=\Pi_\bx$, the tangent cone (that is not a wedge). It coincides with its reduced cone. Its section $\Omega_\bx$ is a polygonal domain in $\dS^2$.
\item $\dx = (\bx,\bx_1)$ where $\bx_1\in \overline \Omega_\bx$. 
\begin{enumerate}
\item If $\bx_1$ is interior to $\Omega_\bx$, $\Pi_\dx=\R^3$. No further chain.
\item If $\bx_1$ is in a side of $\Omega_\bx$, $\Pi_\dx$ is a half-space, containing one of the faces of the cone $\Pi_\bx$.
\item If $\bx_1$ is a corner of $\Omega_\bx$, $\Pi_\dx$ is a wedge. Its edge contains one of the edges of $\Pi_\bx$.
\end{enumerate}
\item $\dx=(\bx,\bx_1,\bx_2)$ where $\bx_1\in\partial\Omega_\bx$ 
\begin{enumerate}
\item If $\bx_1$ is in a side of $\Omega_\bx$, $\bx_2=1$, $\Pi_\dx=\R^3$. No further chain.
\item If $\bx_1$ is a corner of $\Omega_\bx$, $C_{\bx,\bx_1}$ is plane sector, and $\bx_2\in\overline\cI_{\bx,\bx_1}$ where the interval $\cI_{\bx,\bx_1}$ is its section. If $\bx_2$ is an interior point, then $\Pi_\dx=\R^3$. 
\end{enumerate}
\item $\dx=(\bx,\bx_1,\bx_2,\bx_3)$ where $\bx_1$ is a corner of $\Omega_\bx$, $\bx_2\in\partial\cI_{\bx,\bx_1}$ and $\bx_3=1$. Then $\Pi_\dx=\R^3$.
\end{enumerate}
Let $\bx_1^j$, $1\le j\le N$, be the corners of $\Omega_\bx$, and $\bff^j_1$, $1\le j\le N$, be its sides (notice that there are as many corners as sides). There are $2N+2$ equivalence classes in $\gC_\bx(\Omega)$:
\begin{itemize}
\item $\dx=(\bx)$ (vertex)
\item $\dx=(\bx,\bx_1^j)$ with $1\le j\le N$ (edge points limit)
\item $\dx=(\bx,\bx_1^{\circ,j})$ with $\bx_1^{\circ,j}$ any chosen point inside $\bff^j_1$ (face points limit)
\item $\dx=(\bx,\bx_1^\circ)$ with $\bx_1^\circ$ any chosen point in $\Omega_\bx$ (interior points limit).
\end{itemize}
\end{enumerate}

\begin{remark}
\label{rem:length2}
For polyhedral domains $\Omega$, it is a consequence of the description above that chains $(\bx_0,\bx_1)$ of length 2 are enough to describe all equivalence classes of the set of chains $\gC^*_{\bx_0}(\Omega)$ \eqref{eq:gCx}. 
This does not hold anymore if general corner domains are considered. Besides, the notion of equivalence classes as introduced in Definition \ref{def:chaineq} is sufficient for the analysis of operators $\OP_{h}(\bA,\Omega)$ in the case of magnetic fields $\bB$ smooth in Cartesian variables. Should $\bB$ be smooth in polar variables only, the whole hierarchy of singular chains would be needed.
\end{remark}

\section{Taxonomy of model problems}
\label{sec:tax}

\subsection{Tangent and model operators}
We recall that $\bA$ is a magnetic potential associated with the magnetic field $\bB$ on the polyhedral domain $\Omega\in\ogD(\R^3)$. For each singular chain $\dx=(\bx_0,\bx_1,\ldots,\bx_\nu)\in\gC(\Omega)$ we set
\begin{equation}
\label{eq:tangent}
   \bB_\dx = \bB(\bx_0) \quad\mbox{and}\quad
   \bA_\dx(\bx) = \nabla\bA(\bx_0)\cdot\bx,\ \ \bx\in\Pi_\dx,
\end{equation}
so that $\bB_\dx$ is the magnetic field frozen at $\bx_0$ and $\bA_\dx$ the linear part\footnote{In \eqref{eq:tangent}, $\nabla\bA$ is the $3\times3$ matrix with entries $\partial_kA_j$, $1\le j,k\le 3$, and $\cdot\,\bx$ denotes the multiplication by the column vector $\bx=(x_1,x_2,x_3)^\top$.} of the potential at $\bx_0$. 
We have obviously
\[
   \curl\bA_\dx = \bB_\dx\,,
\]
so that the tangent magnetic operator $\OP(\bA_\dx \ee,\Pi_\dx)$ and its ground energy $\En(\bB_\dx \ee,\Pi_\dx)$ make sense. 
By homogeneity,  there holds for tangent problems (cf.\ Lemma~\ref{lem.dilatation})
\begin{lemma}\label{lem.scal}
Let $\Pi\in\gP_3$ be a cone, $\bB$ be a constant magnetic field with norm $b>0$. There holds 
\begin{equation}
\label{eq:norm}
   \En(\bB,\Pi) = b\, 
  \En\Big(\frac{\bB}{b},\Pi\Big)\,.
\end{equation}
Moreover, let $\bA$ be a linear potential associated with $\bB$.
Then $\bx\mapsto \Psi(\sqrt{b}\,\bx)$ is an eigenfunction of $H(\bA,\Pi)$ associated with $E(\bB,\Pi)$ if and only if $\bx\mapsto \Psi(\bx)$ is an eigenfunction of $H(\bA/b,\Pi)$ associated with $E(\bB/b,\Pi)$. 
\end{lemma}

That is why we can reduce to consider {\em model problems} on cones $\Pi\in\ogP_3$ with unitary constant magnetic fields.

\subsection{Singular chains and generalized eigenvectors for model problems}
Let $\Pi\in\ogP_3$ be a polyhedral cone and $\bB$ be a unitary constant magnetic field associated with a linear 
potential $\bA$. Let $\Gamma\in\ogP_d$ be a minimal reduced cone associated with $\Pi$. We recall that this means that $\Pi\equiv\R^{3-d}\times\Gamma$ and that the dimension $d$ is minimal for such an equivalence. 

Let $\gC_\bfz(\Pi)$ denote the singular chains of $\Pi$ originating at its vertex $\bfz$ and let $\gC^*_\bfz(\Pi)$ be the subset of chains of length $\ge2$ (see \eqref{eq:gCx}). Note that $\gC^*_\bfz(\Pi)$ is empty if and only if $\Pi=\R^3$, i.e., if $d=0$. We introduce the energy along higher chains:

\begin{definition}[Energy along higher chains]\label{def.seE}
We define the quantity
\begin{equation}
\label{eq:s*}
 \seE(\bB \ee,\Pi):=
 \begin{cases}
  \inf_{\dx\in\gC^{*}_0(\Pi)}\En(\bB \ee,\Pi_{\dx})  & \mbox{if $d>0$,} \\
  +\infty  & \mbox{if $d=0$,} \\
\end{cases}
\end{equation}
which is the infimum of the ground energy of the magnetic Laplacian over all the singular chains of length $\geq2$.
\end{definition} 

If $d>0$, let $\Omega_0\in\ogD(\dS^{d-1})$ be the section of $\Gamma$. Since $\Pi$ is a polyhedral cone, we have (cf.\ Remark \ref{rem:length2})
\begin{equation}
\label{eq:s*b}
   \seE(\bB \ee,\Pi)=\inf_{\bx_1\in\overline\Omega_0} \En(\bB \ee,\Pi_{(\bfz,\bx_1)}) \ ,
\end{equation}
i.e., among all chains $\dx\in\gC^*_\bfz(\Pi)$, we can restrict to those of length $2$, $\dx=(\bfz,\bx_1)$.

Since the cone $\Pi$ is unbounded, it is relevant to define $\lambda_{\rm ess}(\bB,\Pi)$ as the bottom of the essential spectrum of the operator $H(\bA,\Pi)$. When $d\leq 2$, due to translation invariance we have $\En(\bB,\Pi)=\lambda_{\rm ess}(\bB,\Pi)$.
When $d=3$, the operator $H(\bA,\Pi)$ may have discrete spectrum. 

With the aim of constructing quasimodes for our original problem on $\Omega$, we need generalized eigenvectors for its tangent problems. For this we make use of the localized domain $\dom_{\,\loc} \left(\OP(\bA,\Pi)\right)$ of the model magnetic Laplacian $\OP(\bA,\Pi)$ as introduced in \eqref{D:domloc}.

\begin{definition}[Generalized eigenvector]
\label{def:geneig}
Let $\Pi\in\ogP_3$ be a polyhedral cone and $\bA$ a linear magnetic potential. We call {\em generalized eigenvector} for $\OP(\bA\ee,\Pi)$ a nonzero function $\Psi\in\dom_{\,\loc}(\OP(\bA\ee,\Pi))$ associated with a real number $\Lambda$, so that
\begin{equation}
\label{eq:geneig}
\begin{cases}
(-i\nabla+\bA)^2\Psi=\Lambda\Psi &\mbox{in } \Pi,\\
(-i\partial_n+\mathbf{n}\cdot\bA)\Psi=0 &\mbox{on } \partial\Pi.
\end{cases}
\end{equation}
\end{definition}

\begin{definition}[Admissible generalized eigenvector]
\label{D:Age}
Under the hypothesis of Definition \ref{def:geneig}, a generalized eigenvector $\Psi$ for $\OP(\bA\ee,\Pi)$ is said to be {\em admissible} if there exist a rotation $\udiffeo:\Pi\mapsto\R^{3-k}\times\Upsilon$ with $k \geq d$ and $\Upsilon\in\ogP_k$, and a system of coordinates $(\by,\bz)\in\R^{3-k}\times\R^k$ such that
\begin{equation}
\label{eq:age1}
   \Psi\circ \udiffeo^{-1} (\by,\bz) = \re^{\ri\ee\vartheta(\by,\bz)}\,\Phi(\bz)\, ,
\end{equation}
with some real polynomial function $\vartheta$ of degree $\le2$ and some exponentially decreasing function $\Phi$, namely there exist positive constants $c_\Psi$ and $C_\Psi$ such that
\begin{equation}
\label{eq:agmongeneig}
\|\re^{c_\Psi|\bz|}\Phi\|_{L^2(\Upsilon)}\leq C_\Psi \|\Phi \|_{L^2(\Upsilon)}  \, .
\end{equation}
We will denote by $\bx^\natural=(\by,\bz)\in\R^{3-k}\times\Upsilon$ the natural coordinates and by $\Psi^{\natural}=\Psi\circ \udiffeo^{-1}$ the natural form of $\Psi$. 
\end{definition}

\begin{remark}
\label{rem:geneig}
By Lemma \ref{L:chgvar} it is straightforward that in coordinates $\bx^\natural$ the magnetic potential and the magnetic field are transformed into
\[
   \bA^\natural = \rJ^\top \big(\bA\circ\udiffeo^{-1}\big)
   \quad\mbox{and}\quad
   \bB^\natural = \rJ^\top \big(\bB\circ\udiffeo^{-1}\big)
\]
where the orthogonal $3\times3$ matrix $\rJ$ is such that $\udiffeo^{-1}(\bx^\natural)=\rJ\bx^\natural$. Note also that $\udiffeo(\bx)=\rJ^\top\bx$. Therefore $\Psi$ is an admissible generalized eigenfunction of $H(\bA,\Pi)$ associated with the value $\Lambda$ if and only if $\Psi\circ \udiffeo^{-1}$ is a generalized eigenfunction of $H(\bA^{\natural},\R^{3-k}\times \Upsilon)$ associated with the same value $\Lambda$.
\end{remark}

\begin{lemma}
\label{lem:geneig}
If $\Psi$ is an admissible generalized eigenvector for $\OP(\bA\ee,\Pi)$ associated with $\Lambda$, for any other linear magnetic potential $\bA'$ such that $\curl\bA'=\curl\bA$, the operator $\OP(\bA',\Pi)$ possesses an admissible generalized eigenvector $\Psi'$ associated with the same value $\Lambda$.
\end{lemma}

\begin{proof}
If $\curl\bA=\curl\bA'$ and $\bA$, $\bA'$ are both linear, there exists a polynomial function $\phi$ of degree 2 such that $\bA=\bA'-\nabla\phi$. Using a change of gauge (Lemma~\ref{lem:gauge}), we find that $\Psi'$ defined as 
$$
   \Psi'(\bx)=\re^{\ri\phi(\bx)} \re^{\ri\vartheta(\by,\bz)}\,\Phi(\bz), \quad\bx\in\Pi,
$$ 
is an admissible generalized eigenvector for $\OP(\bA',\Pi)$.
\end{proof}

\subsection{Dichotomy for model problems}
The main result which we prove in this section is a dichotomy statement, as follows.

\begin{theorem}
\label{th:dicho}
Let $\Pi\in\ogP_3$ be a polyhedral cone and $\bB$ be a constant nonzero magnetic field.
Let $\bA$ be any associated linear magnetic potential.
Recall that $\En(\bB,\Pi)$ is the ground energy of $\OP(\bA,\Pi)$ and $\seE(\bB,\Pi)$ is the energy along higher chains introduced in \eqref{eq:s*b}. Then,
\begin{equation}
\label{eq:comp}
   \En(\bB,\Pi) \leq \seE(\bB,\Pi) \ 
\end{equation}
and we have the dichotomy:
\begin{enumerate}
\item[(i)] If $\En(\bB,\Pi)<\seE(\bB,\Pi)$, then $\OP(\bA,\Pi)$ admits an admissible generalized eigenvector associated with the value $\En(\bB,\Pi)$.\smallskip
\item[(ii)] If $\En(\bB,\Pi)=\seE(\bB,\Pi)$, then there exists a singular chain $\dx\in \gC^{*}_{\bfz}(\Pi)$ such that 
$$
   \En(\bB,\Pi_\dx) = \En(\bB,\Pi) \quad\mbox{and}\quad
   \En(\bB,\Pi_\dx)<\seE(\bB,\Pi_\dx).
$$
\end{enumerate}
\end{theorem}

\begin{remark}\label{rem:chaine}
In the case (ii), we note that by statement (i) applied to the cone $\Pi_\dx$, $\OP(\bA,\Pi_\dx)$ admits an admissible generalized eigenvector associated with the value $\En(\bB,\Pi)$.
\end{remark}

\paragraph{Outline of the proof of Theorem \ref{th:dicho}}
Owing to Lemma \ref{lem.scal} we may assume that $\bB:=\curl \bA$ is unitary. The proof relies on an exhaustion of cases, organized according increasing values of $d$, the dimension of the reduced cone of $\Pi$. First note that quantities $\En(\bB,\Pi)$ and $\seE(\bB,\Pi)$, like the definition of admissible generalized eigenvectors, are independent of a choice of Cartesian coordinates. Thus, for each value of $d$, ranging from $0$ to $3$, once $\Pi$ and a constant unitary magnetic field $\bB$ are chosen, we exhibit a system of Cartesian coordinates $\bx=(x_1,x_2,x_3)$ that allows the simplest possible description of the configuration $(\bB,\Pi)$. In these coordinates, the magnetic field can be viewed as a reference one and we denote it by $\uB$. We also choose a corresponding reference linear potential $\uA$, since we have gauge independence by virtue of Lemma \ref{lem:geneig}. Then relying on various results corresponding to each case, we prove that we are either in situation (i) and exhibit coordinates\footnote{As our system of Cartesian coordinates $\bx=(x_1,x_2,x_3)$ will be already chosen in some optimal way, coordinates $(\by,\bz)$ will simply be a splitting of $(x_1,x_2,x_3)$. But it is noticeable that the dimension $k$ of $\bz$ variables may be strictly larger than $d$.} $(\by,\bz)$ and admissible generalized eigenvectors, or in situation (ii). 

In each of Sections \ref{subs:R3}--\ref{subs:coin} we consider one value of $d$, from $0$ to $3$, which will achieve the proof of Theorem \ref{th:dicho}. In Section \ref{s:age}, we collect all possible structures of admissible generalized eigenvectors $\Psi$, organized according to increasing values of $k$, the number of directions in which $\Psi$ has exponential decay.

\subsection{Full space} $d=0$. \ 
\label{subs:R3}
$\Pi$ is the full space. We take coordinates $\bx=(x_1,x_2,x_3)$ so that
\[
   \Pi=\R^3 \quad\mbox{and}\quad \uB = (1,0,0),
\]
and choose as reference potential
\[
   \uA = (0,-\tfrac{x_{3}}{2},\tfrac{x_{2}}{2}) .
\]
Hence
\[
   \OP(\uA \ee,\Pi) = \OP(\uA \ee,\R^3) = 
   \rD_1^2 + (\rD_2-\tfrac{x_{3}}{2})^2  + (\rD_3 + \tfrac{x_{2}}{2})^2
   \quad\mbox{with}\quad \rD_{j}=-i\partial_{x_{j}}.
\]
It is classical (see \cite{LauLifIII}) that the spectrum of $ \OP(\uA \ee,\R^3)$ is $[1,+\infty)$. Therefore 
\begin{equation}
\label{E:spectrespace}
\En(\uB \ee,\R^3) = 1 \ .
\end{equation}
An admissible generalized eigenfunction associated to the ground energy is 
$$\Psi(\bx)= \re^{-(x_{2}^2+x_{3}^2)/4} \, , $$
which has the form \eqref{eq:age1} with $\by=x_{1}$, $\bz=(x_{2},x_{3})$, and $\vartheta\equiv0$ .

\subsection{Half space} $d=1$. \ \label{subs:R3+}
\label{SS:HS}
$\Pi$ is a half-space. We take coordinates $\bx=(x_1,x_2,x_3)$ so that
\[
   \Pi = \R^3_+ := \{(x_1,x_2,x_3)\in\R^3,\ x_{3}>0\} 
   \quad\mbox{and}\quad
   \uB=(0,b_{1},b_{2}) \ \mbox{with}\ b_1^2+b_2^2=1 \, ,
\]
and choose as reference potential
\[
   \uA = (b_{1}x_{3}-b_{2}x_{2},0,0) \, .
\]
Hence
\[
   \OP(\uA \ee,\Pi) = \OP(\uA \ee,\R^3_+) = 
   (\rD_1+b_{1}x_{3}-b_{2}x_{2})^2 + \rD_2^2  + \rD_3^2 .
\]
We note that 
\begin{equation}
\label{eq:En1}
  \seE(\uB \ee,\R^3_+) = \En(\uB,\R^3)=1.
\end{equation}
There exists $\theta\in[0,2\pi)$ such that $b_1=\cos\theta$ and $b_2=\sin\theta$, so that $\theta$ is the angle between the magnetic field and the boundary of $\R^3_{+}$. Due to symmetries we can reduce to $\theta\in[0,\frac{\pi}{2}]$.
Denote by $\cF_{1}$ the Fourier transform in $x_{1}$-variable, $\tau$ the Fourier variable associated with $x_{1}$, and
\[
   \widehat \OP_{\tau}(\uA \ee,\R^3_+) 
   :=  (\tau+b_{1}x_{3}-b_{2}x_{2})^2 + \rD_2^2  + \rD_3^2 , 
\]
acting on $L^2(\R\times \R_{+})$ with natural boundary condition. There holds 
$$
   \cF_{1} \ H(\uA,\R^3_{+})\ \cF_{1}^*=
   \int^{\bigoplus}_{\tau\in\R} \widehat \OP_{\tau}(\uA \ee,\R^3_+)\,\rd \tau . 
$$
We discriminate three cases:
\paragraph{Tangent field}
$\theta=0$, then $\widehat \OP_{\tau}(\uA\ee,\R^3_+) :=  \rD_{2}^2 + \rD_3^2 + (\tau + x_3)^2$, let $\xi$ be the partial Fourier variable associated with $x_{2}$ and define the new operators 
\[
   \widehat \OP_{\xi,\tau}(\uA\ee,\R^3_+) :=  \xi^2 + D_{3}^2 + (\tau + x_{3})^2,\quad
   \DG\tau = \rD_{3}^2 + (\tau + x_{3})^2\,,
\]
where $\DG\tau$ (sometimes called the de Gennes operator) acts on $L^2(\R_{+})$ with Neumann boundary condition. There holds
\[
   \inf\gS(\DG\tau) = \mu(\tau),\quad
   \inf\gS(\widehat \OP_{\tau,\xi}(\uA\ee,\R^3_{+})) = \mu(\tau)+\xi^2,
\]
in which the behavior of the first eigenvalue $\mu(\tau)$ is well-known (see \cite{DauHe93}): The function $\mu$ admits a unique minimum denoted by $\Theta_0\simeq0.59$ for the value $\tau=-\sqrt{\Theta_0}$.
Hence
\[
   \En(\uB \ee,\R^3_+) = \Theta_0 < \seE(\uB \ee,\R^3_+).
\]
We are in case (i) of Theorem \ref{th:dicho}. If $\Phi$ denotes an eigenvector of $\DG\tau$ associated with $\Theta_{0}$ (function of $x_{3}\in\R_+$), a corresponding admissible generalized eigenvector is
\begin{equation}
\label{eq:d1.1}
   \Psi(\bx) = \re^{-\ri\ee\sqrt{\Theta_0}\, x_{1}} \,\Phi(x_{3}).
\end{equation}
which has the form \eqref{eq:age1} with  $\by=(x_{1},x_{2})$, $\bz=x_{3}$, and $\vartheta(\by,\bz)=-x_1\sqrt{\Theta_{0}}$.

\paragraph{Normal field}
$\theta=\frac\pi2$, then $\widehat \OP_{\tau}(\uA,\R^3_+) :=  \rD_2^2 + \rD_3^2 + (\tau - x_2)^2$. There holds
for all $\tau\in \R$, $\inf\gS(\widehat \OP_{\tau}(\uA,\R^3_+))=1$ (see \cite[Theorem 3.1]{LuPan00}), hence
\[
   \En(\uB \ee,\R^3_+) = 1 = \seE(\uB \ee,\R^3_+).
\]
We are in case (ii) of Theorem \ref{th:dicho}. 
 The chain $\dx$ is given by $(\bfz,1)$ and $\Pi_{\dx}=\R^3$, for which $\En(\uB \ee,\R^3)<\seE(\uB \ee,\R^3)$.

\paragraph{Neither tangent nor normal}
$\theta\in(0,\frac\pi2)$. Then for any $\tau\in\R$, $\widehat \OP_{\tau}(\uA,\R^3_+) $ is isospectral to $\widehat \OP_{0}(\uA,\R^3_+) $ the ground energy of which is an eigenvalue $\sigma(\theta)<1$
(see \cite{HeMo02}). We deduce
\[
   \En(\uB,\R^3_+) = \sigma(\theta) 
   \quad \mbox{with}\quad\sigma(\theta)<1.
\]
We are in case (i) of Theorem \ref{th:dicho}. We recall:
\begin{lemma}
\label{P:continuitesigma}
The function $\theta\mapsto \sigma(\theta)$ is continuous and increasing on $(0,\frac{\pi}{2})$ (\cite{HeMo02,LuPan00}). Set $\sigma(0)=\Theta_0$ and $\sigma(\frac{\pi}{2})=1$. Then the function $\theta\mapsto \sigma(\theta)$ is of class $\sC^1$ on $[0,\frac{\pi}{2}]$ (\cite{BoDauPopRay12}).
\end{lemma}

The first eigenvalue of $\widehat \OP_{0}(\uA,\R^3_+) $ is associated with an exponentially decreasing eigenvector $\Phi$ which is a function of $(x_{2},x_{3})\in\R\times\R_+$. An admissible generalized eigenvector for $\OP(\uA,\R^3_{+})$ is given by
\begin{equation}
\label{eq:d1.3}
   \Psi(\bx) = \Phi(x_{2},x_{3}),
\end{equation}
which has the form \eqref{eq:age1} with $\by=x_{1}$, $\bz=(x_{2},x_{3})$, and $\vartheta\equiv0$.

Thus Theorem \ref{th:dicho} is proved for half-spaces.

\subsection{Wedges} $d=2$. \ \label{subs:diedre}
$\Pi$ is a wedge and let $\alpha\in(0,\pi)\cup(\pi,2\pi)$ denote its opening. 
Let us introduce the model sector $\cS_{\alpha}$ and the model wedge $\cW_{\alpha}$
\begin{equation}
\label{eq:Wa}
   \cS_{\alpha} = 
   \begin{cases}
   \{x=(x_2,x_3),\ x_2\tan\tfrac\alpha2>|x_3|\big\}   & \mbox{if $\alpha\in(0,\pi)$} \\
   \{x=(x_2,x_3),\ x_2\tan\tfrac\alpha2>-|x_3|\big\}  & \mbox{if $\alpha\in(\pi,2\pi)$}
\end{cases} 
   \quad\mbox{and}\quad
   \cW_\alpha = \R\times\cS_\alpha \,.
\end{equation}
We take coordinates $\bx=(x_1,x_2,x_3)$ so that
\[
   \Pi=\cW_\alpha \quad\mbox{and}\quad
   \uB=(b_{1},b_{2},b_{3}) \ \mbox{with}\ b_{1}^2+b_2^2+b_3^2=1 \, ,
\]
and choose as reference potential
\[
   \uA = (b_{2}x_{3}-b_{3}x_{2},0,b_{1}x_{2})\,.
\]
Hence
\[
   \OP(\uA \ee,\Pi) =
   \OP(\uA \ee,\cW_\alpha) = (\rD_1+b_{2}x_{3}-b_{3}x_{2})^2 + \rD_2 ^2 + (\rD_3 + b_1x_2 )^2.
\]
Denote by $\tau$ the Fourier variable associated with $x_1$, and
\begin{equation}
\label{D:Hhatsector}
   \widehat \OP_{\tau}(\uA \ee,\cW_\alpha) 
   :=   (\tau+b_{2}x_{3}-b_{3}x_{2})^2 + \rD_2 ^2 + (\rD_3 + b_1x_2 )^2
\end{equation}
acting on $L^2(\cS_{\alpha})$ with natural Neumann boundary condition. We introduce the notation:
$$
   s(\uB,\cS_{\alpha};\tau) := \inf\gS( \widehat \OP_{\tau}(\uA \ee,\cW_\alpha) ),
$$
so that we have the direct Fourier integral decomposition
$$
   \cF_{1}\ H(\uA,\cW_{\alpha})\ \cF_{1}^*=
   \int^{\bigoplus}_{\tau\in\R} \widehat \OP_{\tau}(\uA \ee,\cW_\alpha) \,\rd \tau 
$$
and the relation
\begin{equation}
\label{E:relslambda}
 \En(\uB \ee,\cW_{\alpha}) =\inf_{\tau\in\R} s(\uB,\cS_{\alpha};\tau) \,.
\end{equation}
The singular chains of $\gC^*_\bfz(\cW_\alpha)$ have three equivalence classes, cf.\ Definition \ref{def:chaineq} and Section \ref{sec:3.3}~(3), corresponding to three distinct model operators, associated to half-spaces $\Pi^\pm_\alpha$ corresponding to the faces $\partial^\pm\cW_\alpha$ of $\cW_\alpha$, and to the full space $\R^3$. Thus
\[
   \seE(\uB,\cW_\alpha) = \min\{ \En(\uB,\Pi^+_\alpha), \,\En(\uB,\Pi^-_\alpha), \,
   \En(\uB,\R^3) \}.
\]
Let $\theta^{\pm}\in[0,\frac{\pi}{2}]$ be the angle between $\uB$ and the face $\partial\Pi^\pm_{\alpha}$. We have, cf.\ Lemma \ref{P:continuitesigma},
\begin{equation}
\label{eq:s*Da}
   \seE(\uB,\cW_\alpha) = \min\{ \sigma(\theta^+), \sigma(\theta^-), \, 1\} =
   \sigma(\min\{\theta^+, \,\theta^-\}).
\end{equation}
When $\Pi=\cW_\alpha$, Theorem \ref{th:dicho} relies on the following result \cite[Theorem 3.5]{Pop13}:

\begin{lemma}
\label{lem:pop}
We have $\En(\uB,\cW_{\alpha})\leq\seE(\uB,\cW_{\alpha})$.\\ 
Moreover, if $\En(\uB,\cW_{\alpha})<\seE(\uB,\cW_{\alpha})$, then the function $\tau\mapsto s(\uB,\cS_{\alpha};\tau)$ reaches its infimum. Let $\tau^{*}$ be a minimizer. Then $\En(\uB,\cW_{\alpha})$ is a discrete eigenvalue for the operator $\widehat \OP_{\tau^{*}}(\uA \ee,\cW_\alpha) $ and the associated eigenfunctions have exponential decay.
\end{lemma}

From the previous lemma we deduce 
\begin{itemize}
\item[(i)] 
If $\En(\uB,\cW_{\alpha})<\seE(\uB,\cW_{\alpha})$, there exists $\tau^{*}$ such that the operator $\widehat \OP_{\tau^{*}}(\uA,\cW_{\alpha})$ admits an exponential decaying eigenfunction $\Phi$ of $(x_{2},x_{3})\in\cS_\alpha$ associated with $\En(\uB,\cW_{\alpha})$. The function 
\begin{equation*}
   \Psi(\bx) = \re^{\ri\tau^*x_{1}}\Phi(x_{2},x_{3})
\end{equation*}
is an admissible generalized eigenvector for the operator $\OP(\uA,\cW_{\alpha})$ associated with $\En(\uB,\cW_{\alpha})$. It has the form \eqref{eq:age1} with $\by=x_{1}$, $\bz=(x_{2},x_{3})$, and $\vartheta(\by,\bz)=\tau^*\by$. 
\smallskip
\item[(ii)] 
If $\En(\uB,\cW_{\alpha})=\seE(\uB,\cW_{\alpha})$, let $\circ\in \{-,+\}$ satisfy $\theta^{\circ}=\min(\theta^{-},\theta^{+})$ and $\Pi_{\alpha}^\circ$ be the corresponding face.  We have $\theta^{\circ}\in[0,\frac{\pi}{2})$ and $\seE(\uB,\cW_{\alpha})=\sigma(\theta^{\circ})$. 
We deduce from Section \ref{SS:HS} that $\seE(\uB,\Pi_{\alpha}^\circ)=1$, hence
$\En(\uB,\cW_{\alpha})=\En(\uB,\Pi_{\alpha}^\circ)<\seE(\uB,\Pi_{\alpha}^\circ)$.
\end{itemize}

Thus Theorem \ref{th:dicho} is proved for wedges. We extend the definition of $\cW_{\alpha}$ to $\alpha=\pi$ by setting $\cW_{\pi}:=\R_{+}^3$. Let us quote now the continuity result of \cite[Theorem 4.5]{Pop13}:
\begin{lemma}
\label{lem:contwedge}
The function $(\bB,\alpha)\mapsto \En(\bB,\cW_{\alpha})$ is continuous on $\dS^2\times (0,2\pi)$. 
\end{lemma}

We end this paragraph by a few examples.

\begin{example}
Let $\bB\in \dS^2$ be a constant magnetic field. Let $\alpha$ be chosen in $(0,\pi)\cup(\pi,2\pi)$.

a) For $\alpha$ small enough $\En(\bB,\cW_{\alpha})<\seE(\bB,\cW_{\alpha})$ (see \cite{Pop13} when the magnetic field is not tangent to the plane of symmetry of the wedge and \cite[Ch. 7]{PoTh} otherwise). 

b) Let $\bB=(0,0,1)$ be tangent to the edge. Then $\seE(\bB,\cW_{\alpha})=\Theta_{0}$ and $\En(\bB,\cW_{\alpha})=\En(1 \ee,\cS_{\alpha})$, cf.\ Section \ref{sec:2.1.2}. According to whether the ground energy $\En(1 \ee,\cS_{\alpha})$ of the plane sector $\cS_{\alpha}$ is less than $\Theta_0$ or equal to $\Theta_0$, we are in case (i) or (ii) of the dichotomy.

c) Let $\bB$ be tangent to a face of the wedge and normal to the edge. Then $\seE(\bB,\cW_{\alpha})=\Theta_{0}$. It is proved in \cite{Pof13T} that there holds $\En(\bB,\cW_{\alpha})<\Theta_{0}$ for $\alpha$ small enough, whereas $\En(\bB,\cW_{\alpha})=\Theta_{0}$ for $\alpha\in[\frac{\pi}{2},\pi)$. 
\end{example}

As a direct consequence of the whole description performed in the previous Sections \ref{subs:R3}-\ref{subs:diedre}, we obtain the following continuity statements.

\begin{theorem}\label{th:contEBd<3}
Let $\Pi$ be a cone in $\ogP_{3}$ with $d<3$ (i.e.\ $\Pi$ is the full space, a half-space or a wedge). 
Then the function $\bB\mapsto \En(\bB,\Pi)$ is continuous on $\dS^2$.
\end{theorem}

\begin{corollary}\label{cor:contE*B}
Let $\Pi$ be a cone in $\ogP_{3}$. 
The function $\bB\mapsto \seE(\bB,\Pi)$ is continuous on $\dS^2$.
\end{corollary}


\subsection{Polyhedral cones} $d=3$ \ \label{subs:coin}
The main result of this paragraph is the characterization of the bottom $\lambda_\ess(\bB,\Pi)$ of the essential spectrum of $\OP(\bA,\Pi)$.
\begin{proposition}
\label{prop:cone-ess}
Let $\Pi\in\ogP_3$ be a polyhedral cone with $d=3$, which means that $\Pi$ is not a wedge, nor a half-space, nor the full space. Let $\bB$ be a constant magnetic field. With the quantity $\seE(\bB,\Pi)$ introduced in \eqref{eq:s*}, there holds
\[
   \lambda_\ess(\bB,\Pi) = \seE(\bB,\Pi)\,.
\]
\end{proposition}

Before writing proof details, let us specify what is $\seE(\bB,\Pi)$ in the case of a polyhedral cone. Let $\Omega_0$ be the section of $\Pi$, i.e., $\Omega_0=\Pi\cap\dS^2$. We recall from \eqref{eq:s*b} that
\begin{equation}
\label{E:EstarinfE}
   \seE(\bB \ee,\Pi)=\inf_{\bx_1\in\overline\Omega_0} \En(\bB \ee,\Pi_{(\bfz,\bx_1)}) \,.
\end{equation}
In fact the set of equivalence classes (Definition \ref{def:chaineq}) of the chains $\dx=(\bfz,\bx_1)$ is finite. Let us describe this set,  cf.\ Section \ref{sec:3.3}~(4). Let $\gF$ and $\gE$ be the set of faces $\bff$ and edges $\be$ of $\Pi$. For $\bff\in\gF$, let $\Pi_\bff$ be the half-space whose boundary contains $\bff$ and containing points of $\Pi$ near any point of $\bff$. For $\be\in\gE$, there are two faces $\bff^\pm_\be$ adjacent to $\be$. Let $\Pi_\be$ be the wedge whose boundary contains $\be\cup\bff^+_\be\cup\bff^-_\be$ and containing points of $\Pi$ near any point of $\bff^+_\be\cup\bff^-_\be$.

Let $\bx_1\in\overline\Omega_0$. There are three possibilities:
\begin{enumerate}
\item $\bx_1$ is interior to $\Omega_0$. Then $\Pi_{(\bfz,\bx_1)}=\R^3$.
\item $\bx_1$ belongs to a side of $\Omega_0$. This side is contained in a face $\bff$ of $\Pi$. Then $\Pi_{(\bfz,\bx_1)}=\Pi_\bff$.
\item $\bx_1$ belongs to a vertex of $\Omega_0$. This vertex is contained in an edge $\be$ of $\Pi$. Then $\Pi_{(\bfz,\bx_1)}=\Pi_\be$.
\end{enumerate}
We have that
\begin{equation}
\label{eq:s*3}
   \seE(\bB \ee,\Pi)= \min\big\{ \min_{\be\in\gE} \En(\bB,\Pi_\be),\ 
   \min_{\bff\in\gF} \En(\bB,\Pi_\bff),\ 1 \big\}\,.
\end{equation}
 Since $\eqref{eq:comp}$ is proved for $d=2$, we have $E(\bB,\Pi_{\be}) \leq \min\{E(\bB,\Pi_{\bff_{\be}^{+}}),E(\bB,\Pi_{\bff_{\be}^{-}})\}$. Therefore equation \eqref{eq:s*3} becomes 
\begin{equation}
\label{eq:s*3simple}
   \seE(\bB \ee,\Pi)= \min_{\be\in\gE} \{\En(\bB,\Pi_\be)\}\, .
\end{equation}

We recall the Persson Lemma that gives a characterization of the bottom of the essential spectrum (see \cite{Pers60}):
\begin{lemma}
\label{L:Persson}
We have 
$$
   \lambda_\ess(\bB,\Pi)=\lim_{R\to +\infty} \Sigma \left(\bB,\Pi,R\right)
$$
with
$$\Sigma \left(\bB,\Pi,R\right):=\inf_{\substack{\psi\in \sC_{0}^{\infty}(\overline{\Pi}\cap \complement\mathcal{B}_{R})
\\
 \psi\neq0}}\frac{q[\bA,\Pi](\psi)}{\|\psi\|^2}$$
where $\cB_{R}$ is the ball of radius $R$ centered at the origin and $\complement\cB_{R}$ its complementary in $\R^3$.
\end{lemma}

\begin{proof}
{\em (of Proposition~\ref{prop:cone-ess})}. \
Let $\bA$ be a linear potential associated with $\bB$. \\
{\em Upper bound:} 
We denote by $\bx_{\be^*}$ a vertex of $\overline\Omega_{0}$ and $\be^*$ the associated edge such that  $\seE(\bB \ee,\Pi)=\En(\bB \ee,\Pi_{\be^*})$, cf.\ \eqref{eq:s*3simple}. Let $\varepsilon>0$, there exists $\psi_{\varepsilon}\in\sC_{0}^{\infty}(\overline\Pi_{\be^*})$ a normalized function such that 
$$
   q[\bA,\Pi_{\be^*}] (\psi_{\varepsilon}) \leq \En(\bB \ee,\Pi_{\be^*})+\varepsilon \, . 
$$
For $r>0$ we define 
$$
   \psi_{\varepsilon}^{r}(\bx):=
   \re^{\ri \langle \bx\ee,\ee \bA(r\bx_{\be^*})\rangle} \, \psi_{\varepsilon}(\bx-r\bx_{\be^*}) \, , 
$$
so that we have, due to gauge invariance and translation effect, cf.\ Lemma~\ref{lem.transl},
\[
   \supp(\psi_{\varepsilon}^{r})=\supp(\psi_{\varepsilon})+r\bx_{\be^*}
   \quad\mbox{and}\quad
   q[\bA,\Pi](\psi_{\varepsilon}^{r})=q[\bA,\Pi_{\be^*}](\psi_{\varepsilon})\,.
\]
Let $R>0$, for $r$ large enough we have $\supp(\psi_{\varepsilon}^{r})\subset \complement\mathcal{B}_{R}$ and $\psi_{\varepsilon}^{r}\in \dom(q[\bA,\Pi])$. We get 
$$
   q[\bA,\Pi](\psi_{\varepsilon}^{r})=
   q[\bA,\Pi_{\be^*}](\psi_{\varepsilon})\leq E(\bB,\Pi_{\be^*})+ \varepsilon \, . 
$$
We deduce 
$$
   \forall \varepsilon>0, \forall R>0, \quad \Sigma(\bB,\Pi,R) \leq  
   E(\bB,\Pi_{\be^*})+\varepsilon 
$$
and Lemma \ref{L:Persson} provides the upper bound of Proposition \ref{prop:cone-ess}: $\lambda_\ess(\bB,\Pi) \leq \seE(\bB,\Pi)$.
  
\smallskip
{\em Lower bound:} 
Let
\[
 \cU_\bfz\cup \big(\bigcup_{\bff\in\gF}\ \cU_\bff\ \big)
   \cup \big(\bigcup_{\be\in\gE}\ \cU_\be\ \big)
\]
a covering of $\overline{\Omega}_{0}$ according to its stratification, which means that
\begin{equation*}
   \cU_{\bfz}\subset \Omega_0,\quad
   \cU_\bff\cap\overline{\Omega}_{0} = \Pi_\bff \cap\overline{\Omega}_{0} \ (\forall\bff\in\gF),
   \quad\mbox{and}\quad
   \cU_\be\cap\overline{\Omega}_{0} = \Pi_\be \cap\overline{\Omega}_{0}\ (\forall\be\in\gE).
\end{equation*} 
Let $\hat\chi_{j}$, $j\in\gJ :=\{\bfz\}\cup\gF\cup\gE$, be an associated partition of unity of the section $\overline{\Omega}_{0}$ such that
\[
   \sum_{j\in\gJ}(\hat\chi_{j})^2=1 \quad\mbox{and}\quad
   \supp(\hat\chi_{j})\subset \cU_j,\ \forall j\in\gJ.
\]
Let $\chi\in\sC^{\infty}(\R_+)$ such that $\chi\equiv0$ on $[0,\frac12]$ and $\chi\equiv1$ on $[1,+\infty)$.
We now define a partition of the unity of $\complement\mathcal{\cB}_{R}\cap \overline{\Pi}$ by setting
\[
   \chi_{j}^{R}(\bx) = \chi\Big(\frac{|\bx|}{R}\Big)\, \hat\chi_{j}\Big(\frac{\bx}{|\bx|}\Big),
   \quad j\in\gJ.
\]
We have $\sum_{j}(\chi_{j}^{R})^2=1$ on $\complement\mathcal{\cB}_{R}\cap \overline{\Pi}$ and  
$$
   \forall R>0, \quad \sum_{j\in\gJ}|\nabla{\chi}_{j}^{R}|^2 \leq CR^{-2}.
$$
Moreover we have $\supp({\chi}_{j}^{R})\cap\overline\Pi\subset \overline\Pi_{j}$, $j\in\gJ =\{\bfz\}\cup\gF\cup\gE$, where we have set $\Pi_\bfz=\R^3$. The IMS formula for quadratic forms (see Lemma \ref{lem:IMS}) provides 
\begin{align*}
   \forall \psi\in \sC_{0}^{\infty}(\overline{\Pi}\cap\complement\cB_{R}), \quad  
   q[\bA,\Pi](\psi)&\geq\sum_{j\in\gJ}q[\bA,\Pi](\chi_{j}^{R}\psi)-CR^{-2}\|\psi\|^2  
\\
&=\sum_{j\in\gJ}q[\bA,\Pi_{j}](\chi_{j}^{R}\psi)-CR^{-2}\|\psi\|^2  
\\
&\geq\sum_{j\in\gJ}E(\bB,\Pi_{j})\|\chi_{j}^{R}\psi\|^2-CR^{-2}\|\psi\|^2  
\\
&\geq(\seE(\bB,\Pi)-CR^{-2})\|\psi\|^2.
\end{align*}
Thus we deduce the lower bound of Proposition \ref{prop:cone-ess} by using Lemma \ref{L:Persson}.
\end{proof}

Then it is clear that Theorem \ref{th:dicho} in the case of polyhedral cones is a consequence of Proposition \ref{prop:cone-ess}:

\begin{itemize}
\item if $\En(\bB,\Pi)<\seE(\bB,\Pi)=\lambda_\ess(\bB,\Pi)$, then there exists an eigenvector for $\OP(\bA,\Pi)$, which by standard arguments based on Agmon estimates is exponentially decreasing, see \cite{Ag85}.
\item if $\En(\bB,\Pi)=\seE(\bB,\Pi)=\lambda_\ess(\bB,\Pi)= \min_{\be\in\gE} \{\En(\bB,\Pi_\be)\}$, we are reduced to the previous cases in lower values of $d$.
\end{itemize}

\begin{example}[Octant]
Let $\Pi=(\R_{+})^3$ be the model octant and $\bB$ be a constant magnetic field with $|\bB|=1$. We quote from \cite[\S 8]{Pan02}:
\begin{enumerate}
\item[(i)] If the magnetic field $\bB$ is tangent to a face but not to an edge of $\Pi$, there exists an edge $\be$ such that $\seE(\bB,\Pi)=\En(\bB,\Pi_{\be})$ and there holds $\En(\bB,\Pi)<\En(\bB,\Pi_{\be})$. 

\item[(ii)] If the magnetic field $\bB$ is tangent to an edge $\be$ of $\Pi$, $\seE(\bB,\Pi)=\En(\bB,\Pi_{\be})=\En(\bB,\Pi)$. 
Moreover by \cite[\S 4]{Pan02}, $\En(\bB,\Pi_{\be})=\En(1,\cS_{\pi/2})<\Theta_{0}=\seE(\bB,\Pi_{\be})$.
\end{enumerate}
\end{example}


\subsection{Structure of the admissible generalized eigenvectors}
\label{s:age}
We have described in the previous Sections \ref{subs:R3}-\ref{subs:coin} admissible generalized eigenvectors in every situation. In this section we list the model configurations $(\bB,\Pi)$ owning admissible generalized eigenvectors and give a comprehensive overview of their structure in a table. We also prove some stability properties of the generalized eigenvectors and the associated energy under perturbation of the magnetic field $\bB$.

Let $\bB$ be a constant magnetic field and $\Pi$ a cone in $\ogP_{3}$. Let us assume that $\En(\bB,\Pi)<\seE(\bB,\Pi)$. Therefore by Theorem \ref{th:dicho} there exist  admissible generalized eigenvectors $\Psi$ that have the form \eqref{eq:age1}. We recall the discriminant parameter $k\in \{1,2,3 \}$ which is the number of directions in which the generalized eigenvector has an exponential decay. For further use we call (G1), (G2), and (G3) the situation where $k=1$, $2$, and $3$, respectively. 

In Table \ref{T:age} we gather all possible situations for $(k,d)$ where $d$ is the dimension  of the reduced cone of $\Pi$. We assume that the magnetic field $\bB$ is unitary, similar formulas can be found using Lemma \ref{lem.scal} for any non-zero constant magnetic field. We provide the explicit form of an admissible generalized eigenfunction $\Psi$ of $H(\uA,\Pi)$ in variables $(\by,\bz)\in \R^{3-k}\times \Upsilon$ where $\uA$ is a model linear potential associated with $\bB$ in these variables. We also give the cone $\Upsilon$ on which the generalized eigenfunction has exponential decay (note that $\Upsilon$ does not always coincide with the reduced cone $\Gamma$ of $\Pi$).

\begin{table}[ht!]
\mbox{}\hskip-1em
{\small
\renewcommand{\arraystretch}{1.5}
\begin{tabular}{| p{9mm}|p{3cm}|p{3.6cm}|p{1.3cm}|p{2.2cm}|p{36mm}|}
  \hline
$(k,d)$  & Model $(\bB,\Pi)$ & Potential $\uA$ &  $\Upsilon$ & Explicit $\Psi$  & $\Phi$ eigenvector of  \\
  \hline
  \hline
(3,3)  & \hrulefill & \hrulefill  &  $\Pi=\Gamma$ & $\Phi(\bz)$ & $\OP(\bA,\Pi)$ 
 \\
\hline
(2,2) &  \centering $(b_{0},b_{1},b_{2})$ \ \ \  \ \ \  $\Pi=\R\times \cS_{\alpha}$ & $(b_1z_2-b_2z_1,0,b_0z_1)$ &  $\cS_{\alpha}=\Gamma$ & $\re^{\ri\tau^{*}y}\Phi(\bz)$ &
  $ \widehat \OP_{\tau}(\uA \ee,\cW_\alpha) $, cf \eqref{D:Hhatsector}
\\
\hline
(2,1) &
 \centering $(0,b_{1},b_{2})$, $b_{2}\neq 0$ \ \ 
 $\Pi=\R^2\times \R_{+}$ & $(b_{1}z_{2}-b_{2}z_{1},0,0)$ &  $\R\times \R_{+}$  & $\Phi(\bz)$ & $-\Delta_{\bz}+(b_1z_2-b_2z_1)^2$
\\
\hline
(2,0) & \centering $(1,0,0)$ \quad \ \ \ \ \ \ \ $\Pi=\R^3$ & $(0,-\frac12 z_2, \frac12 z_1)$ & $\R^2$ & 
$\re^{-|\bz|^2/4}$ & $-\Delta_\bz+i\bz\times\!\nabla_\bz+\tfrac{|\bz|^2}{4}$
\\
\hline
(1,1) & \centering $(0,1,0)$ \quad \ \ \ \ \ \ \ $\Pi=\R^2\times \R_{+}$ & $(z,0,0)$ &  $\R_{+}=\Gamma$ & $\re^{-iy_{1}\sqrt{\Theta_{0}}} \Phi(z)$ & $-\partial_{z}^2+(z-\sqrt{\Theta_{0}})^2$
\\
\hline
\end{tabular}}
\mbox{}\\[1ex]
\caption{\label{T:age} Generalized eigenfunctions of $H(\bA,\Pi)$ depending on the geometry $(\bB,\Pi)$ with $E(\bB,\Pi) < \seE(\bB,\Pi)$ written in variables $(\by,\bz)\in \R^{3-k}\times \Upsilon$.}
\end{table}

\begin{remark}
In the case where $\Pi$ is a half-space and $\bB$ is normal to $\partial\Pi$, we have $E(\bB,\Pi)=\seE(\bB,\Pi)=1$ and we are in case (ii) of Theorem \ref{th:dicho}, therefore there exists an admissible generalized eigenvector for the strict singular chain $\R^3$ of $\Pi$. However there also exists an admissible generalized eigenfunction for the operator on the half-plane $\Pi$. Let $\bz=(z_{1},z_{2})$ be coordinates of $\partial\Pi$ and $y$ coordinate normal to $\partial\Pi$. Let $\uA(y,z_{1},z_{2}):=(0,-\frac{z_{2}}{2},\frac{z_{1}}{2})$. As described in \cite[Lemma 4.3]{LuPan00}, the function $\Psi:(y,z_{1},z_{2})\mapsto e^{-(z_{1}^2+z_{2}^2)/4}$ is an admissible generalized eigenvector for $H(\uA,\Pi)$ associated with 1, indeed it satisfies the Neumann boundary condition at the boundary $y=0$ since it is constant in the $y$ direction and it is a solution of the eigenvalue equation $H(\uA,\Pi)\Psi=\Psi$, see Section \ref{subs:R3}. 
\end{remark}

A perturbation of the magnetic field has distinct effects according to the situation : 
(G1) is not stable whereas (G2) is.
We prove this in the following lemma together with local uniform estimates for exponential decay. 
\begin{lemma} \label{L:Elip}
Let $\bB_{0}$ be a non zero constant magnetic field and $\Pi$ be a cone in $\ogP_{3}$ with $d<3$. Assume that
$\En(\bB_{0},\Pi)<\seE(\bB_{0},\Pi)$. 
\begin{itemize}
\item[(a)] In a ball $\cB(\bB_{0},\varepsilon)$, the function  $\bB\mapsto \En(\bB,\Pi)$ is Lipschitz-continuous and
$$
\En(\bB,\Pi)<\seE(\bB,\Pi).
$$
\item[(b)]
We suppose moreover that $(\bB_{0},\Pi)$ is in situation (G2). 
For $\bB\in\cB(\bB_{0},\varepsilon)$, we denote by $\Psi^{\bB}$ an admissible generalized eigenfunction given by Theorem \ref{th:dicho}. For $\varepsilon$ small enough, $(\bB,\Pi)$ is still in situation (G2) and $\Psi^{\bB}$ has the form
$$\Psi^{\bB}(\bx)=\re^{\ri\vartheta^{\bB}(\by,\bz)}\Phi^{\bB}(\bz)\quad\mbox{for }\quad 
\udiffeo^{\bB}\bx=(\by,\bz)\in\R\times\Upsilon,$$
with $\udiffeo^{\bB}$ a suitable rotation, and there exist constants $c>0$ and $C>0$ such that
\begin{equation}
\label{E:agmonuniform}
\forall \bB\in \cB(\bB_{0},\varepsilon), \quad \|\Phi^{\bB}\re^{c|\bz|}\|_{L^2(\Upsilon)}\leq C \|\Phi^{\bB} \|_{L^2(\Upsilon)}  \, .
\end{equation}
\end{itemize}
\end{lemma}

\begin{proof}
Let us distinguish the three possible situations according to the value of $d$:
\begin{itemize}
\item[$d=0\,$:] When $\Pi=\R^3$, we have $\En(\bB,\Pi)=|\bB|$ and $\seE(\bB,\Pi)=+\infty$. The admissible generalized eigenvector $\Psi^{\bB}$ is explicit as explained above.
Thus (a) and (b) are established in this case. 

\item[$d=1\,$:] When $\Pi$ is a half-space, we denote by $\theta(\bB)$ the unoriented angle in $[0,\frac\pi2]$ between $\bB$ and the boundary. The function $\bB\mapsto\theta(\bB)$ is Lipschitz. 
Moreover the function $\sigma$ is $\sC^1$ on $[0,\pi/2]$ (see Lemma \ref{P:continuitesigma}).
We deduce that the function $\bB\mapsto\sigma(\theta(\bB))$ is Lipschitz outside any neighborhood of $\bB= 0$. Thus point (a) is proved. 
Assuming furthermore that $(\Pi,\bB_{0})$ is in situation (G2), we have $\theta(\bB_{0})\in (0,\frac{\pi}{2})$ and there exist $\varepsilon>0$, $\theta_{\min}$ and $\theta_{\max}$ such that 
$$
   \forall \bB \in \cB(\bB_{0},\varepsilon), \quad 
   \theta(\bB)\in [\theta_{\min},\theta_{\max}]
   \subset (0,\tfrac{\pi}{2}) \, .
$$
The admissible generalized eigenvector is constructed above.  
The uniform exponential estimate is proved in \cite[\S2]{BoDauPopRay12}.

\item[$d=2\,$:]  When $\Pi$ is a wedge, point (a) is proved in \cite[\S4]{Pop13}. 
Due to the continuity of $\bB\mapsto \En(\bB,\Pi)$ there exist $c>0$ and $\varepsilon>0$ such that
$$\forall\bB\in \cB(\bB_{0},\varepsilon),\quad \seE(\bB,\Pi)-\En(\bB,\Pi)>c.$$ 
Point (b) is then a direct consequence of \cite[Proposition 4.2]{Pop13}.
\end{itemize}
The proof of Lemma~\ref{L:Elip} is complete.
\end{proof}

\section{Continuity properties of the ground energy}
\label{sec:cont}
Let $\Omega\in \ogD(\R^3)$ and let $\bB\in\sC^0(\overline\Omega)$ be a continuous magnetic field. In this section we investigate the continuity properties on $\overline{\Omega}$ of the application $\Lambda:\bx\mapsto \En(\bB_{\bx},\Pi_{\bx})$. Let $\bt$ be a stratum of $\overline{\Omega}$ (see \eqref{eq:stratif}). We have denoted by $\Lambda_{\bt}$ the restriction of $\Lambda$ to $\bt$ (see \eqref{eq:Lams}). Combining \eqref{eq:norm}, \eqref{E:spectrespace}, Lemma \ref{P:continuitesigma} and Lemma \ref{lem:contwedge} we get that $\Lambda_{\bt}$ is continuous. 

Let us assume that $\bt$ is not reduced to a point. We now describe how we extend $\Lambda_{\bt}$ to the boundary of $\bt$. Let $\bx \in \partial\bt$ and $\Pi_{\bx}\in \ogP_{3}$ be its tangent cone. Let $\cU_{\bx}$, $\cV_{\bx}$ and $\diffeo^{\bx}$ be the open sets and the diffeomorphism introduced in Section \ref{SS:tangent}. Let $\tilde\bt$ be the stratum of $\Pi_{\bx}$ such that 
$$\diffeo^{\bx}(\bt \cap \cU_{\bx})= \tilde\bt\cap \cV_{\bx} \ . $$
To $\tilde\bt$ is associated the singular chain $\dx\in \gC^*_\bfz(\Pi_{\bx})$ such that $\Pi_{\dx}$ is the tangent cone to $\Pi_\bx$ at any point of $\tilde\bt$.

We extend $\Lambda_{\bt}$ in $\bx$ by setting 
\begin{equation}
\label{E:Defextension}
\Lambda_{\bt}(\bx)=\En(\bB_{\bx},\Pi_{\dx}) \, .
\end{equation}

\begin{lemma}
\label{lem:cont}
Let $\Omega\in \ogD(\R^3)$ and let $\bB\in\sC^0(\overline\Omega)$. Let $\bt$ a stratum of $\Omega$ which is not a vertex. Then
formula \eqref{E:Defextension} defines a continuous extension of the function $\Lambda_{\bt}$ to $\overline{\bt}$.
\end{lemma}

\begin{proof}
For $\bx\in \partial\bt$ we show that the extension defined by \eqref{E:Defextension} is continuous in $\bx$. Let $\by\in \cU_{\bx}\cap \bt$ and $\Pi_{\by}$ be the tangent cone to $\Omega$ at $\by$. In the following we will prove that 
\begin{equation}
\label{E:continuiteextension}
\lim_{\by \to \bx}\En(\bB_{\by},\Pi_{\by})=\En(\bB_{\bx},\Pi_{\dx}) \, . 
\end{equation}
 For a tangent cone $\Pi$ we denote by $d(\Pi)$ the dimension of its reduced cone. 
 Since $\bt$ is not reduced to a point, we have $d(\Pi_{\dx})=d(\Pi_{\by})\leq 2$ and we distinguish several cases:
\begin{itemize}
\item $d=0$. $\Pi_{\dx}=\Pi_{\by}=\R^3$. It follows from Section \ref{subs:R3} that $\En(\bB_{\by},\Pi_{\by})=|\bB_{\by}|$ and $\En(\bB_{\bx},\Pi_{\dx})=|\bB_{\bx}|$. Therefore \eqref{E:continuiteextension}.
\item $d=1$. $\Pi_{\dx}$ and $\Pi_{\by}$ are half-spaces. We denote by $\theta_{\bx}$ (respectively $\theta_{\by}$) the angle between $\bB_{\bx}$ and $\Pi_{\dx}$ (respectively $\bB_{\by}$ and $\Pi_{\by}$). We have $\En(\bB_{\by},\Pi_{\by})=|\bB_{\by}|\sigma(\theta_{\by})$ and $\En(\bB_{\bx},\Pi_{\dx})=|\bB_{\bx}|\sigma(\theta_{\bx})$ (see Section \ref{SS:HS}). Since $\theta_{\by}$ goes to $\theta_{\bx}$ when $\by$ goes to $\bx$, \eqref{E:continuiteextension} follows from the continuity of the function $\sigma$, see Lemma \ref{P:continuitesigma}.
\item $d=2$. $\Pi_{\dx}$ and $\Pi_{\by}$ are wedges. We denote by $\alpha_{\bx}$ and $\alpha_{\by}$ their openings. We denote by $\udiffeo^{\bx}\in\gO_3$ (respectively $\udiffeo^{\by}\in\gO_3$) the linear orthogonal transformation which maps $\Pi_{\dx}$ on $\cW_{\alpha_{\bx}}$ (respectively $\Pi_{\by}$ on $\cW_{\alpha_{\by}}$). We have 
\begin{equation}
\label{E:utiliseunetransfo}
\En(\bB_{\bx},\Pi_{\dx})=\En(\uB_{\bx},\cW_{\alpha_{\bx}})\quad \mbox{and}\quad\En(\bB_{\by},\Pi_{\by})=\En(\uB_{\by},\cW_{\alpha_{\by}}) \, .
\end{equation}
where we have denoted $\uB_{\bx}=\udiffeo^{\bx}(\bB_{\bx})$ and $\uB_{\by}=\udiffeo^{\by}(\bB_{\by})$.
 We have 
\begin{equation*}
\lim_{\by\to\bx}\|\udiffeo^{\bx}-\udiffeo^{\by} \|=0 \quad \mbox{and}\quad \lim_{\by\to\bx}|\alpha_{\bx}-\alpha_{\by}|=0 \, ,
\end{equation*}
therefore we deduce \eqref{E:continuiteextension} from \eqref{E:utiliseunetransfo} and Lemma \ref{lem:contwedge}.
\end{itemize}
Hence we have proved \eqref{E:continuiteextension} in all cases.
\end{proof}

Let $\bx\in \partial\bt$, we deduce from \eqref{eq:comp} that 
$$
   \Lambda_{\bt}(\bx)=\En(\bB_{\bx},\Pi_{\dx})\geq \En(\bB_{\bx},\Pi_{\bx})=\Lambda(\bx) \ . 
$$ 
Combining this with Lemma \ref{lem:cont}, we obtain the following:

\begin{theorem}
\label{T:sci}
Let $\Omega\in \ogD(\R^3)$ and let $\bB\in\sC^0(\overline\Omega)$ be a continuous magnetic field. Then 
the function $\Lambda:\bx\mapsto \En(\bB_{\bx},\Pi_{\bx})$ is lower semi-continuous on $\overline{\Omega}$.
\end{theorem}

\section{Upper bound for first eigenvalues in polyhedral domains}
\label{sec:up}
In this section we prove general upper bounds for the first eigenvalue $\lambda_{h}(\bB,\Omega)$ by $h\sE(\bB,\Omega)$ up to a remainder of size $h^{\kappa}$ with $\kappa >1$. The first theorem provides an upper bound with $\kappa=\frac{5}{4}$ using a novel construction of quasimodes depending on the geometry. Under more regularity assumptions on the potential and using more knowledge on the model problems we refine the quasimodes and we reach an upper bound with $\kappa=\frac{4}{3}$. 

Here is our first result:
\begin{theorem}
\label{T:generalUB}
Let $\Omega\in \ogD(\R^3)$ be a polyhedral domain, $\bA\in W^{2,\infty}(\overline{\Omega})$ be a twice differentiable magnetic potential such that the associated magnetic field $\bB$ does not vanish on $\overline\Omega$. Then there exist $C(\Omega)>0$ and $h_{0}>0$ such that
\begin{equation}
\label{eq:above1}
   \forall h\in (0,h_{0}), \quad \lambda_{h}(\bB,\Omega) \leq 
   h\sE(\bB,\Omega)+C(\Omega)(1+\|\bA\|_{W^{2,\infty}(\Omega)}^2)\,h^{5/4} \ . 
\end{equation}
We recall that the quantity $\sE(\bB,\Omega)$ is the lowest local energy defined in \eqref{eq:s}.
\end{theorem}

It is possible to obtain an upper bound in \eqref{eq:above1} depending on the magnetic field $\bB$ and not on the magnetic potential. For this, we consider $\bB$ as a datum and associate a potential $\bA$ with it. Operators $\sA:\bB\mapsto\bA$ lifting the curl (i.e.\ such that $\curl\circ\,\sA=\Id$) and satisfying suitable estimates do exist in the literature. We quote \cite{CostabelMcIntosh10} in which it is proved that such lifting can be constructed as a pseudo-differential operator of order $-1$. As a consequence $\sA$ is continuous between H\"older classes of non integer order:
\[
   \forall\alpha\in(0,1),\quad
   \exists C_\alpha>0,\quad
   \|\sA\bB\|_{W^{2+\alpha,\infty}(\Omega)} \le C_\alpha  
   \|\bB\|_{W^{1+\alpha,\infty}(\Omega)} \,.
\]
Choosing $\bA=\sA\bB$ in Theorem \ref{T:generalUB}, we deduce the following.

\begin{corollary}
\label{co:T:generalUBB}
Let $\Omega\in \ogD(\R^3)$ be a polyhedral domain, $\bB\in W^{1+\alpha,\infty}(\overline{\Omega})$ be a non-vanishing  H\"older continuous magnetic field of order $1+\alpha$ with some $\alpha\in(0,1)$. Then there exist $C(\Omega)>0$ and $h_{0}>0$ such that
\begin{equation}
\label{eq:above1B}
   \forall h\in (0,h_{0}), \quad \lambda_{h}(\bB,\Omega) \leq 
   h\sE(\bB,\Omega)+C(\Omega)(1+\|\bB\|_{W^{1+\alpha,\infty}(\Omega)}^2)\,h^{5/4} \ . 
\end{equation}
\end{corollary}

Theorem \ref{T:generalUB} is proved in Sections \ref{SS:CV}--\ref{SS:estimQM} according to the following plan. 
Since the energy $\bx \mapsto \En(\bB_{\bx},\Pi_{\bx})$ is lower semi-continuous (see Theorem \ref{T:sci}), it reaches its infimum over the compact $\overline{\Omega}$. We denote by $\bx_{0}\in \overline{\Omega}$ a point such that 
\begin{equation}
\label{D:x0}
   \En(\bB_{\bx_{0}}\ee,\Pi_{\bx_{0}})= \sE(\bB \ee,\Omega)
\end{equation}
where $\Pi_{\bx_{0}}$ is the tangent cone at $\bx_{0}$. Starting from this, the proof is organized in three main steps, developed in Sections \ref{SS:CV}, \ref{SS:QM}, and \ref{SS:estimQM} respectively:
\begin{enumerate}
\item By the local diffeomorphism $\diffeo^{\bx_{0}}$ in the neighborhood $\cU_{\bx_0}$ of $\bx_{0}$ introduced in \eqref{eq:diffeo}, we reduce to a local magnetic operator set on the tangent cone $\Pi_{\bx_{0}}$.
\item We construct quasimodes for the tangent magnetic operator at the vertex $\bfz$ of $\Pi_{\bx_{0}}$. Here we use Theorem \ref{th:dicho} which provides us with 
a suitable admissible generalized eigenfunction associated with the energy $\En(\bB_{\bx_{0}}\ee,\Pi_{\bx_{0}})$.
This generalized eigenfunction will be scaled, truncated and translated in order to give a quasimode for the local magnetic operator.
\item The estimation of various terms in the associated Rayleigh quotient and the min-max principle will finally prove Theorem \ref{T:generalUB}.
\end{enumerate}

If steps (1) and (3) are very classical, step (2) reveals much more originality, because our constructions are valid in any configuration: We do not need to know {\it a priori} whether $(\bB_{\bx_{0}}\ee,\Pi_{\bx_{0}})$ is in situation (i) or (ii) of the dichotomy theorem. If we are in situation (i) -- the most classical one -- our quasimodes will be classical too, and qualified as ``sitting''. If we are in situation (ii), we define our quasimodes on a higher structure $\Pi_\dx$ and make then ``slide'' towards the vertex $\bfz$ of $\Pi_{\bx_{0}}$.

Our second result is also general, the unique additional assumption is a supplementary regularity on the magnetic potential (or equivalently on the magnetic field). 

\begin{theorem}
\label{T:sUB}
Let $\Omega\in \ogD(\R^3)$ be a polyhedral domain,  $\bA\in W^{3,\infty}(\overline{\Omega})$ be a magnetic potential such that the associated magnetic field does not vanish.
Then there exist $C(\Omega)>0$ and $h_{0}>0$ such that
\begin{equation}
\forall h\in (0,h_{0}), \quad \lambda_{h}(\bB,\Omega) \leq h\sE(\bB,\Omega)+C(\Omega)(1+\|\bA\|_{W^{3,\infty}(\Omega)}^2)\,h^{4/3} \ . 
\end{equation}
\end{theorem}

Like for Corollary~\ref{co:T:generalUBB}, we can deduce an upper bound where the constant depends on the magnetic field. 
\begin{corollary}
\label{co:T:generalUBB2}
Let $\Omega\in \ogD(\R^3)$ be a polyhedral domain, $\bB\in W^{2+\alpha,\infty}(\overline{\Omega})$ be a non-vanishing  H\"older continuous magnetic field of order $2+\alpha$ with some $\alpha\in(0,1)$. Then there exist $C(\Omega)>0$ and $h_{0}>0$ such that
\begin{equation}
\label{eq:above1B2}
   \forall h\in (0,h_{0}), \quad \lambda_{h}(\bB,\Omega) \leq 
   h\sE(\bB,\Omega)+C(\Omega)(1+\|\bB\|_{W^{2+\alpha,\infty}(\Omega)}^2)\,h^{4/3} \ . 
\end{equation}
\end{corollary}

 Note that the $h^{4/3}$ bound was known for smooth three-dimensional domains, \cite[Proposition 6.1 \& Remark 6.2]{HeMo04} and that our result extends this result to polyhedral domains without loss.

Theorem \ref{T:sUB} is proved in Section \ref{S:improvement4/3} and we give now some hint on the proof. Like for Theorem \ref{T:generalUB} we start from admissible generalized eigenfunctions and construct sitting or sliding quasimodes adapted to the geometry. However, unlike for the proof of Theorem \ref{T:generalUB}, we are going to actually take advantage of the decaying property of the generalized eigenfunctions and adopt different strategies depending on the number $k$ of directions in which the generalized eigenfunction has exponential decay: A Feynman-Hellmann formula if $k=1$, a refined Taylor expansion of the potential if $k=2$, and an Agmon decay estimate if $k=3$.

\subsection{Change of variables}
\label{SS:CV}

Let us recall from Section \ref{SS:tangent} that the local smooth diffeomorphism $\diffeo^{\bx_0}$ maps a neighborhood $\cU_{\bx_0}$ of $\bx_0$ onto a neighborhood $\cV_{\bx_0}$ of $\bfz$ so that
\begin{equation}
\label{eq:diffeox0}
   \diffeo^{\bx_0}(\cU_{\bx_0}\cap\Omega) = \cV_{\bx_0}\cap\Pi_{\bx_0} \quad\mbox{and}\quad
   \diffeo^{\bx_0}(\cU_{\bx_0}\cap\partial\Omega) = \cV_{\bx_0}\cap\partial\Pi_{\bx_0} .
\end{equation}
The differential of $\diffeo^{\bx_0}$ at the point $\bx_0$ is the identity matrix $\Id_3$.
Let
$$
   \rJ:=\rd (\diffeo^{\bx_0})^{-1} \quad\mbox{and}\quad
   \rG:=\rJ^{-1}(\rJ^{-1})^{\top} 
$$ 
be the jacobian matrix of the inverse of $\diffeo^{\bx_0}$ and the associated metric. 
Lemma \ref{L:chgvar} leads to introduce the following formulas for the transformed magnetic potential $\pot$ and magnetic field $\tbB=\curl\pot$ in $\cV_{\bx_0}\cap\Pi_{\bx_0}$
\begin{equation}
\label{E:ABtilde}
   \pot:=\rJ^{\top} \big((\bA-\bA(\bx_0))\circ (\diffeo^{\bx_0})^{-1} \big) \quad 
   \mbox{and}\quad 
   \tbB:=|\det \rJ|\,\rJ^{-1} \big(\bB\circ (\diffeo^{\bx_0})^{-1} \big) \ . 
\end{equation}
We also introduce the phase shift
\begin{equation}
\label{E:phase}
   \zeta^{\bx_0}_h(\bx) = \re^{-i\langle \bA(\bx_0),\ee\bx\rangle/h},
   \quad\bx\in\cU_{\bx_0}.
\end{equation}
To any function $f$ in $H^1(\Omega)$ with support in $\cU_{x_0}$ corresponds the function
\begin{equation}
\label{E:psi}
   \psi:= (\overline{\zeta^{\bx_0}_h\!}\, f)\circ (\diffeo^{\bx_0})^{-1}
\end{equation} 
defined in $\Pi_{\bx_0}$, with support in $\cV_{\bx_0}$. 
For any $h>0$ we have 
\begin{equation}
\label{eq:shift}
   q_{h}[\bA,\Omega](f)=q_{h}[\bA-\bA(\bx_0),\Omega](\overline{\zeta^{\bx_0}_h\!}\, f)
\end{equation}
and thus
\begin{equation}
\label{E:chgGx0}
   q_{h}[\bA,\Omega](f)
   = q_{h}[\pot,\Pi_{\bx_0},\rG](\psi) 
   \quad \mbox{and} \quad \| f\|_{L^2(\Omega)}=\| \psi\|_{L^2_{\rG}(\Pi_{\bx_0})}\,,
\end{equation}
 where the quadratic forms $q_{h}[\bA,\Omega]$ and $q_{h}[\pot,\Pi_{\bx_0},\rG]$ are defined in \eqref{D:fq}  and \eqref{D:fqG}, respectively. 

Since $\rd \diffeo^{\bx_{0}}(\bx_{0})=\Id_{3}$ by definition, there holds
\begin{equation}
\label{eq:field0}
   \tbB(\bfz)=\bB(\bx_{0})\,.
\end{equation}
Likewise, let $\pot_{\bfz}$ be the 
linear part of $\pot$ at the vertex $\bfz$ of $\Pi_{x_0}$. 
Note that, as a consequence of \eqref{E:ABtilde}, $\pot(\bfz)=0$ and 
the two potentials $\pot_{\bfz}$ and $\bA_{\bx_{0}}$ coincide (we recall that $\bA_{\bx_{0}}$ is the linear part of $\bA$ at $\bx_0$).

\begin{lemma}
\label{L:changvar}
Let $r_0>0$ be such that $\cV_{\bx_0}$ contains the ball $\cB(\bfz,r_0)$ of center $\bfz$ and radius $r_0$. Then there exists a constant $C(\Omega)$ such that for all $r\in(0,r_0]$, if $\psi\in H^1(\Pi_{\bx_{0}})$ with $\supp(\psi)\subset \cB(\bfz,r)$ we have the two estimates
\begin{gather}
   \big|q_{h}[\pot,\Pi_{\bx_0}](\psi)-
   q_{h}[\bA-\bA(\bx_0),\Omega](\psi\circ\diffeo^{\bx_{0}})\big| 
   \leq C(\Omega)\, r\, q_{h}[\pot,\Pi_{\bx_0}](\psi) , 
\\[1ex]
   \big| \| \psi\|_{L^2(\Pi_{\bx_0})}-\| \psi\circ\diffeo^{\bx_{0}}\|_{L^2(\Omega)} \big| 
   \leq C(\Omega)\, r \,\| \psi\|_{L^2(\Pi_{\bx_0})} .
\end{gather}
\end{lemma}

\begin{proof}
Recall that $\rJ=\rd (\diffeo^{\bx_0})^{-1}$ and $\rG:=\rJ^{-1}(\rJ^{-1})^{\top}$. Since $\rd \diffeo^{\bx_{0}}(\bx_{0})=\rd (\diffeo^{\bx_{0}})^{-1}(\bfz)=\Id_{3}$ by definition, we have $\rJ(\bfz)=\Id_{3}$ and  $\rG(\bfz)=\Id_{3}$. We deduce 
\begin{equation}
\label{E:taylorG}
\| \rG-\Id_{3}\|_{L^{\infty}(\cB(\bfz,r))} \leq r \|\rG\|_{W^{1,\infty}(\cV_{\bx_0})}.
\end{equation}
Since $\Omega$ is assumed to be polyhedral, its curvature (curvature of the faces and curvature of the edges) is bounded, therefore there exists a uniform bound $C(\Omega)$ for the norm in ${W^{1,\infty}(\cV_{\bx_0})}$ of the metric $\rG=\rG^{\bx_0}$ when $\bx_0$ runs through $\overline\Omega$. Notice that $q_{h}[\pot,\Pi_{\bx_{0}}]=q_{h}[\pot,\Pi_{\bx_{0}},\Id_{3}]$, so we deduce the Lemma by using \eqref{E:taylorG} in \eqref{E:chgGx0}. 
\end{proof}

Therefore we are reduced to study the Laplacian with magnetic potential $\pot$ on the cone $\Pi_{\bx_{0}}$ with the identity metric.

\subsection{Construction of quasimodes}
\label{SS:QM}
Let $\bx_{0}\in\overline\Omega$ be a point satisfying \eqref{D:x0}. Thus $\bx_0$ minimizes the local ground energy. For shortness we denote by $\Lambda_{\bfz}$ this energy:
\begin{equation}
\label{eq:lam}
   \Lambda_{\bfz}=\En(\bB_{\bx_{0}},\Pi_{\bx_{0}}).
\end{equation} 
In order to prove Theorem \ref{T:generalUB}, we are going to construct a family of quasimodes $f_h\in H^1(\Omega)$ satisfying the estimate for $h\le h_0$ (with some chosen positive $h_0$)
\begin{equation}
\label{eq:qm}
   \frac{q_{h}[\bA,\Omega](f_{h})}{\|f_{h}\|^2}\leq 
   h\Lambda_{\bfz}+C(\Omega)(1+\|\bA\|_{W^{2,\infty}(\Omega)}^2)h^{5/4}.
\end{equation}
Let $\pot$ be the magnetic potential in the tangent cone $\Pi_{\bx_{0}}\cap\cV_{\bx_0}$ given by \eqref{E:ABtilde}. We recall that $\pot_{\bfz}$ is the linear part of $\pot$ at $\bfz$. 
We recall that since $\rd \diffeo^{\bx_{0}}(\bx_{0})=\Id_{3}$, the point values of the corresponding fields $\tbB(\bfz)$ and $\bB_{\bx_0}$ coincide. 

Theorem \ref{th:dicho} provides a singular chain\footnote{In the case (i) of Theorem \ref{th:dicho}, $\dx$ is the trivial chain $(\bx_{0})$ and $\Pi_{\dx}=\Pi_{\bx_{0}}$.} $\dx$ in $\gC_{\bx_0}(\Omega)\equiv \gC_{\bfz}(\Pi_{\bx_{0}})$ and its associated cone $\Pi_{\dx}$ such that the operator $H(\pot_{\bfz},\Pi_{\dx})$ has an admissible generalized eigenfunction $\Psi^\dx$ associated with the energy $\En(\bB_{\bx_{0}},\Pi_{\bx_{0}})=\Lambda_\bfz$. 
 
There exists a rotation $\udiffeo$ which transforms $\Pi_{\dx}$ into $\R^{3-k}\times \Upsilon$ so that in coordinates $\bx^\natural\in \R^{3-k}\times \Upsilon$, with $\Upsilon\in\ogP_k$ the generalized eigenfunction writes:
\begin{equation}
\label{D:Generalizedef}
\Psi^\dx(\bx) =\Psi^\natural(\bx^\natural) = \re^{\ri\vartheta(\by,\bz)}\,\Phi(\bz) \quad \mbox{with}\quad
\udiffeo\bx  =\bx^\natural=(\by,\bz)\in \R^{3-k}\times \Upsilon,
\end{equation}
 where $\Phi$ is an exponentially decreasing function.
The function $\Psi^\dx$ satisfies
\begin{equation}
\label{eq:genEP}
\begin{cases}
   (-i\nabla+\pot_\bfz)^2\Psi^\dx=\Lambda_{\bfz}\Psi^\dx &\mbox{in } \Pi_{\dx},\\
   (-i\partial_n+\mathbf{n}\cdot\pot_\bfz)\Psi^\dx=0 &\mbox{on } \partial\Pi_{\dx}.
\end{cases}
\end{equation}
Then the scaled function
\begin{equation}\label{eq.AGEsc}
\Psi^\dx_{h}(\bx):=\Psi^\dx\Big(\frac{\bx}{\sqrt{h}}\Big) ,\quad\mbox{for}\quad \bx\in\Pi_\dx,
\end{equation}
defines a generalized eigenfunction for the operator $H_{h}(\pot_\bfz,\Pi_{\dx})$ associated with $h\Lambda_{\bfz}$.

According as $\Pi_{\bx_{0}}$ equals $\Pi_{\dx}$ or not (cases (i) or (ii) in Theorem \ref{th:dicho}, respectively),  our constructions will be different, leading to two types of quasimodes qualified as ``sitting'' or ``sliding''. 

\subsubsection{Sitting quasimodes}
When we are in case (i) of the dichotomy given by Theorem~\ref{th:dicho}, $\dx=(\bx_0)$ and $\Pi_\dx=\Pi_{\bx_0}$. In a classical way, the construction amounts to realize a suitable cut-off of the scaled generalized eigenfunction $\Psi^\dx_{h}$.
For doing this, let us choose, once for all, a model cut-off function $\underline\chi\in \sC^\infty(\R^+)$ such that
\begin{equation}
\label{D:chi}
\underline\chi(r)=\begin{cases} 1\mbox{ if }r\leq 1,\\ 0\mbox{ if }r\geq 2.
\end{cases}
\end{equation}
For any $R>0$, let $\underline\chi_{R}$ be the cut-off function defined by
\begin{equation}
\label{D:chiR}
\underline\chi_{R}(r)=\underline\chi \left(\frac{r}{R} \right),
\end{equation}
and, finally
\begin{equation}\label{eq.chih}
\chi_{h}(\bx) = \underline\chi_{R}\left(\frac{|\bx|}{h^\delta} \right) 
= \underline\chi\left(\frac{|\bx|}{Rh^\delta} \right) 
\quad\mbox{ with }\quad 0\le \delta\le \frac 12\,.
\end{equation}
Here the exponent $\delta$ is the decay rate of the cut-off. It will be tuned to optimize remainders. We choose $R=1$ in the formula for the cut-off and set\footnote{The reason for the double notation $\varphi^\dx_{h}(\bx) =\psi^\dx_{h}(\bx)$ will appear clearly later on, see \eqref{eq.defQMa}-\eqref{eq.defQMb}.} 
\begin{equation}
\label{eq:psih}
   \varphi^\dx_{h}(\bx) =\psi^\dx_{h}(\bx) =\chi_{h}(\bx)\Psi^\dx_{h}(\bx)
\end{equation}
which provides a quasimode for $q_{h}[\pot_\bfz,\Pi_{\dx}]$ satisfying  $(-ih\partial_n+\mathbf{n}\cdot\pot_\bfz)\psi^\dx_h=0$ on $\partial\Pi_{\dx}$. Note that when $h$ decreases, the supports of $\varphi^\dx_{h}$ decrease while staying embedded in each other.

\subsubsection{Sliding quasimodes}
Now we are in case (ii) of the dichotomy given by Theorem~\ref{th:dicho}, which means that $\Lambda_\bfz=\seE(\bB_{\bx_{0}},\Pi_{\bx_{0}})$ and that there exists a chain $\dx\in \gC^*_{\bx_{0}}(\Omega)$ such that $\Lambda_\bfz=\En(\bB_{\dx},\Pi_{\dx})<\seE(\bB_{\dx},\Pi_{\dx})$. Let $\udiffeo^{0}\in\gO_3$ such that $\Pi_{\bx_{0}}=\udiffeo^{0}(\R^{3-d}\times \Gamma)$ where $\Gamma$ is the reduced cone of $\Pi_{\bx_{0}}$. Let $\Omega_0=\Gamma\cap\dS^{d-1}$ be the section of $\Gamma$. According to Remark \ref{rem:length2}, there exists $\bx_{1}\in \overline\Omega_0$ so that $\dx=(\bx_0,\bx_{1})$. Then we define the unitary vector $\dir$ by the formulas
\begin{equation}
\label{def:tau}
   \underline{\dir}:=(0,\bx_{1})\in \R^{3-d}\times \Gamma
   \quad\mbox{and}\quad
   \dir=\udiffeo^{0}\,\underline{\dir}\in\dS^2.
\end{equation}

\begin{remark}
\label{R:decritdir}
The cone $\Pi_{\dx}$ can be the full space, a half-space or a wedge, and $\dir$ gives a direction associated with $\Pi_{\dx}$ starting from the origin $\bfz$ of $\Pi_{\bx_{0}}$:
\begin{enumerate}
\item If $\Pi_{\dx}\equiv\R^3$, $\dir$ belongs to the interior of $\Pi_{\bx_{0}}$.
\item If $\Pi_{\dx}\equiv\R^3_{+}$, $\dir$ belongs to a face of $\Pi_{\bx_{0}}$.
\item If $\Pi_{\dx}\equiv\cW_{\alpha}$, $\dir$ belongs to an edge  of $\Pi_{\bx_{0}}$. 
\end{enumerate}
Note that unless we are in the latter case ($\Pi_{\dx}$ is a wedge), the choice of $\dir$ is not unique.
\end{remark}

At this point, let us emphasize that we need that our quasimodes on the tangent cone $\Pi_{\bx_{0}}$ are compatible with the structure of $\Pi_\dx$ that provides the admissible generalized eigenvector $\Psi^\dx$.
Here the vector $\dir$ introduced in \eqref{def:tau} comes into play. Instead of being concentric like before, the supports of the quasimodes are {\em sliding} along $\dir$ and concentrate at the same rate $h^\delta$:
We define our quasimode $\varphi_{h}^{\dx}$ by setting
\begin{subequations}
\label{eq.defQM}
\begin{gather}
\label{eq.defQMa}
   \psi_{h}^{\dx}(\bx) = 
   \chi_{h}(\bx)\Psi^\dx_{h}(\bx),
   \quad(\forall\bx\in\Pi_{\dx}) \, , \\
\label{eq.defQMb}
   \varphi_{h}^{\dx}(\bx) = 
   \re^{-\ri\langle\pot_\bfz(\dec),\ee\bx\rangle/h} \,
   \psi_{h}^{\dx}(\bx-\dec) \quad\mbox{with}\quad
   \dec = h^\delta\dir \quad(\forall\bx\in\Pi_{\bx_{0}}) \,.
\end{gather}
\end{subequations}

The vector $\dec$ is a shift and the translation $\diffeoT^{\dec}$ by $-\dec$ sends a neighborhood of $\dec$ in $\Pi_{\bx_0}$ onto a neighborhood of $\bfz$ in $\Pi_\dx$.
We check that
\begin{equation}
\label{eq.A0tildeA0}
   (-ih\nabla+\pot_\bfz(\bx))\varphi_{h}^{\dx}(\bx)
   =   \re^{-\ri\langle\pot_\bfz(\dec),\ee\bx\rangle/h} \,
   (-ih\nabla+\pot_\bfz(\bx-\dec))\psi^\dx_{h}(\bx-\dec) ,
   \quad \forall\bx\in\Pi_{\bx_{0}}\,.
\end{equation}

We  choose $R>0$ in \eqref{D:chiR} such that $\cB(\dir,2R)\cap \Pi_{\bx_{0}}=\cB(\dir,2R)\cap \Pi_{\dx}$. Note that $R$ depends only on the geometry of $\Omega$ near $\bx_{0}$. 
Hence, using the translation $\diffeoT^{\dir}:\bx\mapsto\bx-\dir$,
\begin{equation}
\label{eq:R}
   \diffeoT^{\dir}\Big\{\supp \big( \underline\chi_{R}(\cdot-\dir)\big)\cap \Pi_{\bx_{0}}\Big\}
   =\supp(\underline\chi_{R})\cap \Pi_{\dx} \ . 
\end{equation}
It follows by scaling that, with the translation $\diffeoT^{\dec}:\bx\mapsto\bx-\dec$,
$$
   \forall h>0, \quad \diffeoT^{\dec}
   \Big\{\supp (\varphi_{h}^{\dx})\cap \Pi_{\bx_{0}}\Big\}
   =\supp(\psi_{h}^{\dx})\cap \Pi_{\dx}  \, .
$$ 
Therefore, in virtue of \eqref{eq.A0tildeA0} we have
\begin{equation}\label{eq.qhpotPidx}
   q_{h}[\pot_\bfz,\Pi_{\bx_{0}}](\varphi_{h}^{\dx})  =  
   q_{h} [\pot_\bfz,\Pi_{\dx}](\psi^\dx_{h}) \, , 
\end{equation}
with $\psi^\dx_{h}$ satisfying the Neumann boundary conditions $(-ih\partial_n+\mathbf{n}\cdot\pot_\bfz)\psi^\dx_h=0$ on $\partial\Pi_{\dx}$, which allows to take advantage of better cut-off estimates, cf.\ Lemma~\ref{lem:tronc}. 

The localization of the quasimodes $\varphi_{h}^{\dx}$ can be deduced from Remark \ref{R:decritdir}:
\begin{enumerate}
\item If $\Pi_{\dx}\equiv\R^3$, 
 $\varphi_{h}^{\dx}$ is centered in the interior of $\Pi_{\bx_{0}}$.
\item If $\Pi_{\dx}\equiv\R^3_{+}$,
 $\varphi_{h}^{\dx}$ is centered on the face of $\Pi_{\bx_{0}}$ associated with $\dir$.
\item If $\Pi_{\dx}\equiv\cW_{\alpha}$,
$\varphi_{h}^{\dx}$ is centered on the edge of $\Pi_{\bx_{0}}$ associated with $\dir$.
\end{enumerate}

\subsubsection{Synthesis}
\label{d:qm}
We have constructed the functions $\varphi_{h}^{\dx}\in\dom(H_{h}(\pot_{\bfz},\Pi_{\bx_{0}}))$
by formulas \eqref{D:Generalizedef}-\eqref{eq:R}: 

\begin{subequations}
\label{eq.defQM1}
(i) If $\dx=(\bx_0)$, the functions $\varphi_{h}^{\dx}$ are {\em sitting quasimodes}:
\begin{equation}\label{eq.defQMbis0}
   \varphi_{h}^{\dx}(\bx) = 
   \underline\chi_{R}\left(\frac{|\bx|}{h^{\delta}}\right)  
   \Psi^\dx\left(\frac{\bx}{h^{1/2}}\right),
   \quad\mbox{for}\ \ \bx\in\Pi_{\bx_{0}} ,
\end{equation}

(ii) If $\dx=(\bx_0,\bx_1)$, the functions $\varphi_{h}^{\dx}$ are {\em sliding quasimodes}:
\begin{equation}\label{eq.defQMbis}
   \varphi_{h}^{\dx}(\bx) =
   \re^{-\ri\langle\pot_\bfz(\dec),\ee\bx\rangle/h} \,
   \underline\chi_{R}\left(\frac{|\bx-\dec|}{h^{\delta}}\right)  
   \Psi^\dx\left(\frac{\bx-\dec}{h^{1/2}}\right),
   \quad\mbox{for}\ \ \bx\in\Pi_{\bx_{0}}\ \ \mbox{and}\ \ \dec=h^{\delta} \dir,
\end{equation}
\end{subequations}

\subsection{Estimation of the quasimodes}
\label{SS:estimQM}
We separately estimate the cut-off errors, the linearization errors, and the error due to the change of metric. 
\subsubsection{Cut-off effect}
In both situations of sitting and sliding quasimodes, relying on formulas \eqref{eq:psih},  \eqref{eq.defQMa} and \eqref{eq.qhpotPidx}, we can apply Lemma~\ref{lem:tronc} with the magnetic potential $\pot_\bfz$, $\chi=\chi_{h}$ and $\psi=\Psi^{\dx}_{h}$, we obtain for the Rayleigh quotient of $\varphi_{h}^{\dx}$:
\begin{subequations}
\label{eq:tronc1}
\begin{equation}
   \frac{q_{h}[\pot_\bfz,\Pi_{\bx_{0}}](\varphi_{h}^{\dx})}{\|\varphi_{h}^{\dx}\|^2}=  
   \frac{q_{h}[\pot_\bfz,\Pi_{\dx}](\psi_{h}^{\dx})}{\|\psi_{h}^{\dx}\|^2}=  
   \frac{q_{h}[\pot_\bfz,\Pi_{\dx}](\chi_{h}\Psi^{\dx}_{h})}{\|\chi_{h}\Psi^{\dx}_{h}\|^2} =
   h\Lambda_{\bfz} + h^2\rho_{h} 
\end{equation}
with
\begin{equation}
\label{eq:rhoh}
  \rho_{h}=\frac{\| \,|\nabla\chi_{h}|\, \Psi^{\dx}_{h}\|^2}{\|\chi_{h}\Psi^{\dx}_{h}\|^2}.
\end{equation} 
\end{subequations}
The fact that $\Psi^{\dx}_{h}$ belongs to $\dom_{\,\loc} (\OP_{h}(\pot_{\bfz},\Pi_{\dx}))$ is essential for the validity of the identities above.

The following lemma estimates the remainder due to the cut-off effect:

\begin{lemma}\label{lem:rhoh}
Let $\Psi$ be an admissible generalized eigenvector given by \eqref{D:Generalizedef} and  $\Psi_{h}$  the rescaled associated function given by \eqref{eq.AGEsc}. Let $\chi_{h}$ be the cut-off function defined by \eqref{eq.chih} involving parameters $R>0$ and $\delta\in[0,\frac12]$. 
Then there exist constants $C_0>0$ and $c_0>0$ depending only on $h_0>0$, $R_0>0$ and $\Psi$ such that
$$
   \rho_{h}=\frac{\|\,|\nabla\chi_{h}|\, \Psi_{h}\|^2}{\|\chi_{h}\Psi_{h}\|^2}\leq
   \begin{cases} C_0\, h^{-2\delta} & \mbox{ if }k<3,\\[0.5ex]
   C_0\, h^{-2\delta} \,\re^{-c_0h^{\delta-1/2}} & \mbox{ if }k=3, 
   \end{cases}
   \qquad \forall R\ge R_0,\ \forall h\le h_0,\ \forall\delta\in[0,\tfrac12]\,.
$$
\end{lemma}

\begin{proof}
By assumption $\Psi(\bx) = \re^{\ri\vartheta(\by,\bz)}\,\Phi(\bz)$ for $\udiffeo\bx=(\by,\bz)\in \R^{3-k}\times \Upsilon$ and there exist positive constants $c_\Psi,C_\Psi$ such that
\begin{equation}
\label{eq:agmon}
   \int_{\Upsilon} \re^{2 c_\Psi |\bz|}|\Phi(\bz)|^2 \rd\bz\leq C_\Psi\|\Phi\|^2_{L^2(\Upsilon)}.
\end{equation}
 Let us set $T=Rh^\delta$, so that $\chi_h(\bx) = \underline{\chi}(|\bx|/T)$.

Let us first give an upper bound for $\|\,|\nabla\chi_{h}|\,\Psi_{h}\|$:\\
If $k<3$, then 
\begin{eqnarray*}
\|\,|\nabla\chi_{h}|\,\Psi_{h}\|^2 
&\leq& CT^{-2} \int_{|\by|\leq 2T} \rd \by\ 
\int_{\Upsilon\, \cap\, \{|\bz|\leq2T\}} 
\left|\Phi\Big(\frac{\bz}{\sqrt h}\Big)\right|^2\rd \bz \\ 
&\leq& C T^{-2}\, T^{3-k} \, h^{k/2} \|\Phi\|^2_{L^2(\Upsilon)},
\end{eqnarray*} 
else, if $k=3$ 
\begin{eqnarray*}
\|\,|\nabla\chi_{h}|\,\Psi_{h}\|^2 
&\leq& C T^{-2} \int_{\Upsilon\, \cap\, \{T \leq |\bz|\leq2T\}} \left|\Phi\Big(\frac{\bz}{\sqrt h}\Big)\right|^2\rd \bz \\
&\leq& C T^{-2} \, h^{k/2} \int_{\Upsilon\, \cap\, 
   \big\{Th^{-\frac12} \leq |\bz|\leq2Th^{-\frac12}\big\}} 
   \left|\Phi (\bz)\right|^2\rd \bz \\
&\leq& C T^{-2} \, h^{k/2}  \ \re^{-2 c_\Psi T /\sqrt{h}}
   \int_{\Upsilon\, \cap\, 
   \big\{Th^{-\frac12} \leq |\bz|\leq2Th^{-\frac12}\big\}} 
   \re^{2 c|\bz|} \left|\Phi(\bz)\right|^2\rd \bz\\
&\leq& C T^{-2} \, h^{k/2}  \ \re^{-2 c_\Psi T/\sqrt{h}}\, \|\Phi\|^2_{L^2(\Upsilon)}.
\end{eqnarray*}
Let us now consider $\|\chi_{h}\Psi_{h}\|$ (we use that $2|\by|<R$ and $2|\bz|<R$ implies $|\bx|<R$):
\begin{eqnarray}
   \|\chi_{h}\Psi_{h}\|^2  &\geq& \int_{ 2|\by|\leq T} \rd \by\ 
   \int_{\Upsilon\, \cap\, \{2|\bz|\leq T\}} 
   \left|\Phi\Big(\frac{\bz}{\sqrt h}\Big)\right|^2\rd \bz \nonumber \\
\label{eq:minL2a}
   &\geq & C T^{3-k} h^{k/2} \, \int_{\Upsilon\, \cap\, \big\{2|\bz|\leq Th^{-\frac12}\big\}} 
   \left|\Phi(\bz)\right|^2\rd \bz \\
\label{eq:minL2b}
   &\geq & C T^{3-k} h^{k/2} \, \cI(Th^{-\frac12}) \, \|\Phi\|^2_{L^2(\Upsilon)}
\end{eqnarray}
where we have set for any $S\ge0$
\[
   \cI(S) := \bigg(\int_{\Upsilon\, \cap\, \{2|\bz|\leq S\}} 
   \left|\Phi(\bz)\right|^2\rd \bz \bigg)  \bigg(
   \int_{\Upsilon} 
   \left|\Phi(\bz)\right|^2\rd \bz\bigg)^{-1}.
\]
The function $S\mapsto\cI(S)$ is continuous, non-negative and non-decreasing on $[0,+\infty)$. It is moreover {\em increasing and positive} on $(0,\infty)$ since $\Phi$, as a solution of an elliptic equation with polynomial coefficients and null right hand side, is analytic inside $\Upsilon$.  
Consequently, $\cI(T h^{-\frac12})=\cI(R h^{\delta-\frac12})$ is uniformly bounded from below for $R\geq R_{0}$, $h\in (0,h_{0})$, $\delta\in [0,\frac 12]$ and thus
\begin{eqnarray*}
\rho_{h}&\leq& 
\begin{cases}
C T^{-2}\,\big\{\cI(Th^{-\frac12})\big\}^{-1}
\le C_0 h^{-2\delta}
&\mbox{ if }k<3,\\[0.5ex]
C T^{-2} \ \re^{-2 c_\Psi T/\sqrt{h}} \,\big\{\cI(Th^{-\frac12})\big\}^{-1}
\le C_0 h^{-2\delta} \re^{-c_0 h^{\delta-1/2}} &\mbox{ if }k=3,
\end{cases}
\end{eqnarray*}
where the constants $C_0$ and $c_0$ in the above estimation depend only on the lower bound $R_0$ on $R$, the upper bound $h_0$ on $h$, and on the model problem associated with $\bx_{0}$, provided $\delta\in[0,\frac12]$. Lemma \ref{lem:rhoh} is proved. 
\end{proof}

 If the exponent $\delta$ is bounded from above by a number $\delta_0<\frac12$, we obtain the following improvement of the previous lemma.

\begin{lemma}
\label{lem:rhohb}
Under the conditions of Lemma \ref{lem:rhoh}, let $\delta_0<\frac12$ be a positive number. Let $R_0>0$.
Then there exist constants $h_0>0$, $C_1>0$ and $c_1>0$ depending only on $R_0$, $\delta_0$ and on the constants $c_\Psi$, $C_\psi$ in \eqref{eq:agmon} such that
$$
   \rho_{h}=\frac{\|\,|\nabla\chi_{h}|\, \Psi_{h}\|^2}{\|\chi_{h}\Psi_{h}\|^2}\leq
   \begin{cases} C_1\, h^{-2\delta} & \mbox{ if }k<3,\\[0.5ex]
   C_1\, \re^{-c_1h^{\delta-1/2}} & \mbox{ if }k=3, 
   \end{cases}
   \qquad \forall R\ge R_0,\ \forall h\le h_0,\ \forall\delta\in[0,\delta_0]\,.
$$
\end{lemma}

\begin{proof}
We obtain an upper bound of $\|\,|\nabla\chi_{h}|\,\Psi_{h}\|^2$ as in the proof of Lemma~\ref{lem:rhoh}. Let us now deal with the lower-bound of $\|\chi_{h}\Psi_{h}\|^2$. With $T=R h^\delta$, we have
\begin{eqnarray}
\|\chi_{h}\Psi_{h}\|^2 
   &\geq & C T^{3-k} h^{k/2} \, \int_{\Upsilon\, \cap\, \big\{2|\bz|\leq Th^{-\frac12}\big\}} 
   \left|\Phi(\bz)\right|^2\rd \bz 
\nonumber\\
   &\geq& C  T^{3-k} h^{k/2} \,  \left(1-C_{\Psi}\re^{-c_{\Psi} R h^{\delta-1/2}}\right) 
   \|\Phi\|^2_{L^2(\Upsilon)}.\label{eq.minpsih}
\end{eqnarray}
Since $0\leq\delta\leq\delta_{0}<\frac 12$, there holds $C_{\Psi}\re^{-c_{\Psi} R h^{\delta-1/2}}<\frac12$ for $h$ small enough or $R$ large enough. Thus we deduce the lemma.
\end{proof}

\subsubsection{Linearization}
\label{SS:linear}
Note that in any case (sitting or sliding) the quasimode $\varphi_{h}^{\dx}$ \eqref{eq.defQM1} on $\Pi_{\bx_0}$ belongs to $\dom(q_{h}[\pot,\Pi_{\bx_{0}}])$. 
We can compare the quadratic form for the magnetic potential and its linear part by using \eqref{eq:diffAA'} with $\bA=\pot$, $\bA'=\pot_\bfz$, $\cO=\Pi_{\bx_{0}}$, and $\psi=\varphi_{h}^{\dx}$:
\begin{multline}
\label{eq:diffA}
  q_{h}[\pot,\Pi_{\bx_{0}}](\varphi_{h}^{\dx}) =   q_{h}[\pot_\bfz,\Pi_{\bx_{0}}](\varphi_{h}^{\dx}) \\
 +2  \Re\int_{\Pi_{\bx_{0}}} (-ih\nabla+\pot_\bfz)\varphi_{h}^{\dx}(\bx)\cdot(\pot-\pot_\bfz)(\bx)\overline{\varphi_{h}^{\dx}(\bx)}\,\rd \bx 
  + \|(\pot-\pot_\bfz)\varphi_{h}^{\dx}\|^2.
\end{multline}

Combining with \eqref{eq:tronc1} we get
\begin{multline}
\label{eq:troncA}
  \frac{q_{h}[\pot,\Pi_{\bx_{0}}](\varphi_{h}^{\dx})}{\|\varphi_{h}^{\dx}\|^2} 
  = h \Lambda_{\bfz} + h^2\rho_{h} 
  \\[-1.5ex]
  + \frac{ 2 \Re\int_{\Pi_{\bx_{0}}} (-ih\nabla+\pot_\bfz)\varphi_{h}^{\dx}(\bx)\cdot(\pot-\pot_\bfz)(\bx)\overline{\varphi_{h}^{\dx}(\bx)}\,\rd \bx }{\|\varphi_{h}^{\dx}\|^2}
  + \frac{\|(\pot-\pot_\bfz)\varphi_{h}^{\dx}\|^2}{\|\varphi_{h}^{\dx}\|^2}.
\end{multline}
By Cauchy-Schwarz inequality, we obtain easily
\begin{subequations}
\label{eq:troncA2}
\begin{equation}
  \frac{q_{h}[\pot,\Pi_{\bx_{0}}](\varphi_{h}^{\dx})}{\|\varphi_{h}^{\dx}\|^2} 
  \leq h\Lambda_{\bfz} + h^2\rho_{h} +2 \sqrt h \sqrt{\Lambda_{\bfz}+h\rho_{h}}\ a_{h} + a_{h}^2
 \end{equation}
where we have set
\begin{equation}\label{eq.defah}
  a_{h}=\frac{\|(\pot-\pot_\bfz)\varphi_{h}^{\dx}\|}{\|\varphi_{h}^{\dx}\|}.
\end{equation}
\end{subequations}
We now estimate the remainder due to the linearization of $\pot$. 

\begin{lemma}\label{lem.TaylorA}
Let $r_0>0$ be such that $\cV_{\bx_0}\supset\cB(\bfz,r_0)$. 
For any $r\in(0,r_0]$, we have
\begin{equation}
\label{E:taylorA}
  \forall \bx\in\cB(\bfz,r),\quad |\pot(\bx)-\pot_\bfz(\bx)|\leq 
  \tfrac1{2} \|\pot\|_{W^{2,\infty}(\cB(0,r))} |\bx|^2\, .
\end{equation}
\end{lemma}
\begin{proof}
Since $\pot_\bfz$ is the linear part of $\pot$ and $\pot(\bfz)=0$, the formula is deduced from the Taylor expansion of $\pot$ at $\bx=\bfz$. 
\end{proof}

By construction, there exists $C(\Omega)>0$ such that the support of $\varphi_{h}^{\dx}$ is included in the ball $\cB(\bfz,C(\Omega)h^{\delta})$. 
Consequently, we obtain immediately
\begin{equation}\label{eq.ah2delta}
a_{h}\leq C(\Omega)\|\pot\|_{W^{2,\infty}(\supp(\varphi_{h}^{\dx}))} h^{2\delta}.
\end{equation}
Moreover using the definition of $\pot$ (see \eqref{E:ABtilde}) we get
\begin{align*}
\|\pot\|_{W^{2,\infty}(\supp(\varphi_{h}^{\dx}))} & \leq \left(1+\|\Id_{3}-\rJ\|_{L^{\infty}(\cV_{\bx_{0}})}h^{\delta}\right)\|\bA\|_{W^{2,\infty}(\cU_{\bx_{0}})}
\\
& \leq \left(1+C(\Omega)h^{\delta}\right)\|\bA\|_{W^{2,\infty}(\Omega)}
\end{align*}
Thus, putting this last inequality in \eqref{eq.ah2delta},  we deduce
\begin{equation}\label{eq.ah54}
a_{h}\leq C(\Omega)\|\bA\|_{W^{2,\infty}(\Omega)} h^{2\delta}.
\end{equation}
Combining \eqref{eq:troncA2}, \eqref{eq.ah54} and Lemma \ref{lem:rhoh} we get 
\begin{equation}\label{eq:Ray1}
\frac{q_{h}[\pot,\Pi_{\bx_{0}}](\varphi_{h}^{\dx})}{\|\varphi_{h}^{\dx}\|^2} \leq h\Lambda_{\bfz}+C(\Omega)(1+\| \bA \|_{W^{2,\infty}(\Omega)}^2)
(h^{2-2\delta}+ h^{\frac12+2\delta}+ h^{4\delta}) \, . 
\end{equation}
Note that we have also used $\Lambda_{\bfz} \leq \| \bB \|_{L^{\infty}(\Omega)}\le \| \bA \|_{W^{1,\infty}(\Omega)}$ (since $\bB = \curl\bA$) in order to control the cross term $\sqrt{h}\sqrt{\Lambda_{\bfz}+h\rho_{h}}a_{h}$ in the right hand side of \eqref{eq:troncA2}.
   
\subsubsection{Quasimode on $\Omega$ and estimation of the remainders}
\label{SS:qmfinal}
We now define a quasimode for $q_{h}[\bA,\Omega]$. Let us note that for $h$ small enough, in any situation (sitting or sliding) $\varphi_{h}^{\dx}$  is supported in $\cV_{\bx_{0}}$. Therefore we can define $f_{h}$ by 
\begin{equation}
\label{D:qmsuromega}
   f_{h}(\bx)=\varphi_{h}^{\dx}\circ \diffeo^{\bx_{0}}(\bx)\; \zeta^{\bx_0}_h(\bx),\quad \bx\in\cU_{\bx_0} \ , 
\end{equation}
where the phase shift $\zeta^{\bx_0}_h$ was introduced in \eqref{E:phase}.
We extend $f_{h}$ by zero, thus defining a function $f_{h}\in H^1(\Omega)$. 
Combining \eqref{eq:Ray1} and \eqref{eq:shift} with Lemma \ref{L:changvar} for $r=h^{\delta}$ we get 
\begin{equation*}
  \frac{q_{h}[\bA,\Omega](f_{h})}{\|f_{h}\|^2} \leq 
  \left(h\Lambda_{\bfz}+C(\Omega)(1+\| \bA \|_{W^{2,\infty}(\Omega)}^2)
  (h^{2-2\delta}+h^{\frac12+2\delta}+h^{4\delta})\right)(1+C(\Omega)h^{\delta}) \ . 
\end{equation*}
Therefore there exists a constant $C(\Omega)>0$ such that 
\begin{equation}\label{eq:Ray3}
  \frac{q_{h}[\bA,\Omega](f_{h})}{\|f_{h}\|^2} \leq 
  h\Lambda_{\bfz}+C(\Omega)(1+\| \bA \|_{W^{2,\infty}(\Omega)}^2)
  (h^{2-2\delta}+h^{\frac12+2\delta}+h^{4\delta}+h^{1+\delta}) \, .
  \end{equation}
We optimize this upper bound by taking $\delta=\frac{3}{8}$, which provides immediately estimate \eqref{eq:qm}. The min-max principle then yields Theorem \ref{T:generalUB}.

\subsubsection{Improvement in case of corner concentration}
When the geometry minimizing the energy is given by a corner  whose tangent problem has an eigenvalue under its essential spectrum, we get a better upper bound by improving the estimate on $a_{h}$:

\begin{proposition}
\label{P:Cornerconcentration}
Let $\Omega\in \ogD(\R^3)$ be a polyhedral domain, $\bA\in W^{2,\infty}(\overline{\Omega})$ be a twice differentiable magnetic potential such that the associated magnetic field $\bB$ does not vanish on $\overline\Omega$. We assume moreover that there exists a corner $\bx_{0}\in \overline{\Omega}$ such that 
$$\sE(\bB,\Omega)=\En(\bB_{\bx_{0}},\Pi_{\bx_{0}})<\seE(\bB_{\bx_{0}},\Pi_{\bx_{0}}).$$ 
Then there exist $C(\Omega)>0$ and $h_{0}>0$ such that
$$ \forall h\in (0,h_{0}), \quad \lambda_{h}(\bB,\Omega) \leq 
h\sE(\bB,\Omega)+C(\Omega)(1+\|\bA\|_{W^{2,\infty}(\Omega)}^2)\, h^{3/2}|\log h| \, .$$
\end{proposition}

\begin{remark}
As in Corollary \ref{co:T:generalUBB} we can express this upper bound in function of the H\"older norm of the magnetic field. 
\end{remark}

\begin{proof}
Since $\Pi_{\bx_{0}}$ is a polyhedral corner and $\En(\bB_{\bx_{0}},\Pi_{\bx_{0}})<\seE(\bB_{\bx_{0}},\Pi_{\bx_{0}})$, by Proposition \ref{prop:cone-ess} the generalized eigenfunction $\Psi$ of $\OP(\bA_{\bx_{0}},\Pi_{\bx_{0}})$ provided by Theorem \ref{th:dicho} is an eigenfunction and has exponential decay when $|\bx|\to+\infty$. Here $\dx=(\bx_0)$, the quasimode $\varphi^{\dx}_h$ is sitting and defined by \eqref{eq:psih}. From now on we may drop in this proof the superscript $\dx$. Using \eqref{E:taylorA} we get $C(\Omega)>0$ such that
$$\forall \bx\in \supp (\varphi_{h}), \quad |(\pot-\pot_\bfz)(\bx)| \leq C(\Omega) \|\pot\|_{W^{2,\infty}(\supp(\varphi_{h}))}|\bx|^2\, .$$
Using the change of variable ${\bf X}=\bx h^{-1/2}$ and the exponential decay of $\Psi$ we get
$$a_{h}=\frac{\|(\pot-\pot_\bfz)\varphi_{h}\|}{\|\varphi_{h}\|} \leq C(\Omega)  \|\pot\|_{W^{2,\infty}(\supp(\varphi_{h}))} h,$$
where $a_{h}$ is set in \eqref{eq.defah}. Now \eqref{eq:troncA2} provides, with Lemma \ref{lem:rhoh}, for any $\delta\in(0,\frac12]$
\begin{eqnarray*}
\frac{q_{h}[\pot,\Pi_{\bx_{0}}](\varphi_{h})}{\|\varphi_{h}\|^2} 
&\leq& h\Lambda_{\bfz}+ C(\Omega)\, h^{2-2\delta}\re^{-ch^{\delta-\frac 12}}
 + C(\Omega)\| \bA \|_{W^{2,\infty}(\Omega)} h^{3/2}+ C(\Omega)\| \bA \|^2_{W^{2,\infty}(\Omega)} h^2\\
&\leq& h\Lambda_{\bfz}+ C(\Omega)(1+\|\bA \|^2_{W^{2,\infty}(\Omega)}) 
 \, (h^{2-2\delta}\re^{-ch^{\delta-\frac 12}}+h^{3/2}).
\end{eqnarray*}
Therefore the quasimode $f_{h}$ defined in \eqref{D:qmsuromega} satisfies
\begin{eqnarray*}
\frac{q_{h}[\bA,\Omega](f_{h})}{\|f_{h}\|^2} 
&\leq (1+C(\Omega)h^\delta) \big\{h\Lambda_{\bfz}+ C(\Omega)(1+\|\bA \|^2_{W^{2,\infty}(\Omega)}) 
 (h^{2-2\delta}\re^{-ch^{\delta-\frac 12}}+h^{3/2})\big\} \\
&\leq h\Lambda_{\bfz}+ C(\Omega)(1+\|\bA \|^2_{W^{2,\infty}(\Omega)}) 
 \big\{ h^{1+\delta} + h^{2-2\delta}\re^{-ch^{\delta-\frac 12}}+h^{3/2}\big\} .
\end{eqnarray*}
Here $C(\Omega)$ denotes various constants independent from $h\le h_0$ and $\delta\le\frac12$.
We optimize this by taking $\delta=\frac12-\varepsilon(h)$ with $\varepsilon(h)$ so that
$h^{1+\delta} = h^{2-2\delta}\re^{-ch^{\delta-\frac 12}}$, i.e.
\[
   h^{\frac32-\varepsilon(h)} = h^{1+2\varepsilon(h)}\re^{-ch^{-\varepsilon(h)}} .
\] 
We find
\[
   \re^{ch^{-\varepsilon(h)}} = h^{-\frac12+3\varepsilon(h)},\quad\mbox{i.e.}\quad
   h^{-\varepsilon(h)} = \tfrac{1}{c}(-\tfrac12+3\varepsilon(h))\log h \,.
\] 
The latter equation has one solution $\varepsilon(h)$ which tends to $0$ as $h$ tends to $0$.
Replacing $h^{-\varepsilon(h)}$ by the value above in $h^{\frac32-\varepsilon(h)}$, we find that the remainder is a $O(h^{3/2}|\log h|)$
and the min-max principle provides the proposition.
\end{proof}

\subsection{Improvement for more regular magnetic fields}
\label{S:improvement4/3}
The object of this section is the proof of Theorem~\ref{T:sUB}.
In fact, our proof of the $h^{5/4}$ upper bound as done in previous sections weakly uses the exponential decay of generalized eigenfunctions in some directions. It would also work with purely oscillating generalized eigenfunctions. 

Now the proof of the $h^{4/3}$ upper bound makes a more extensive use of fine properties of the model problems: First, the decay properties of admissible generalized eigenvectors and their stability upon perturbation, and second, the Lipschitz regularity of the ground energy depending on the magnetic field, cf.\ Lemma \ref{L:Elip}. 

We recall that $\bx_{0}\in \overline{\Omega}$ is a point such that 
$\En(\bB_{\bx_{0}}\ee,\Pi_{\bx_{0}})= \sE(\bB \ee,\Omega)=:\Lambda_{\bfz}$.
We apply Theorem \ref{th:dicho} (and Remark~\ref{rem:chaine}) with $(\bB_{\bx_{0}},\Pi_{\bx_{0}})$: 
we denote by $\dx$ the corresponding singular chain which satisfies
$$
\sE(\bB \ee,\Omega) = \En(\bB_{\bx_{0}},\Pi_{\bx_{0}})=
\En(\bB_{\bx_{0}},\Pi_{\dx})< \seE(\bB_{\bx_{0}},\Pi_{\dx}).
$$ 
We now split our analysis between three geometric configurations depending on the number of variables $k$ in which the generalized eigenfunction has exponential decay (see Section \ref{s:age}):
 \begin{itemize}
\item[(G1)] $\Pi_{\dx}$ is a half-space and $\bB_{\bx_{0}}$ is tangent to the boundary.
\item[(G2)] We are in one of the following situations:
\begin{itemize}
\item Either $\Pi_{\dx}$ is a wedge,
\item or $\Pi_{\dx}$ is a half-space and $\bB_{\bx_{0}}$ is neither tangent nor normal to the boundary. 
\item or $\Pi_{\dx}=\R^3$.
\end{itemize}
 \item[(G3)] $\Pi_{\dx}$ is a polyhedral cone of dimension 3 and coincides with $\Pi_{\bx_0}$.
 \end{itemize}
Let us now deal with each situation. The arguments are specific to each case. 
 
\paragraph{Assume that we are in situation (G3)} 
This means that $\bx_0$ is a corner and that we have the strict inequality $\En(\bB_{\bx_{0}},\Pi_{\bx_{0}})<\seE(\bB_{\bx_{0}},\Pi_{\bx_{0}})$. In that case we can rely on Proposition \ref{P:Cornerconcentration} in which we have already proved a better upper bound for $\lambda_{h}(\bB,\Omega)$, even with a weaker regularity assumption on the magnetic field.
  
\paragraph{Assume that we are in situation (G2)}
The generalized eigenfunction $\Psi$ associated with $\OP(\pot_{\bfz},\Pi_{\dx})$ has two directions of decay, $z_1$ and $z_2$, leaving one direction $y$ with a purely oscillating character. In this case, we are going to improve the linearization error: Until now we have used that $\pot(\bx)-\pot_\bfz(\bx)$ is a $O(|\bx|^2)$. Here, by a suitable phase shift (which corresponds to a change of gauge), we can eliminate from this error the term in $O(|y|^2)$, replacing it by a $O(|y|^3)$. The other terms containing at least one power of $|\bz|$, we can take advantage of the decay of $\Psi$. The sitting modes will be constructed following exactly this strategy, whereas concerning sliding modes, we have to linearize the potential at the point $\dec:=h^{\delta}\dir$, instead of $\bfz$ as previously. Let us develop details now.

Quasimodes $f_h$ on $\Omega$ are still defined by the formula \eqref{D:qmsuromega}, but we alter now the definition of $\varphi^{\dx}_{h}$. We first treat sitting modes, and second, sliding modes.

-- {\em Sitting quasimodes.} This is the case when $\dx=(\bx_0)$. Here we use the admissible generalized eigenvector $\Psi^{\natural}$  in natural variables as introduced in \eqref{D:Generalizedef} and its scaled version $\Psi^{\natural}_h$. We set
\begin{equation}
\label{eq:psidx}
    \psi^{\natural}_h(\bx^{\natural}) = 
    \chi_h(\bx^{\natural})\, \Psi^{\natural}_h(\bx^{\natural}) = 
   \underline\chi_{R}\left(\frac{|\bx^{\natural}|}{h^{\delta}}\right)  
   \Psi^{\natural}\left(\frac{\bx^{\natural}}{h^{1/2}}\right),\quad \bx^{\natural}\in \R\times\Upsilon\,.
\end{equation}
We are going to apply Lemma \ref{L:d2ell0} in the variables $(y,\bz)=\bx^\natural$.
We recall that $\udiffeo$ is the rotation ($\rJ$ its associated matrix) such that $\udiffeo(\bx)=\rJ^{\top}(\bx) =\bx^\natural$ (here $\bx\in\Pi_\dx$). Let $\bA^\natural$ be the magnetic potential associated with $\pot$ in variables $\bx^\natural$. Let $\bA^\natural_\bfz$ and $\pot_\bfz$ be their linear parts at $\bfz$. There holds, cf.\ Remark \ref{rem:geneig}
\begin{equation}
\label{E:lienbapot}
   \bA^\natural(\bx^\natural) = \rJ^\top\big(\pot(\bx)\big)
   \quad\mbox{and}\quad 
   \bA^\natural_\bfz(\bx^\natural) = \rJ^\top(\pot_\bfz(\bx)),
   \quad \forall\bx\in\cV_{\bx_0}\,.
\end{equation}
Lemma \ref{L:d2ell0} 
in variables $(u_1,u_2,u_3)=(y,z_1,z_2)$ with $\ell=1$ then gives us a function $F$ such that $\partial^2_{y}(\bA^\natural-\nabla F)(\bfz) = 0$ leading to the estimates 
\begin{equation}
\label{eq:d2y0est}
   \big|\big(\bA^\natural - \bA^\natural_\bfz - 
   \nabla F\big)(\bx^\natural)\big| 
   \le C(\cV_{\bx_0})\,\|\pot\|_{W^{3,\infty}(\cV_{\bx_{0}})}
   \big(|y|^3 + |y||\bz| + |\bz|^2\big)\,.
\end{equation}
 After \eqref{eq:psidx}, we define 
\begin{equation}
\label{eq:spm}
\spm_{h}^{\natural}(\bx^\natural):=
   \re^{-iF(\bx^\natural)/h}\psi_{h}^\natural(\bx^\natural) \quad \mbox{for} \quad 
   \bx^\natural\in\R\times\Upsilon,
\end{equation}
which leads to our new quasimode given by
\begin{equation}
\label{eq:qmX}
   \varphi^{\dx}_h(\bx) = \spm_{h}^{\dx}(\bx):= \spm_{h}^{\natural}(\bx^\natural)
\quad \mbox{for} \quad    \bx\in\Pi_{\bx_0}=\Pi_{\dx}.
\end{equation}
We have obviously $\|\spm_{h}^{\dx}\|=\|\psi_{h}^{\natural}\|$.
Formulas \eqref{E:chgG} and \eqref{eq:diffAA'} then yield: 
\begin{align}\label{eq:diffAAdec}
q_{h}[\pot,\Pi_{\dx}](\spm_{h}^{\dx}) 
&= q_{h} [\bA^\natural,\R\times\!\Upsilon](\spm_{h}^{\natural}) \\ \nonumber
&= q_{h} [\bA^\natural-\nabla F,\R\times\!\Upsilon](\psi_{h}^{\natural}) \\ \nonumber
&= q_{h}[\bA^\natural_{\bfz},\R\times\!\Upsilon](\psi_{h}^{\natural}) 
+ \big\|(\bA^\natural-\bA^\natural_{\bfz}-\nabla F)\psi_{h}^{\natural}\big\|^2
\\ \nonumber
   & +2  \Re\int_{\R\times\Upsilon} (-ih\nabla+\bA^\natural_{\bfz})\psi_{h}^{\natural}(\bx^\natural)\cdot
   \big(\bA^\natural-\bA^\natural_{\bfz}-\nabla F\big)
   \overline{\psi_{h}^{\natural}}(\bx^\natural)\,\rd\bx^\natural .
\end{align}
By analogy with \eqref{eq.defah}, we define
\begin{equation}
\label{eq:hata}
   \hat a_{h}=
   \frac{\|(\bA^\natural -\bA^\natural_{\bfz} - \nabla F)\psi_{h}^{\natural}\|}
   {\|\psi_{h}^{\natural}\|}\, .
\end{equation}
Combining \eqref{eq:diffAAdec} with \eqref{eq:tronc1}, we obtain
\begin{equation}\label{eq:troncA22}
  \frac{q_{h}[\pot,\Pi_{\dx}](\spm_{h}^{\dx})}{\|\spm_{h}^{\dx}\|^2} 
  \leq h \Lambda_\bfz + h^2\rho_{h} +
  2 \sqrt h \sqrt{\Lambda_\bfz+h\rho_{h}}\ \hat a_{h} + \hat a_{h}^2
\end{equation}
In comparison with \eqref{eq:troncA2}, we have $\hat a_{h}$ instead of $a_h$. It remains to bound $\hat a_h$ from above. 

The following Lemma provides an improvement when compared to \eqref{eq.ah2delta}, due to estimates \eqref{eq:d2y0est} which replace \eqref{E:taylorA}. 

\begin{lemma}\label{lem:ahsuper}
With the previous notation, there exist constants $C(\Omega)>0$ and $h_{0}>0$ such that for all $h\in (0,h_{0})$
\begin{equation}\label{eq:ah}
   \hat a_{h} = 
  \frac{\|(\bA^\natural -\bA^\natural_{\bfz} - \nabla F)\psi_{h}^{\natural}\|}
   {\|\psi_{h}^{\natural}\|}
   \leq C(\Omega)\|\pot\|_{W^{3,\infty}(\cV_{\bx_0})}
   h^{\min(1,\frac12+\delta,3\delta)} .
\end{equation}
\end{lemma}

\begin{proof}
 Using the form of the admissible generalized eigenvector $\Psi^{\natural}$:
$$
   \Psi^{\natural}(\bx^\natural)=
   \re^{\ri\vartheta(\bx^\natural)}\Phi(\bz)
   \quad\mbox{with}\quad
   \bx^\natural=(y,\bz)\, ,
$$ 
we obtain by definition of $\psi^{\natural}_h$ \eqref{eq:psidx}
\[
   |\psi^{\natural}_h(\bx^\natural)| = 
   \underline\chi_{R}\left(\frac{|\bx^\natural|}{h^{\delta}}\right)  
   \Big|\Phi\left(\frac{\bz}{h^{1/2}}\right)\Big|\,.
\] 
Then 
relying on  \eqref{eq:d2y0est} and using
the changes of variables $\bZ=\bz h^{-1/2}$ and $Y=yh^{-\delta}$, we find the bounds 
\begin{eqnarray*}
   \Big\||y|^3 \ \underline\chi_{R}\left(\frac{|\bx^\natural|}{h^{\delta}}\right)  
   \Phi\left(\frac{\bz}{h^{1/2}}\right)\Big\| &\le& h^{3\delta} \,
   \|\psi^{\natural}_h\|\\
   \Big\||y|\,|\bz| \ \underline\chi_{R}\left(\frac{|\bx^\natural|}{h^{\delta}}\right)  
   \Phi\left(\frac{\bz}{h^{1/2}}\right)\Big\| &\le& h^{\delta+\frac12} \,
   \|\psi^{\natural}_h\| \\
   \Big\||\bz|^2 \ \underline\chi_{R}\left(\frac{|\bx^\natural|}{h^{\delta}}\right)  
   \Phi\left(\frac{\bz}{h^{1/2}}\right)\Big\| &\le& h  \,
   \|\psi^{\natural}_h\|.
\end{eqnarray*}
Summing up the latter three estimates leads to the lemma.
\end{proof}

We now take the same arguments as in Section \ref{SS:linear} but instead of \eqref{eq.ah54} we use Lemma \ref{lem:ahsuper} to estimate $\hat a_{h}$ and \eqref{eq:troncA22} becomes 
\[
   \frac{q_{h}[\pot,\Pi_{\dx}](\spm_{h}^{\dx})}{\|\spm_{h}^{\dx}\|^2} 
   \leq h\Lambda_\bfz
   +C(\Omega)(1+\| \bA \|_{W^{3,\infty}(\Omega)}^2)
   (h^{3/2}+h^{2-2\delta}+ h^{1+\delta}+ h^{3\delta+\frac12}+h^{6\delta}) \, . 
\]
Let ${f}_{h}$ be the quasimode defined in the same way as in \eqref{D:qmsuromega}. The same arguments as in Section \ref{SS:qmfinal} and Lemma \ref{L:Elip} combined with \eqref{E:QRspm} yields
\begin{equation}\label{eq:Ray4}
  \frac{q_{h}[\bA,\Omega]({f}_{h})}{\|{f}_{h}\|^2} \leq 
  h\Lambda_{\bfz}+C(\Omega)(1+\| \bA \|_{W^{3,\infty}(\Omega)}^2)
  (h^{2-2\delta}+h^{3\delta+\frac12}+h^{6\delta}+h^{1+\delta}) \, .
  \end{equation}
We optimize this upper bound by taking $\delta=\frac13$. The min-max principle provides Theorem \ref{T:sUB} in the case (G2) if $\dx=(\bx_0)$.

-- {\em Sliding modes.} This is the case when $\dx=(\bx_0,\bx_1)$. Let us explain now how the above arguments adapt to sliding quasimodes. We recall that we have introduced a vector $\dir$ in \eqref{def:tau} depending on $\Pi_{\bx_0}$ and $\Pi_\dx$.
Mimicking definition \eqref{eq.defQMbis} for quasimodes $\varphi_{h}^{\dx}$, we are going to construct new quasimodes with two adaptations: first, we linearize the potential at the point $\dec:=h^{\delta}\dir$, second, we shift the phase to optimize remainders. 

Let us denote by $\pot^{\dec}$ the potential transformed from $\pot$ by the translation $\diffeoT^\dec:\bx\to\bx-\dec$ 
\[
   \pot^\dec(\bx) = \pot(\bx+\dec) - \pot(\dec), \quad \bx\in\Pi_\dx
\]
(compare with \eqref{E:ABtilde}) and by $\pot^\dec_\bfz$ the linear part of $\pot^\dec$ at $\bfz$: 
$$
   \pot^{\dec}_{\bfz}(\bx)=\nabla\pot(\dec)\cdot\bx, \quad \bx\in\Pi_\dx.
$$
We have $\curl \pot^{\dec}_{\bfz}=\tbB_{\dec}$ where the constant $\tbB_{\dec}$ is the magnetic field $\tbB$ frozen at $\dec$. 

We have $\En(\bB_{\bx_{0}},\Pi_{\dx})< \seE(\bB_{\bx_{0}},\Pi_{\dx})$. Due to Lemma~\ref{L:Elip}, we have 
\begin{equation}
\label{E:remarkimportante}
 \exists \varepsilon>0, \ \ \forall \dec\in \cB(0,\varepsilon)\cap \overline\Pi_{\bx_{0}}, \quad \En(\tbB_{\dec},\Pi_{\dx})<\seE(\tbB_{\dec},\Pi_{\dx}) \,,
 \end{equation}
and $(\tbB_{\dec},\Pi_{\dx})$ is still in situation (G2). 
There exists an admissible generalized eigenfunction $\Psi^{\dx,\dec}$ for the operator $H(\pot^{\dec}_{\bfz},\Pi_{\dx})$ associated with $\En(\tbB_{\dec},\Pi_{\dx})$ (denoted by $\Lambda_{\dec}$ for shortness):
\begin{equation}\label{D:psidec}
    \Psi^{\natural,\dec}(\bx^\natural) = 
    \Psi^{\dx,\dec}\circ (\udiffeo^{\dec})^{-1}(y,\bz)=
    \re^{\ri\vartheta^{\dec}(\by,\bz)}\Phi^{\dec}(\bz)\quad\mbox{for }\quad 
    (y,\bz)\in\R\times\Upsilon,
\end{equation}
where the rotation\footnote{The rotation $\udiffeo^{\dec}$ may depend on $\dec$ but that does not hamper our analysis.} $\udiffeo^{\dec}$ maps $\Pi_{\dx}$ onto $\R\times \Upsilon$. Here $\Upsilon$ is a cone in $2$ dimensions (namely a sector, a half-space or $\R^2$). We have
\begin{equation}
\label{eq:genEPdec}
\begin{cases}
   (-i\nabla+\pot^{\dec}_{\bfz})^2\,\Psi^{\dx,\dec}=\Lambda_{\dec}\Psi^{\dx,\dec} 
   &\mbox{in } \Pi_{\dx},\\
   (-i\partial_n+\mathbf{n}\cdot\pot^{\dec}_{\bfz})\,\Psi^{\dx,\dec}=0 
   &\mbox{on } \partial\Pi_{\dx}.
\end{cases}
\end{equation}

An important point is that, choosing $\varepsilon>0$ small enough, we may assume that, in virtue of Lemma \ref{L:Elip}, the functions $\Phi^\dec$ are uniformly exponentially decreasing
\begin{equation}
\label{E:agmuniform}
\exists c>0,\ C>0,\quad
\forall \dec\in \cB(\bfz,\varepsilon), \quad 
\|\Phi^{\dec}\re^{c|\bz|}\|_{L^2(\Upsilon)}\leq C \|\Phi^{\dec} \|_{L^2(\Upsilon)}  \, .
\end{equation}
Now we take $\dec= h^{\delta} \dir$ with $h$ small enough. Using the $\dec$-dependent variables $\bx^\natural=(y,\bz)=\udiffeo^\dec(\bx)$, we set $\Psi_{h}^{\natural,\dec}(\bx^{\natural}):=\Psi^{\natural,\dec}(\frac{\bx^{\natural}}{\sqrt h})$ and, 
cf.\ \eqref{eq:psidx}
\begin{equation}
\label{eq:psidxdec}
    \psi_{h}^{\natural,\dec}(\bx^{\natural}) =  
   \chi_h(\bx^{\natural})\, \Psi_{h}^{\natural,\dec}(\bx^{\natural}) = 
   \underline\chi_{R}\left(\frac{|\bx^{\natural}|}{h^{\delta}}\right)  
   \Psi^{\natural,\dec}\left(\frac{\bx^{\natural}}{h^{1/2}}\right),
   \quad\bx^{\natural}\in\R\times\Upsilon \, .  
\end{equation}
We are arrived at point where the situation is similar as in the sitting case, with the new feature that the generalized eigenvectors $\Psi_{h}^{\natural,\dec}$ depend (in some smooth way) on the parameter $\dec$. 
The potential $\bA^{\natural,\dec}$ in natural variables corresponding to $\pot^\dec$ and its linear part at $\bfz$ satisfy
\begin{equation}
   \bA^{\natural,\dec}(\bx^\natural) = \rJ^\top\big(\pot^\dec(\bx)\big) 
   \quad\mbox{and}\quad
   \bA^{\natural,\dec}_\bfz(\bx^\natural) = \rJ^\top(\pot^\dec_\bfz(\bx)) \,.
\end{equation}
Like before we find a function $F^\dec$ satisfying  
\begin{equation}
\label{eq:d2y0estsliding}
   \big|\big(\bA^{\natural,\dec} - \bA^{\natural,\dec}_\bfz - 
   \nabla F^\dec\big)(\bx^\natural)\big| 
   \le C(\cV_{\bx_0})\,\|\pot\|_{W^{3,\infty}(\cV_{\bx_{0}})}
   \big(|y|^3 + |y||\bz| + |\bz|^2\big)\,.
\end{equation}
We define the new functions
\begin{equation}
\label{eq:spmp}
   \spm_{h}^{\natural,\dec}(\bx^\natural):=
   \re^{-iF^\dec(\bx^\natural)/h}\psi_{h}^{\natural,\dec}(\bx^\natural) \quad \mbox{for} \quad 
   \bx^\natural\in\R\times\Upsilon\,,
\end{equation}
and
$$   \spm_{h}^{\dx,\dec}(\bx)=\spm_{h}^{\natural,\dec}(\bx^\natural)\quad \mbox{for} \quad 
   \bx\in\Pi_{\dx} \:.$$
Still with $\dec=h^{\delta}\dir$, we define the quasimode $\varphi_{h}^{\dx}$ in $\Pi_{\bx_0}$ by (cf. \eqref{eq.defQMb})
\begin{equation}
\label{D:psidirdec}
   \varphi_{h}^{\dx}(\bx) = 
   \re^{-i\langle\pot(\dec) ,\, {\bx}\rangle/h} \,
   \spm_{h}^{\dx,\dec}(\bx-\dec),
   \quad\bx\in\Pi_{\bx_{0}} \, .  
\end{equation}
Using formulas \eqref{E:chgG} and \eqref{eq:diffAA'}, we have:
\begin{equation*}
q_{h}[\pot,\Pi_{\bx_{0}}](\varphi_{h}^{\dx}) 
=  q_{h} [\pot^\dec,\Pi_{\dx}](\spm_{h}^{\dx,\dec})=q_{h} [\bA^{\natural,\dec},\R\times\!\Upsilon](\spm_{h}^{\natural,\dec}) \,.
\end{equation*}
Relations~\eqref{eq:diffAAdec} are still valid if we replace $\psi^\natural$ by $\psi^{\natural,\dec}$ and $\bA^\natural$ by $\bA^{\natural,\dec}$. 
Recall that we have denoted $\Lambda_{\dec}$ for $\En(\tbB_{\dec},\Pi_{\dx})$. Using definition \eqref{eq:psidxdec} and  Lemma \ref{lem:tronc} we get like for \eqref{eq:tronc1}
\begin{equation}
\label{eq:tronc1p}
   \frac{q_{h}[\bA^{\natural,\dec}_{\bfz},\R\times\Upsilon](\psi_{h}^{\natural,\dec})}{\|\psi_{h}^{\natural,\dec}\|^2}
   = h\Lambda_{\dec} + h^2\rho_{h}
\quad\mbox{ with }\quad 
  \rho_{h}=\frac{\| \,|\nabla\chi_{h}|\, \Psi_{h}^{\natural,\dec}\|^2}
  {\|\chi_{h}\Psi_{h}^{\natural,\dec}\|^2}.
\end{equation}
Here $\rho_h$ satisfies the estimates as given in Lemma \ref{lem:rhohb} because of the uniformly exponential decay \eqref{E:agmuniform}. Now $\hat a_h$ takes the form
$$
\hat a_{h}=
   \frac{\|(\bA^{\natural,\dec}-\bA^{\natural,\dec}_\bfz-\nabla F^\dec)\psi_{h}^{\natural,\dec}\|}
   {\|\psi_{h}^{\natural,\dec}\|}\, ,
$$ 
and we obtain
\begin{equation}\label{eq:troncA22p}
  \frac{q_{h}[\pot^\dec,\Pi_{\dx}](\spm_{h}^{\dx,\dec})}{\|\spm_{h}^{\dx,\dec}\|^2} 
  \leq h \Lambda_\dec + h^2\rho_{h} +
  2 \sqrt h \sqrt{\Lambda_\dec+h\rho_{h}}\ \hat a_{h} + \hat a_{h}^2
\end{equation}
Now, when comparing with \eqref{eq:troncA22}, we have $\Lambda_\dec$ instead of $\Lambda_\bfz$.
Using the uniform exponential decay \eqref{E:agmuniform}, we find that Lemma \ref{lem:ahsuper} holds uniformly with respect to $\dec$ and we find
\[
  \frac{q_{h}[\pot^\dec,\Pi_{\dx}](\spm_{h}^{\dx,\dec})}{\|\spm_{h}^{\dx,\dec}\|^2} 
   \leq h\Lambda_\dec
   +C(\Omega)(1+\| \bA \|_{W^{3,\infty}(\Omega)}^2)
   (h^{3/2}+h^{2-2\delta}+ h^{1+\delta}+ h^{3\delta+\frac12}+h^{6\delta}) \, . 
\]

Now we use Lemma~\ref{L:Elip} to obtain the estimate $|\Lambda_\dec-\Lambda_{\bfz}|\le C|\dec|\le Ch^\delta$  (as previously $\Lambda_{\bfz}$ denotes $\En(\bB_{\bx_{0}},\Pi_{\bx_{0}})$). Hence
 \begin{multline}
 \label{E:QRspm}
   \frac{q_{h}[\pot,\Pi_{\bx_{0}}](\varphi_{h}^{\dx}) }{\|\varphi_{h}^{\dx}\|^2} =
   \frac{q_{h}[\pot^\dec,\Pi_{\dx}](\spm_{h}^{\dx,\dec})}{\|\spm_{h}^{\dx,\dec}\|^2} 
   \leq \\ h\Lambda_{\bfz}
   +C(\Omega)(1+\| \bA \|_{W^{3,\infty}(\Omega)}^2)
   (h^{3/2}+h^{2-2\delta}+ h^{1+\delta}+ h^{3\delta+\frac12}+h^{6\delta}) \, . 
\end{multline}

We end the proof as in the sitting case and Theorem \ref{T:sUB} is proved in the case (G2).

\paragraph{Assume we are in situation (G1)} 
In situation (G1) the generalized eigenfunction has exponential decay in one variable $z$. The upper bound \eqref{eq:troncA2} obtained by a Cauchy-Schwarz inequality is too rough and we will deal with the previous identity \eqref{eq:troncA}. A Feynman-Hellmann formula will simplify the cross term in \eqref{eq:troncA} and the exponential decay in one variable will provide the desired result.

In situation (G1) $\Pi_{\dx}$ is a half-plane and $\bB_{\bx_{0}}$ is tangent to its boundary. 
Denote by $(\by,z)=(y_1,y_2,z)\in \R^2\times \R_{+}$ a system of coordinates of $\Pi_{\dx}$ such that $\bB_{\bx_{0}}$ is tangent to the $y_{2}$-axis. In these coordinates, the magnetic field $\bB_{\bx_{0}}$ writes $(0,b,0)$.

In the rest of this proof, we will assume without restriction that $b=1$. Indeed, once quasimodes are constructed for $b=1$, Lemmas \ref{lem.dilatation} and \ref{lem:sense} allow to convert them into quasimodes for any $b$.

Let us define the canonical reference potential
\begin{equation}\label{eq.uA}
\uA(\by,z)=(z,0,0).
\end{equation}
such that $\curl \uA=(0,1,0)$. 
We know (see Section \ref{s:age}) that the function
\begin{equation}\label{eq.uPsi}
 \uPsi_{h}(\by,z):=
 \re^{-i\sqrt{\Theta_{0}}\,y_{1}/\sqrt{h}} \,  \Phi\Big(\frac{z}{\sqrt{h}}\Big)
 \end{equation}
is a generalized eigenvector of $H_{h}(\uA,\R^2\times\R_{+})$ where $\Phi$ is a normalized eigenvector associated with the first eigenvalue of the de Gennes operator $-\partial_{z}^2+(z-\sqrt{\Theta_{0}})^2$ on $\R_+$. 
 
 We define $\upsi_{h}$ in the same spirit as \eqref{eq:psidx} but for convenience we take a cut-off function in tensor product form (here for simplicity we denote $\underline\chi_{R}$ by $\chi$) 
\begin{equation}\label{eq.defQM2nat}
  \upsi_{h}(\by,z)
 :=\chi\Big(\frac{|\by|}{h^\delta}\Big)\,\chi\Big(\frac{z}{h^\delta}\Big)\,\uPsi_{h}(\by,z),\qquad \forall(\by,z)\in\R^2\times\R_{+}.
\end{equation}

Let $\udiffeo:\Pi_{\dx}\mapsto \R^2\times \R_{+}$ be the rotation associated with the coordinates $(\by,\bz) = \bx^{\natural}$ and $\rJ$ be its associated matrix. The magnetic potential $\bA^{\natural}$ and its linear part still satisfy \eqref{E:lienbapot}. 
Since $\bA^\natural_{\bfz}$ and  the canonical reference potential $\uA$ introduced in \eqref{eq.uA} are both linear with $\curl\bA^\natural_{\bfz}=\curl\uA$,  
there exists a homogenous polynomial function of degree two $F^\natural$ such that 
\begin{equation}\label{eq.uAnat}
\bA^\natural_{\bfz}-\nabla_{\natural} F^\natural= \uA.
\end{equation} 
Therefore, $\re^{-iF^\natural/h}\uPsi_{h}$ is an admissible generalized eigenvector for $\OP_{h}(\bA^\natural_{\bfz},\R^3_{+})$. So we define 
\begin{equation}
\label{D:spmnatG1}
 \psi_{h}^\natural(\by,z):=\re^{-iF^\natural(\bx^\natural)/h} \upsi_{h}(\by,z),
\qquad\forall(\by,z)\in\R^2\times\R_{+},
\end{equation}
and
\begin{equation}
\label{D:spmG1}
\psi_{h}^\dx(\bx):=\psi_{h}^\natural(\by,z),\qquad \forall \bx\in\Pi_{\dx}.
\end{equation}
According as $\Pi_{\bx_{0}}$ equals $\Pi_{\dx}$ or not we now construct adjusted quasimodes using the generalized eigenvector $\psi_{h}^\dx$.

-- {\em Sitting quasimodes.} This is the case when $\dx=(\bx_0)$. 
The quasi-mode is defined by 
$$\varphi_{h}^\dx(\bx):=\psi_{h}^\dx(\bx)\, ,$$ 
where $\psi_{h}^{\dx}(\bx)$ is set in \eqref{D:spmnatG1}-\eqref{D:spmG1}. This gives a quasimode for $\OP_{h}(\pot,\Pi_{\bx_{0}})$ and we have 
\begin{eqnarray}
\label{eq:G1qm}
  q_{h}[\pot,\Pi_{\bx_{0}}](\psi_{h}^{\dx}) 
  &=&   q_{h}[\bA^\natural,\R^3_{+}](\psi_{h}^{\natural}) \\
  &=&   q_{h}[\bA^\natural-\nabla F^\natural,\R^3_{+}](\upsi_{h}). \nonumber 
\end{eqnarray}
Now we  apply \eqref{eq:diffAA'} with $\bA=\bA^\natural-\nabla F^\natural$ and $\bA'=\uA$.  Using \eqref{eq.uAnat} we find $\bA-\bA'=\bA^\natural-\bA^\natural_{\bfz}$, and thus
\begin{align} 
\label{eq:G1refest}
  q_{h}[\bA^\natural-\nabla F^\natural,\R^3_{+}](\upsi_{h}) 
  &=   q_{h}[\uA,\R^3_{+}](\upsi_{h})\\
  & \quad +2  \Re\int_{\R^3_{+}} (-ih\nabla+\uA)\upsi_{h}(\bx^\natural)\cdot 
  (\bA^\natural-\bA^\natural_{\bfz})(\bx^\natural)\,\overline{\upsi_{h}(\bx^\natural})\,
  \rd\bx^\natural \label{E:termcrois}\\
  & \quad +\|(\bA^\natural-\bA^\natural_{\bfz})\upsi_{h}\|^2.
  \label{eq.RayleighG1}
\end{align}
We bound from above the term \eqref{eq.RayleighG1} like in \eqref{eq.ah2delta}:
\begin{equation}
\label{eq:G1carr}
   \|(\bA^\natural-\bA_{\bfz}^\natural)\upsi_{h}\|^2 \le
   C(\Omega)\|\bA^\natural\|^2_{W^{2,\infty}( \supp(\upsi_{h}))} \,
   h^{4\delta}\, \|\upsi_{h}\|^2 .
\end{equation} 

Let us now deal with the term \eqref{E:termcrois}. We calculate $(-ih\nabla+\uA)\upsi_{h}$ 
by using \eqref{eq.uPsi}--\eqref{eq.defQM2nat}:
\begin{multline*}
   (-ih\nabla+\uA)\upsi_{h}(\bx^\natural)
   = \re^{-i\sqrt{\Theta_{0}}\,y_{1}/\sqrt{h}} \,\times \\
   \left\{\chi\big(\tfrac{|\by|}{h^{\delta}}\big)\ 
   \chi\big(\tfrac{z}{h^{\delta}}\big)
   \begin{bmatrix}
   (z-\sqrt{h\Theta_{0}}\ee)\,\Phi\big(\tfrac{z}{\sqrt{h}}\big) \\[5pt]
   0\\[5pt]
   -i\sqrt h \,\Phi'\big(\tfrac{z}{\sqrt{h}}\big) 
   \end{bmatrix}
   -ih^{1-\delta}  
\begin{bmatrix}
   \frac{y_{1}}{|\by|}\chi'(\tfrac{|\by|}{h^{\delta}})\ 
\chi(\tfrac{z}{h^{\delta}})\\[5pt]
   \frac{y_{2}}{|\by|} \chi'(\tfrac{|\by|}{h^{\delta}})\ 
\chi(\tfrac{z}{h^{\delta}})\\[5pt]
   \chi(\tfrac{|\by|}{h^{\delta}})\ \chi'(\tfrac{z}{h^{\delta}})
\end{bmatrix}
   \Phi\big(\tfrac{z}{\sqrt{h}}\big) \right\}.
\end{multline*}
Since $\Phi$ and $\chi$ are real valued functions, the term \eqref{E:termcrois} reduces to a single term: 
\begin{align}\label{eq.termecrois}
\Re\int_{\R^3_{+}} (-ih\nabla&+\uA)\upsi_{h}(\bx^\natural)\cdot (\bA^\natural-\bA^\natural_{\bfz})(\bx^\natural)\overline{\upsi_{h}(\bx^\natural})\,\rd\bx^\natural \\
   &=\int_{\R^3_{+}} (z-\sqrt{h\Theta_{0}}\ee)\ 
   |\upsi_{h}(\bx^\natural)|^2 A^{(\mathrm{rem},2)}_{2}(\bx^\natural) \, \rd\bx^\natural
   \nonumber\\
   &=\int_{\R^3_{+}}(z-\sqrt{h\Theta_{0}}\ee)\,
   \left|\Phi\big(\tfrac{z}{\sqrt{h}}\big)\right|^2
   \ \left|\chi\big(\tfrac{|\by|}{h^{\delta}}\big)\right|^2\ 
   \left|\chi\big(\tfrac{z}{h^{\delta}}\big)\right|^2
   A^{(\mathrm{rem},2)}_{2}(\bx^\natural)\,\rd\bx^\natural,\nonumber
\end{align}
where $A^{(\mathrm{rem},2)}_{2}$ denotes the second component of $\bA^\natural-\bA_{\bfz}^\natural$.
We write 
$$
   A^{(\mathrm{rem},2)}_{2}(\bx^\natural)=
   P_{2}(\by)+R_{2}(\bx^\natural)+A^{(\mathrm{rem},3)}_{2}(\bx^\natural),
$$
where $A^{(\mathrm{rem},3)}_{2}$ is the Taylor remainder of degree $3$ of the second component of $\bA^\natural$ at $\bfz$, whereas $P_{2}(\by)+R_{2}(\bx^\natural)$ is a representation of its quadratic part in the form
$$
   P_{2}(\by)=a_{1}y_{1}^2+a_{2}y_{2}^2+a_{3}y_{1}y_{2}\quad\mbox{and}\quad
   R_{2}(\bx^\natural)=b_{1} z^2+b_{2}z y_{1}+b_{3} zy_{2}.
$$
As in \eqref{eq:rem3} there holds
$$
   \|A^{(\mathrm{rem},3)}_{2}\|_{L^\infty(\supp(\upsi_{h}))}\leq 
   C\|\pot\|_{W^{3,\infty}( \supp(\upsi_{h}))} \,h^{3\delta},
$$
leading to, with the help of the variable change $Z=z/\sqrt h$ and the exponential decay of $\Phi$:
\begin{equation}
\label{eq:G1A3}
   \left|\int_{\R^3_{+}} (z-\sqrt{h\Theta_{0}}\ee)\,|\upsi_{h}(\bx^\natural)|^2\, 
    A^{(\mathrm{rem},3)}_{2}(\bx^\natural)\,\rd\bx^\natural\right|
   \leq Ch^{\frac12+3\delta} \|\upsi_{h}\|^2 \ .
\end{equation}
Likewise, combining the exponential decay of $\Phi$, the change of variable $Z=z/\sqrt h$ and the localization of the support in balls of size $C h^\delta$, we deduce
\begin{equation}
\label{eq.G1majo1}
   \left| \int_{\R^3_{+}} (z-\sqrt{h\Theta_{0}}\ee)\,|\upsi_{h}(\bx^\natural)|^2\, 
    R_{2}(\bx^\natural)\,\rd\bx^\natural\right|
   \leq Ch^{\min(\frac32,1+\delta)} \|\upsi_{h}\|^2 \ .
\end{equation}
Let us now deal with the term involving $\by\mapsto P_{2}(\by)$. 
Due to a Feynman-Hellmann formula applied to the de Gennes operator $\DG\tau$ at $\tau= -\sqrt{\Theta_0}$ (cf.\ \cite[Lemma A.1]{HeMo01}) we find by the scaling $z\mapsto z/\sqrt{h}$ the identity
$$
   \int_{\R_{+}}(z-\sqrt{h\Theta_{0}}\ee)\,\left|\Phi\big(\tfrac{z}{\sqrt{h}}\big)\right|^2
   \, \rd z = 0 \,. 
$$
Thus we can write 
\begin{align*}
   \int_{\R^3_{+}} (z-\sqrt{h\Theta_{0}}\ee)\, &|\upsi_{h}(\bx^\natural)|^2\, 
   P_{2}(\by) \,\rd\bx^\natural
\\
   &= \int_{\R^2} P_{2}(\by)\, 
   \left|\chi\big(\tfrac{|\by|}{h^{\delta}}\big)\right|^2 \!\rd\by \ \, 
   \int_{z\in\R_{+}} (z-\sqrt{h\Theta_{0}}\ee)\, 
   \left|\Phi\big(\tfrac{z}{\sqrt{h}}\big)\right|^2 \chi\big(\tfrac{z}{h^{\delta}}\big)^2
   \,\rd z
\\
   &= \int_{\R^2} P_{2}(\by)\, 
   \left|\chi\big(\tfrac{|\by|}{h^{\delta}}\big)\right|^2 \!\rd\by \ \, 
   \int_{z\in\R_{+}} (z-\sqrt{h\Theta_{0}}\ee)\, 
   \left|\Phi\big(\tfrac{z}{\sqrt{h}}\big)\right|^2 
   \left(\chi\big(\tfrac{z}{h^{\delta}}\big)^2 - 1\right)
   \,\rd z
\end{align*}
The support of the integral in $z$ is contained in $z\ge Rh^\delta$ with $\delta<\frac12$.
Therefore, using once more the changes of variables $\mathbf Y=\by/h^\delta$ and $Z=z/\sqrt{h}$, we find:
$$
   \left| \int_{\R^3_{+}} (z-\sqrt{h\Theta_{0}}\ee)\, \ |\upsi_{h}(\bx^\natural)|^2 
   P_{2}(\by) \,\rd\bx^\natural \right|
  \leq C \|\bA^\natural\|_{W^{2,\infty}( \supp(\upsi_{h}))} 
  h^{4\delta+\frac12}\re^{-ch^{\delta-1/2}}.
$$
Since $\|\upsi_{h}\|^2\geq C h^{2\delta+\frac 12}$ (see \eqref{eq.minpsih}), 
this leads to: 
\begin{equation}\label{eq.G1majo2}
   \left| \int_{\R^3_{+}} (z-\sqrt{h\Theta_{0}}\ee)\, \ |\upsi_{h}(\bx^\natural)|^2 
   P_{2}(\by) \,\rd\bx^\natural \right|
  \leq C \|\bA^\natural\|_{W^{2,\infty}( \supp(\upsi_{h}))} 
  \re^{-ch^{\delta-1/2}}\,\|\upsi_{h}\|^2\ . 
\end{equation}
Collecting \eqref{eq:G1A3}, \eqref{eq.G1majo1}, and \eqref{eq.G1majo2} in \eqref{E:termcrois}, we find the upper bound 
\begin{multline}
\label{eq:G1crois}
   \left|\Re\int_{\R^3_{+}} (-ih\nabla +\bA_{\bfz}^\natural)\upsi_{h}(\bx^\natural) 
   \cdot (\bA^\natural-\bA_{\bfz}^\natural)\overline{\upsi_{h}(\bx^\natural)}
   \,\rd\bx^\natural\right|\\
   \leq C(1+ \|\bA^\natural\|_{W^{3,\infty}( \supp(\upsi_{h}))}^2 )\,
   h^{\min(\frac 12+3\delta,1+\delta)} 
   \|\upsi_{h}\|^2 \ . 
\end{multline}
Returning to \eqref{eq:G1qm} via \eqref{eq:G1refest} and combining \eqref{eq:G1crois} with \eqref{eq:G1carr}, we deduce
$$
  \frac{q_{h}[\pot,\Pi_{\bx_{0}}](\psi_{h}^{\dx})}{\|\psi_{h}^{\dx}\|^2} 
  \leq \frac{q_{h}[\uA,\R^3_{+}](\upsi_{h})}{\|\upsi_{h}\|^2} + 
  C(1+ \|\bA^\natural\|_{W^{3,\infty}( \supp(\upsi_{h}))}^2 )\,
  (h^{\min(1+\delta,\frac 12+3\delta)}+h^{4\delta}).
$$
Inserting the cut-off error for $q_{h}[\uA,\R^3_{+}](\upsi_{h})$ we obtain
$$
  \frac{q_{h}[\pot,\Pi_{\bx_{0}}](\psi_{h}^{\dx})}{\|\psi_{h}^{\dx}\|^2} 
  \leq h\En(\bB_{\bx_{0}},\Pi_{\bx_{0}}) + 
  C(1+ \|\pot\|_{W^{3,\infty}( \supp(\psi_{h}^{\dx}))}^2)\,
  (h^{2-2\delta}+h^{\min(1+\delta,\frac 12+3\delta)}+h^{4\delta}).
$$
As already seen,  we have $\|\pot\|_{W^{3,\infty}(\supp(\psi_{h}^{\dx}))} \leq (1+C(\Omega)h^{\delta})\|\bA\|_{W^{3,\infty}(\Omega)}$. The quasimode $f_{h}$ on $\Omega$ is defined as in \eqref{D:qmsuromega} and the estimation of Section \ref{SS:qmfinal} provides
\begin{equation}\label{eq:Ray45}
  \frac{q_{h}[\bA,\Omega](f_{h})}{\|f_{h}\|^2}
   \leq h\En(\bB_{\bx_{0}},\Pi_{\bx_{0}})+C(\Omega)(1+\| \bA \|_{W^{3,\infty}(\Omega)}^2)
  (h^{2-2\delta}+h^{3\delta+\frac12}+h^{4\delta}+h^{1+\delta}) \, .
\end{equation}
Choosing $\delta=\frac{1}{3}$ and using the min-max principle we get Theorem \ref{T:sUB} in situation (G1).

-- {\em Sliding quasimodes.} This is the case when $\dx=(\bx_0,\bx_{1})$. 
Let $\dir$ be the vector introduced in \eqref{def:tau} and $\dir^\natural=\udiffeo\dir$. 
We note that $\dir^\natural$ has no component in the $z$ direction, because $\dir$  
is tangent to the boundary of the half-space $\Pi_{\dx}$, see Remark \ref{R:decritdir}. We define $\dec^\natural = h^\delta \dir^\natural$, which we write $\dec^\natural=(\dec_{\by},0)=(p_{1},p_{2},0)$
in coordinates $\bx^{\natural}$.

Our quasimode for $H_{h}(\pot,\Pi_{\bx_{0}})$ is defined by a $\dec$-translation of the  quasimode $\psi_{h}^\dx$, cf \eqref{eq.defQMb}:
\begin{equation}
\label{D:psidirdecG1}
   \varphi_{h}^{\dx}(\bx) = 
   \re^{-\ri\langle\pot_{\bfz}(\dec),\,{\bx}\rangle/h} \,\psi_{h}^\dx(\bx-\dec)
   \qquad\forall\bx\in\Pi_{\bx_{0}} \, .
\end{equation}
There holds the following sequence of identities, cf \eqref{eq:G1qm} for the last one,
\[
\begin{aligned}
   q_{h}[\pot,\Pi_{\bx_{0}}](\varphi_{h}^{\dx}) &=
   q_{h}[\pot(\cdot+\dec)-\pot_{\bfz}(\dec),\Pi_{\dx}](\psi_{h}^\dx) 
\\ &=
   q_{h}[\bA^{\natural}(\cdot+\dec^{\natural})-\bA_{\bfz}^{\natural}(\dec^{\natural}),
   \R^3_{+}](\psi_{h}^{\natural})
\\ &=
   q_{h}[\bA^{\natural}(\cdot+\dec^{\natural})-\bA_{\bfz}^{\natural}(\dec^{\natural}) 
   - \nabla F^\natural,
   \R^3_{+}](\upsi_{h})
\end{aligned}
\]
and we check that
\[
\begin{aligned}
   \bA^{\natural}(\cdot+\dec^{\natural})-\bA_{\bfz}^{\natural}(\dec^{\natural}) 
   - \nabla F^\natural &= 
   \bA^{\natural}(\cdot+\dec^{\natural}) - \bA_{\bfz}^{\natural}(\cdot+\dec^{\natural})
   + \bA_{\bfz}^{\natural}(\cdot+\dec^{\natural})
   -\bA_{\bfz}^{\natural}(\dec^{\natural}) 
   - \nabla F^\natural \\ &=
   \bA^{\natural}(\cdot+\dec^{\natural}) - \bA_{\bfz}^{\natural}(\cdot+\dec^{\natural})
   + \bA_{\bfz}^{\natural} - \nabla F^\natural \\ &=
   \bA^{\natural}(\cdot+\dec^{\natural}) - \bA_{\bfz}^{\natural}(\cdot+\dec^{\natural})
   + \uA \,.
\end{aligned}
\]
Then, instead of \eqref{eq:G1refest}-\eqref{eq.RayleighG1} we obtain now that $q_{h}[\bA^{\natural}(\cdot+\dec^{\natural})-\bA_{\bfz}^{\natural}(\dec^{\natural})
  -\nabla F^\natural,\R^3_{+}](\upsi_{h})$ is the sum of the three following terms:
\begin{align*}
  & q_{h}[\uA,\R^3_{+}](\upsi_{h})\\
  & \quad +2  \Re\int_{\R^3_{+}} (-ih\nabla+\uA)\upsi_{h}(\bx^\natural)\cdot 
  \big(\bA^{\natural}(\bx^\natural+\dec^{\natural}) - 
  \bA_{\bfz}^{\natural}(\bx^\natural+\dec^{\natural})\big)\,\overline{\upsi_{h}(\bx^\natural})\,
  \rd\bx^\natural \\
  & \quad +\|\big(\bA^{\natural}(\cdot+\dec^{\natural}) - 
  \bA_{\bfz}^{\natural}(\cdot+\dec^{\natural})\big)\upsi_{h}\|^2.
\end{align*}
Since $|\dec|=h^\delta$, the estimates \eqref{eq:G1carr}-\eqref{eq:Ray45} of the sitting case are still valid now, replacing the norm in $W^{\ell,\infty}(\supp(\upsi_h))$ by the norm in $W^{\ell,\infty}(\dec^\natural+\supp(\upsi_h))$ (for $\ell=2,3$).

The proof of Theorem~\ref{T:sUB} is now complete since we have explored all possible geometric situations for $(\bB_{\bx_{0}},\Pi_{\dx})$.

\subsection{Improvement for a straight polyhedron with constant magnetic field}
In this paragraph we improve Theorem \ref{T:generalUB} for a straight polyhedral domain with constant magnetic field. Since there is no curvature, we expect smaller remainders in the asymptotics of $\lambda_{h}(\bB,\Omega)$. Moreover, in that case, we will see that the function $\bx\mapsto E(\bB,\Pi_{\bx})$ attains its minimum at a vertex of $\Omega$.

\begin{theorem}
Let $\Omega$ be a straight polyhedron and $\bB$ be a constant magnetic field with $|\bB|=1$.
Then
$$ \sE(\bB \ee,\Omega)=\min_{\bv\in\gV}\En(\bB,\Pi_{\bv})$$
where $\gV$ denotes the set of the vertices of $\Omega$. We have
$$\lambda_h(\bB,\Omega) \leq h\sE(\bB \ee,\Omega) + C h^2 \ . $$
If there exists $\bv\in\gV$ such that $\En(\bB,\Pi_{\bv})=\sE(\bB \ee,\Omega)<\seE(\bB \ee,\Pi_{\bv})$, then there exist positive  constants $C,c$ such that
$$\lambda_h(\bB,\Omega) \leq h\sE(\bB \ee,\Omega) + C \re^{-c h^{-1/2}}.$$
\end{theorem}

\begin{proof}
Since the polyhedral domain is assumed to have straight faces and edges and the magnetic field is constant, the function $\bx\mapsto \En(\bB,\Pi_{\bx})$ is constant on each stratum of $\overline\Omega$. 
Let $\bv\in\gV$. We apply Theorem~\ref{th:dicho} and relations \eqref{eq:comp} and \eqref{eq:s*3simple} with $\Pi=\Pi_{\bv}$:
$$\En(\bB,\Pi_{\bv})\leq\seE(\bB,\Pi_{\bv})=\min_{\be\in\gE_{\bv}}\En(\bB,\Pi_{\be}),$$
with $\gE_{\bv}$ the subset of $\gE$ such that for any $\be\in\gE_{\bv}$, $\bv\in\partial\be$ and $\Pi_{\be}$ the wedge associated with the edge $\be$.
In the same way we prove for each edge $\be\in \gE$:
$$E(\bB,\Pi_{\be}) \leq \min_{{\bf f}\in \gF_{\be}}E(\bB,\Pi_{\bf f}) \leq 1 $$
where $\gF_{\be}$ denotes the set of the faces adjacent to an edge $\be$. Therefore  
$$\min_{\bv\in\gV}\En(\bB,\Pi_{\bv})=\sE(\bB,\Omega).$$
Let $\bv_{0}$ be a vertex minimizing $\bx\mapsto E(\bB,\Pi_{\bx})$. Let $\Pi_{\dx}$ be the tangent cone given by Theorem \ref{th:dicho}. If $E(\bB,\Pi_{\bv_{0}})<\sE^{*}(\bB,\Pi_{\bv_{0}})$ then $\Pi_{\dx}=\Pi_{\bv_{0}}$ and we have an (admissible generalized) eigenfunction on $\Pi_{\bv_{0}}$ associated with $E(\bB,\Pi_{\bv_{0}})$. If $E(\bB,\Pi_{\bv})=\sE^{*}(\bB,\Pi_{\bv})$, then there exists a stratum $\bt$ of $\Omega$ associated with $\Pi_{\dx}$ such that $\Pi_{\dx}$ is the tangent cone to any point of $\bt$. Moreover for any point $\bx\in \bt$ we have $E(\bB,\Pi_{\bx})<\sE^{*}(\bB,\Pi_{\bx})$ therefore there exists a generalized eigenfunction on $\Pi_{\bx}$ associated to $E(\bB,\Pi_{\bx})$. In both cases we have found a point $\bx\in \overline{\Omega}$ such that there exists a generalized eigenfunction on $\Pi_{\bx}$ associated to $\sE(\bB,\Omega)$.
There exists $R_{\bx}>0$ such that 
\begin{equation}\label{eq.BR}
\Omega\cap \cB(0,2R_{\bx}) = \Pi_{\bx}\cap \cB(0,2R_{\bx}).
\end{equation} 
We define the quasimode $\psi_{h}$ as in \eqref{eq.defQM} with $\delta=0$, $\dir=0$ and $R=R_{\bx}$. We have $\psi_{h}\in H^1(\Omega)$ and $q_{h}[\bA,\Pi_{\bx}](\psi_{h})=q_{h}[\bA,\Omega](\psi_{h})$. Using \eqref{eq:troncA} and the fact that $\bA$ equals its affine part, we have:
\begin{equation}\label{eq.troncdroit}
 \frac{q_{h}[\bA,\Omega](\psi_{h})}{\|\psi_{h}\|^2} = h \En(\bB,\Pi_{\bx}) + h^2\rho_{h}.
\end{equation}
Applying Lemma~\ref{lem:rhoh} with $\chi_{h}$ as defined in \eqref{eq.chih}, $\delta=0$ and $R=R_{\bx}$, we have
\begin{equation}\label{eq.rhohdroit}
\rho_{h}=\begin{cases} O(1) & \mbox{ if }k<3,\\
O(\re^{-ch^{-1/2}}) & \mbox{ if }k=3.
\end{cases}
\end{equation}
Then, by the min-max principle and \eqref{eq.troncdroit}, we deduce when $k<3$:
$$\lambda_h(\bB,\Omega) \leq h \inf_{\bx\in\overline\Omega}\En(\bB,\Pi_{\bx}) + O(h^2)=h\sE(\bB \ee,\Omega) + C h^2.$$
If there exists $\bv\in\gV$ such that $\En(\bB,\Pi_{\bv})=\sE(\bB,\Omega)<\seE(\bB \ee,\Pi_{\bv})$, we use Theorem~\ref{th:dicho}, Proposition~\ref{prop:cone-ess} and there exists an (admissible generalized) eigenfunction with $k=3$ of $\OP(\bA,\Pi_{\bv})$ for $\En(\bB,\Pi_{\bv})$. According to \eqref{eq.troncdroit} and \eqref{eq.rhohdroit}, we have:
\begin{equation}\label{eq.sommetdroit}
\lambda_h(\bB,\Omega)\leq \frac{q_{h}[\bA,\Omega](\psi_{h})}{\|\psi_{h}\|^2} \leq h\En(\bB,\Pi_{\bv})+C \re^{-ch^{-1/2}}.
\end{equation}
\end{proof}

\section{Lower bound for first eigenvalues}
\label{sec:low}
In this section we give a general lower bound on the first eigenvalue, namely:
\begin{theorem}
\label{T:generalLB}
Let $\Omega\in \ogD(\R^3)$ be a polyhedral domain, $\bA\in W^{2,\infty}(\overline{\Omega})$ be a twice differentiable magnetic potential such that the associated magnetic field $\bB$ does not vanish on $\overline\Omega$. Then there exist $C(\Omega)>0$ and $h_{0}>0$ such that
\begin{equation}
\label{eq:below1}
   \forall h\in (0,h_{0}), \quad \lambda_{h}(\bB,\Omega) \geq
   h\sE(\bB,\Omega)-C(\Omega)(1+\|\bA\|_{W^{2,\infty}(\Omega)}^2)h^{5/4} \ . 
\end{equation}
We recall that the quantity $\sE(\bB,\Omega)$ is the lowest local energy defined in \eqref{eq:s}.
\end{theorem}

\paragraph{Idea of the proof}
We first make a partition of the unity of $\overline{\Omega}$ such that on each element we 
are able to use the change of variable given in \eqref{eq:diffeo}. The local energy of the associated tangent model problem with frozen magnetic field is then bounded from below by $h\sE(\bB,\Omega)$. As above we then estimate the remainders due to the cut-off effects, the change of variables and the linearization of the potential.

\paragraph{IMS localization}
Let $\delta\in(0,\frac12)$. For $h$ small enough, we denote by $\{(\bx_{j},r_{j}),j\in\gJ_{h}\}$ a $h$-dependent finite set of pairs (center, radius) provided by Lemma \ref{lem:IMScov} for $\rho=h^\delta$. Relying on Lemma \ref{lem:IMSpart}, we choose a finite associate partition of the unity $(\tronc)_{j\in\gJ_{h}}$ with 
$\tronc\in \sC^{\infty}_0(\cB(\bx_{j},2r_{j}))$ satisfying
$$ \sum_{j\in\gJ_{h}}\tronc^2=1\quad\mbox{on}\quad \overline\Omega  $$
and the uniform estimate of gradients 
\begin{equation}
\label{E:controlegradtronc}
\exists C>0,\quad \forall h\in (0,h_{0}), \ \forall j\in\gJ_{h},\quad 
\|\nabla\tronc\|_{L^{\infty}(\Omega)}^2 \leq C h^{-2\delta} \, . 
\end{equation}
The IMS formula (see Lemma \ref{lem:IMS}) provides for all $f\in H^1(\Omega)$
\[
   q_{h}[\bA,\Omega](f) = \sum_{j\in\gJ_{h}} q_{h}[\bA,\Omega](\tronc f)
   - h^2 \sum_{j\in\gJ_{h}} \|\nabla\tronc f\|_{L^2(\Omega)}^2
\]
and using \eqref{E:controlegradtronc} we get $C(\Omega)>0$ such that
\begin{equation}
\label{E:minorationpartition}
   q_{h}[\bA,\Omega](f) \geq \sum_{j\in\gJ_{h}} q_{h}[\bA,\Omega](f_{j})
   -C(\Omega) h^{2-2\delta}\|f\|_{L^2(\Omega)}^2 \, .
\end{equation}
where $f_{j}$ denotes the localized function $\tronc f$.

\paragraph{Local control of the energy}
For each $j\in\gJ_{h}$, we estimate the term $q_{h}[\bA,\Omega](f_{j})$ appearing in \eqref{E:minorationpartition}. By construction $\supp(f_{j})\subset \cU_{\bx_{j}}$. Let $\pot{}^{j}$ defined as in \eqref{E:ABtilde} with $\bx_{j}$ playing the same role as $\bx_0$. Lemma \ref{L:changvar} applied with $r=2r_{j}\le Ch^{\delta}$ provides $C(\Omega)>0$ such that
\begin{equation}
\label{E:f-psi}
   \frac{q_{h}[\bA,\Omega](f_{j})}{\|f_{j}\|^2 } \geq 
   (1-C(\Omega)h^{\delta})\frac{q_{h}[\pot{}^{j},\Pi_{\bx_{j}}](\psi_{j})}{\|\psi_{j}\|^2}
\end{equation}
where we have denoted $\psi_{j}=f_{j}\circ (\diffeo^{\bx_{j}})^{-1}$. Let $\pot{}^{j}_{\bfz}$ be the linear part of $\pot{}^{j}$ at the origin $\bfz$ of $\Pi_{\bx_{j}}$. We use \eqref{eq:diffAA'} with $\psi=\psi_{j}$ and $\cO=\Pi_{\bx_{j}}$:
\begin{multline}
\label{eq:diff}
  q_{h}[\pot{}^{j},{\Pi_{\bx_{j}}}](\psi_{j}) =   
  q_{h}[\pot{}^{j}_{\bfz},\Pi_{\bx_{j}}](\psi_{j}) \\
  +2  \Re\big\langle (-ih\nabla+\pot{}^{j}_{\bfz})\psi_{j},(\pot{}^{j}-\pot{}^{j}_{\bfz})\psi_{j}\big\rangle + \|(\pot{}^{j}-\pot{}^{j}_{\bfz})\psi\|^2.
\end{multline}
Therefore using the Cauchy-Schwarz inequality:
\begin{equation*}
  q_{h}[\pot{}^{j},{\Pi_{\bx_{j}}}](\psi_{j}) \geq   
  q_{h}[\pot{}^{j}_{\bfz},\Pi_{\bx_{j}}](\psi_{j})
  -2  \left(q_{h}[\pot{}^{j}_{\bfz},\Pi_{\bx_{j}}](\psi_{j})\right)^{1/2}
  \|(\pot{}^{j}-\pot{}^{j}_{\bfz})\psi_{j}\|\, .
\end{equation*}
We cannot conclude like in \eqref{eq:troncA2} because we do not have any {\it a priori} upper bound on $q_{h}[\pot{}^{j}_{\bfz},\Pi_{\bx_{j}}](\psi_{j})$. That is why we use the parametric estimate
$$
   \forall \eta>0, \quad 
   q_{h}[\pot{}^{j},{\Pi_{\bx_{j}}}](\psi_{j}) \geq  
   (1-\eta) q_{h}[\pot{}^{j}_{\bfz},\Pi_{\bx_{j}}](\psi_{j})
  -\eta^{-1}\|(\pot{}^{j}-\pot{}^{j}_{\bfz})\psi_{j}\|^2
$$
based on the simple inequality $2ab\leq \eta a^2+\eta^{-1}b^2$. Since $\curl \pot{}^{j}_{\bfz}=\bB_{\bx_{j}}$ we have 
$$q_{h}[\pot{}^{j}_{\bfz},\Pi_{\bx_{j}}](\psi_{j}) \geq h\En(\bB_{\bx_{j}},\Pi_{\bx_{j}})\|\psi_{j}\|^2\, .$$
 Moreover using \eqref{E:taylorA} and the same arguments as in Section \ref{SS:linear} we get 
$$
   \|(\pot{}^{j}-\pot{}^{j}_{\bfz})\psi_{j}\|^2 \leq 
   C(\Omega)(1+\| \bA \|_{W^{2,\infty}({\Omega})}^2 )h^{4\delta} \|\psi_{j}\|^2\ . 
$$
We deduce for all $\eta>0$:
$$
   q_{h}[\pot{}^{j},{\Pi_{\bx_{j}}}](\psi_{j}) \geq   
   (1-\eta)h\En(\bB_{\bx_{j}},\Pi_{\bx_{j}})\|\psi_{j}\|^2
  -\eta^{-1}C(\Omega)(1+\| \bA\|_{W^{2,\infty}({\Omega})}^2)h^{4\delta}\|\psi_{j}\|^2 .
$$
Choosing $\eta=h^{2\delta-\frac12}$ we get 
\begin{align}
\label{eq:psijh}
   \frac{q_{h}[\pot{}^{j},{\Pi_{\bx_{j}}}](\psi_{j}) }{\|\psi_{j}\|^2 }
   & \geq  h\En(\bB_{\bx_{j}},\Pi_{\bx_{j}})-C(\Omega)
   (1+\| \bA\|_{W^{2,\infty}({\Omega})}^2)h^{2\delta+\frac12}
\\ 
   & \geq h\sE(\bB,\Omega)-C(\Omega)(1+\| \bA\|_{W^{2,\infty}({\Omega})}^2)
   h^{2\delta+\frac12}.\nonumber
\end{align}

\paragraph{Conclusion}
Combining the previous localized estimate \eqref{eq:psijh} with \eqref{E:f-psi} we deduce:
\begin{equation*}
   \frac{q_{h}[\bA,\Omega](f_{j})}{\|f_{j}\|^2}  \geq
   h\sE(\bB,\Omega)-C(\Omega)(1+\| \bA\|_{W^{2,\infty}({\Omega})}^2)
   (h^{2\delta+\frac12}+h^{1+\delta}).
\end{equation*}
Summing up in $j\in\gJ_{h}$ and using that $\sum_{j\in\gJ_{h}}\|f_{j}\|^2=\|f\|^2$ we obtain
\begin{equation}
\label{E:minoration99}
   \frac{\sum_{j\in\gJ_{h}} q_{h}[\bA,\Omega](f_{j})}{\|f\|^2}  \geq
   h\sE(\bB,\Omega)-C(\Omega)(1+\| \bA\|_{W^{2,\infty}({\Omega})}^2)
   (h^{2\delta+\frac12}+h^{1+\delta}),
\end{equation}
and combining \eqref{E:minoration99} with \eqref{E:minorationpartition} we get $C(\Omega)>0$ such that
\begin{multline}
   \forall f\in H^1(\Omega), \quad \\ \frac{q_{h}[\bA,\Omega](f)}{\|f\|^2}  \geq 
   h\sE(\bB,\Omega)-C(\Omega)(1+\| \bA\|_{W^{2,\infty}({\Omega})}^2)
   \left(h^{2\delta+\frac12}+h^{1+\delta}+h^{2-2\delta}\right).
\end{multline}
We optimize this by taking $\delta=\frac{3}{8}$ and we deduce Theorem \ref{T:generalUB} from the min-max principle.

Like in the last section, we have a result using only a H\"older norm of the magnetic field:

\begin{corollary}
\label{co:T:generalUB}
Let $\Omega\in \ogD(\R^3)$ be a polyhedral domain, $\bB\in W^{1+\alpha,\infty}(\overline{\Omega})$ be a non-vanishing  H\"older continuous magnetic field of order $1+\alpha$ with some $\alpha\in(0,1)$. Then there exist $C(\Omega)>0$ and $h_{0}>0$ such that
\begin{equation}
\label{eq:below1B}
   \forall h\in (0,h_{0}), \quad \lambda_{h}(\bB,\Omega) \geq
   h\sE(\bB,\Omega)-C(\Omega)(1+\|\bB\|_{W^{1+\alpha,\infty}(\Omega)}^2)h^{5/4} \ . 
\end{equation}
\end{corollary}

%
%
%
%
%

\appendix
\section{Technical lemmas}
\label{sec:techniq}

\subsection{Gauge transform}
\begin{lemma}\label{lem:gauge}
Let $\cO \subset \R^3$ be a domain and let $\vartheta$ be a regular function on $\overline\cO$. Let $\bA$ be a regular potential. Then
\begin{equation*}
\forall \psi\in \dom(q_{h}[\bA,\cO]),\quad q_{h}[\bA+\nabla\vartheta,\cO](\re^{-i\vartheta/h}\psi) =  q_{h}[\bA,\cO](\psi).
\end{equation*}
Furthermore a function $\psi$ is an eigenfunction for the operator $\OP_h(\bA,\cO)$ if an only if $\re^{-i\vartheta/h}\psi$ is an eigenfunction for $\OP_{h}(\bA+\nabla\vartheta,\cO)$ associated with the same eigenvalue.
\end{lemma}
\begin{proof}
The commutation formula 
\[(-ih\nabla+\bA+\nabla{\vartheta})\left(\re^{-i\vartheta/h}\psi\right)
= \re^{-i\vartheta/h} (-ih\nabla+\bA)\psi
\]
yields the result.
\end{proof}

For the sake of completeness we provide the following standard lemma describing the effect of a translation when the potential is affine:

\begin{lemma}[Translation]\label{lem.transl}
Let $\cO \subset \R^3$ be a domain and $\bA$ be an affine magnetic potential. Let $\bd\in \R^3$ be a vector and $\cO_{\bd}:=\cO+\bd$ be the translated domain. For $\psi\in \dom(\OP_{h}(\bA,\cO))$, we define the translated function on $\cO_{\bd}$ by
\[
   \psi_{\bd}:\bx\mapsto \re^{-\ri\langle\bA(\bd)-\bA(\bfz),\bx\rangle/h}\psi(\bx-\bd)\,.
\]
Then
$ q_{h}[\bA,\cO_{\bd}](\psi_{\bd}) =  q_{h}[\bA,\cO](\psi)$
 and $\psi$ is an eigenfunction of $\OP_{h}(\bA,\cO)$ if and only if $\psi_{\bd}$ is an eigenfunction of $\OP_{h}(\bA,\cO_{\bd})$.
\end{lemma}

\begin{lemma} 
\label{L:d2ell0}
Let $\cO$ be a bounded domain such that $\bfz\in\overline\cO$. 
Let $\bA\in W^{3,\infty}(\cO)$ be a magnetic potential such that $\bA(\bfz)=0$. Let $\bA_{\bfz}$ denote the linear
part of $\bA$ at $\bfz$. Let $\ell$ be an index in $\{1,2,3\}$. 

(a) There exists a change of gauge $\nabla F$ where $F$ is a polynomial function of degree $3$, so that
\begin{enumerate}
\item The linear part of $\bA-\nabla F$ at $\bfz$ is still $\bA_\bfz$,
\item The second derivative of $\bA-\nabla F$ with respect to $u_\ell$ cancels at $\bfz$:
\begin{equation*}
   \partial^2_{u_\ell}(\bA-\nabla F)(\bfz) = 0.
\end{equation*}
\item The coefficients of $F$ are bounded by $\|\bA\|_{W^{2,\infty}(\cO)}$.
\end{enumerate}
(b) Let us choose $\ell=1$ for instance. We have the estimate
\begin{equation}
\label{eq:d2ell0est}
   |\bA(\bu) - \bA_\bfz(\bu) - \nabla F(\bu)| \le C(\cO)\,\|\bA\|_{W^{3,\infty}(\cO)}
   \big(|u_1|^3 + |u_1u_2| + |u_1u_3| + |u_2|^2 + |u_3|^2\big)\,,
\end{equation}
where the constant $C(\cO)$ depends only on the outer diameter of $\cO$.
\end{lemma}

\begin{proof}
The Taylor expansion of $\bA$ at $\bfz$ takes the form
\[
   \bA = \bA_\bfz + \bA^{(2)} + \bA^{(\mathrm{rem},3)},
\]
where $\bA^{(2)}$ is a homogeneous polynomial of degree $2$ with $3$ components and $\bA^{(\mathrm{rem},3)}$ is a remainder:
\begin{equation}
\label{eq:rem3}
   |\bA^{(\mathrm{rem},3)}(\bu)| \le 
   \|\bA\|_{W^{3,\infty}(\cO)} |\bu|^3
   \quad\mbox{for}\quad\bu\in\cO.
\end{equation}
Let us write the $m$-th component $A^{(2)}_{m}$ of $\bA^{(2)}$ as
\[
   A^{(2)}_{m}(\bu) = 
   \sum_{|\alpha|=2} a_{m,\alpha} u_1^{\alpha_1}u_2^{\alpha_2}u_3^{\alpha_3}
   \quad\mbox{for}\quad\bu=(u_1,u_2,u_3)\in\cO .
\]

(a) Now, the polynomial $F$ can be explicitly determined. It suffices to take
\[
   F(\bu) =
   u_\ell^2 \big(a_{1,\alpha^*}u_1 +  a_{2,\alpha^*}u_2 + a_{3,\alpha^*}u_3 
   - \tfrac{2}{3} a_{\ell,\alpha^*}u_\ell \big),
\]
where $\alpha^*$ is such that $\alpha^*_\ell=2$ (and the other components are $0$). Then
\[
   \nabla F(\bu) =  
   u_\ell^2
   \begin{pmatrix}
   a_{1,\alpha^*} \\ a_{2,\alpha^*} \\ a_{3,\alpha^*} 
   \end{pmatrix}
\]
and point (a) of the lemma is proved.

(b) Choosing $\ell=1$, we see that the $m$-th components of $\bA^{(2)}-\nabla F$ is
\begin{multline*}
   A^{(2)}_{m}(\bu) - (\nabla F)_m(\bu) \\= 
   a_{m,(1,1,0)} u_1u_2 + a_{m,(1,0,1)} u_1u_3 + a_{m,(0,1,1)} u_2u_3 + 
   a_{m,(0,2,0)} u^2_2 + a_{m,(0,0,2)} u^2_3 \,.
\end{multline*}
Hence $\bA^{(2)}-\nabla F$ satisfies the estimate
\[
   |(\bA^{(2)}(\bu)-\nabla F(\bu)| \le \|\bA\|_{W^{2,\infty}(\cO)}
   \big(|u_1u_2| + |u_1u_3| + |u_2|^2 + |u_3|^2\big).
\]
But
\[
   \bA - \bA_\bfz - \nabla F = \bA^{(2)} - \nabla F + \bA^{(\mathrm{rem},3)}.
\]
Therefore, with \eqref{eq:rem3}
\[
   |\bA(\bu) - \bA_\bfz(\bu) - \nabla F(\bu)| \le 
   \|\bA\|_{W^{2,\infty}(\cO)}
   \big(|u_1u_2| + |u_1u_3| + |u_2|^2 + |u_3|^2\big) +
   \|\bA\|_{W^{3,\infty}(\cO)} |\bu|^3.
\]
Using finally that $|\bu|^3\le 12(|u_1|^3 + |u_2|^3 + |u_3|^3) \le  C(\cO) (|u_1|^3 + |u_2|^2 + |u_3|^2)$, we conclude the proof of estimate \eqref{eq:d2ell0est}.
\end{proof}

\subsection{Change of variables}\label{SA:CV}

Let $\rG$ be a metric of $\R^3$, that is a $3\times3$ positive symmetric matrix with regular coefficients. For a smooth magnetic potential, the quadratic form of the associated magnetic Laplacian with the metric $\rG$ is denoted by $q_{h}[\bA,\cO,\rG]$ and is defined in \eqref{D:fqG}.
The following lemma describes how this quadratic form is involved when using a change of variables: 

\begin{lemma}\label{L:chgvar}
Let $\diffeo:\cO\to\cO'$, $\bu\mapsto\bv$ be a diffeomorphism with $\cO,\cO'$ domains. 
We denote by 
$$\rJ:=\rd (\diffeo^{-1})$$ the jacobian matrix of the inverse of $\diffeo$. 
Let $\bA$ be a smooth magnetic potential and $\bB=\curl\bA$ the associated magnetic field. 
Let $f$ be a function of $\dom(q_{h}[\bA,\cO])$ and $\psi:=f\circ \diffeo^{-1}$ defined in $\cO'$. 
For any $h>0$ we have 
\begin{equation}
\label{E:chgG}
q_{h}[\bA,\cO](f)=q_{h}[\pot,\cO',\rG](\psi) \quad \mbox{and} \quad \| f\|_{L^2(\cO)}=\| \psi\|_{L^2_{\rG}(\cO')}
\end{equation}
where the new magnetic potential and the metric are respectively given by 
\begin{equation}
\label{E:Atilde}
   \pot:=\rJ^{\top} \big(\bA\circ \diffeo^{-1} \big) \quad 
   \mbox{and}\quad \rG:=\rJ^{-1}(\rJ^{-1})^{\top} \ . 
\end{equation}
The magnetic field $\tbB=\curl\pot$ in the new variables is given by 
\begin{equation}
\label{D:tbB}
\tbB:=|\det \rJ|\,\rJ^{-1} \big(\bB\circ \diffeo^{-1} \big).
\end{equation}
\end{lemma}

Let $\rho>0$, using the previous Lemma with the scaling $\diffeo^{\rho}:=\bx \mapsto \sqrt{\rho}\,\bx$ we get 
\begin{lemma}\label{lem.dilatation}
Let $\cO\subset \R^3$ be a domain and for $r>0$, we denote by $r\cO$ the domain $\{\bx\in\R^3,\ \bx=r\bx'\ \mbox{with}\ \bx'\in\cO\}$. Let $\bB$ be a constant magnetic field. Then
\[
\forall \rho>0,\quad   \En(\bB \ee,\cO) = \rho\,\En\Big(\frac\bB\rho \ee,\sqrt\rho\,\cO\Big)\,.
\]
\end{lemma}

\subsection{Miscellaneous}
\paragraph{Orientation of the magnetic field}  
Let $\bB$ be a magnetic field. It is known that changing $\bB$ into $-\bB$ does not affect the spectrum of the associated magnetic Laplacian. More precisely we have:
\begin{lemma}
\label{lem:sense}
Let $\cO\subset \R^3$ be a domain, $\bB$ be a magnetic field and $\bA$ an associated potential. Then $\OP_{h}(-\bA,\cO)$ and $\OP_{h}(\bA,\cO)$ are unitary equivalent. We have
$$\forall \psi \in \dom(q_{h}[\bA,\cO]), \quad q_{h}[-\bA,\cO](\overline{\psi})=q_{h}[\bA,\cO](\psi)$$
and $\psi$ is an eigenfunction of $\OP_{h}(\bA,\cO)$ if and only if $\overline\psi$ is an eigenfunction of $\OP_{h}(-\bA,\cO)$.
\end{lemma}

\paragraph{Model linear potential}
Let us remark that if $\bB$ is a constant magnetic field, an associated magnetic potential is given by 
\begin{equation}
\label{PotS}
\bA^{\rm S}(\bx):=\frac{1}{3}\bB\wedge \bx \ . 
\end{equation}
Indeed we have 
$$\curl \bA^{\rm S}=\frac{1}{3}\nabla \wedge(\bB \wedge \bx)=\frac{1}{3}\left( (\nabla\cdot \bx)\bB- (\nabla\cdot\bB)\bx\right)=\bB \ .$$

\paragraph{Comparison between two potentials}

 Let $\cO\subset \R^3$ be a domain and let $\bA$ and $\bA'$ be two magnetic potentials. Then, for any function $\psi$ of $\dom(q_{h}[\bA,{\cO}])\cap \dom(q_{h}[\bA',{\cO}])$, we have:
\begin{equation}
\label{eq:diffAA'}
  q_{h}[\bA,{\cO}](\psi) =   q_{h}[\bA',{\cO}](\psi)
  + 2  \Re\int_{\cO} (-ih\nabla+\bA')\psi(\bx)\cdot(\bA-\bA')(\bx)\overline{\psi(\bx)}\,\rd\bx
  + \|(\bA-\bA')\psi\|^2\, .
\end{equation}

\subsection{Cut-off effect}
In this section we recall standard IMS\footnote{IMS stands for Ismagilov-Morgan-Sigal (or Simon)} formulas. This kind of formulas appear for Schr\"odinger operators in \cite{CyFrKiSi87}, but they can also be found in older works like \cite{Mel71}. In this section $\bA$ denotes a regular magnetic potential and $\cO$ a generic domain of $\R^3$.

The first formula describes the effect of a partition of the unity on the energy of a function which is in the form domain, see for example \cite[Lemma 3.1]{Si82}:
\begin{lemma}[IMS formula]\label{lem:IMS}
Assume that $\chi_1,\ldots,\chi_L\in\sC^\infty(\overline\cO)$ are such that
\[
   \sum_{\ell=1}^L \chi_\ell^2 \equiv 1 \quad\mbox{on}\quad\cO.
\]
Then, for any $\psi\in \dom(q_{h}[\bA,\cO])$ 
\[
   q_{h}[\bA,\cO](\psi) = \sum_{\ell=1}^L q_{h}[\bA,\cO](\chi_\ell \psi)
   - h^2 \sum_{\ell=1}^L \|\psi\nabla\chi_\ell\|_{L^2(\cO)}^2
\]
\end{lemma}
 Recall that $\dom_{\,\loc} (\OP_{h}(\bA,\cO))$ denotes the functions that are locally in the domain of the operator, see \eqref{D:domloc}. In particular they satisfy the Neumann boundary condition. The second formula describes the energy of such a function when applying a cut-off function, see for example \cite[(6.11)]{HeMo01}: 
\begin{lemma}
\label{lem:tronc}
Let  $\chi\in\sC^\infty_0(\overline{\cO})$ a real smooth function. Then for any $\psi\in\dom_{\,\loc} (\OP_{h}(\bA,\cO))$
\begin{equation}
\label{eq:tronc}
  q_{h}[\bA,\cO](\chi \psi) = 
  \Re \int_{\cO}\chi(\bx)^2 \OP_{h}(\bA,\cO) \psi(\bx) \ \overline\psi(\bx)\,\rd\bx
  + h^2\| \nabla\chi\,\psi \|^2_{L^2(\cO)} \, .
\end{equation}
\end{lemma}

\section{Partition of unity suitable for IMS formulas}

We need a preliminary definition.

\begin{definition}
\label{def:map-neigh}
Let $\Omega\in\gD(M)$ with $M=\R^n$ or $M=\dS^n$. Let $\bx\in\overline\Omega$ and $\cU$ be an open neighborhood of $\bx$ in $M$. We say that {\em $\cU$ is a map-neighborhood of $\bx$ for $\Omega$} if there exists  a local smooth diffeomorphism $\diffeo^\bx$ which maps the neighborhood $\cU$  onto a neighborhood $\cV$ of $0$ in $\R^n$ and such that
\begin{equation}
\label{eq:map-neigh}
   \diffeo^\bx(\cU\cap\Omega) = \cV\cap\Pi_\bx \quad\mbox{and}\quad
   \diffeo^\bx(\cU\cap\partial\Omega) = \cV\cap\partial\Pi_\bx \,,
\end{equation}
where $\Pi_\bx$ is the tangent cone to $\Omega$ at $\bx$ (compare with \eqref{eq:diffeo}).
\end{definition}

\begin{lemma}
\label{lem:IMScov}
Let $n\ge1$ be the space dimension. $M$ denotes $\R^n$ or $\dS^n$. 
Let $\Omega\in\gD(M)$ and $K>1$. There exist a positive integer $L$ and two positive constants $\rho_{\max}$ and $\kappa\le1$ (depending on $\Omega$ and $K$) such that for all $\rho\in(0,\rho_{\max}]$, there exists a (finite) set $\sZ\subset \overline\Omega \times [\kappa\rho,\rho]$ satisfying the following three properties
\begin{enumerate}
\item We have the inclusion $\overline\Omega \subset \cup_{(\bx,r)\in\sZ}\,\overline\cB(\bx,r)$
\item For any $(\bx,r)\in\sZ$, the ball $\cB(\bx,Kr)$ is a map-neighborhood of $\bx$ for $\Omega$ 
\item Each point $\bx_0$ of $\overline\Omega$ belongs to at most $L$ different balls $\cB(\bx,Kr)$.
\end{enumerate}
\end{lemma}

Here are preparatory notations and lemmas.

Let $\Omega\in\gD(M)$ and $K>1$. If the assertions of Lemma \ref{lem:IMScov} are true for this $\Omega$ and this $K$, we say that Property $\sP(\Omega,K)$ holds. We may also specify that the assertion by the sentence
\[
   \mbox{Property $\sP(\Omega,K)$ holds with parameters $(L,\rho_{\max},\kappa)$.}
\]
Let $\cU^*\subset\subset\cU$ be two nested open sets. We say that the property $\sP(\Omega,K;\cU^*,\cU)$ holds\footnote{This is the localized version of property $\sP(\Omega,K)$.} if the assertions of Lemma \ref{lem:IMScov} are true for this $\Omega$ and this $K$, with discrete sets $\sZ\subset (\cU^*\cap\overline\Omega) \times [\kappa_\Omega\rho,\rho]$ and with (1)-(3) replaced by
\begin{enumerate}
\item We have the inclusion $\cU^*\cap\overline\Omega \subset \cup_{(\bx,r)\in\sZ}\,\overline\cB(\bx,r)$
\item For any $(\bx,r)\in\sZ$, the ball $\cB(\bx,Kr)$ is included in $\cU$ and is a map-neighborhood of $\bx$ for $\Omega$ 
\item Each point $\bx_0$ of $\cU\cap\overline\Omega$ belongs to at most $L$ different balls $\cB(\bx,Kr)$.
\end{enumerate}
Like above the specification is
\[
   \mbox{Property $\sP(\Omega,K;\cU^*,\cU)$ holds with parameters $(L,\rho_{\max},\kappa)$.}
\]
In the process of proof, we will construct coverings which are not exactly balls, but domains uniformly comparable to balls. Let us introduce the local version of this new assertion. For $0<a\le a'$ we say that
\[
   \mbox{Property $\sP[a,a'](\Omega,K;\cU^*,\cU)$ holds with parameters $(L,\rho_{\max},\kappa)$}
\]
if for all $\rho\in(0,\rho_{\max}]$, there exists a finite set $\sZ\subset (\cU^*\cap\overline\Omega )\times [\kappa_\Omega\rho,\rho]$ and open sets $\cD(\bx,r)$ satisfying the following four properties
\begin{enumerate}
\item We have the inclusion $\cU^*\cap\overline\Omega \subset \cup_{(\bx,r)\in\sZ}\,\overline\cD(\bx,r)$
\item For any $(\bx,r)\in\sZ$, the set\footnote{Here $\cD(\bx,Kr)$ is the set of $\by$ such that $\bx+(\by-\bx)/K\in\cD(\bx,r)$.} $\cD(\bx,Kr)$ is included in $\cU$ and is a map-neighborhood of $\bx$ for $\Omega$ 
\item Each point $\bx_0$ of $\cU\cap\overline\Omega$ belongs to at most $L$ different sets $\cD(\bx,Kr)$
\item For any $(\bx,r)\in\sZ$, we have the inclusions
$\cB(\bx,ar) \subset\cD(\bx,r) \subset\cB(\bx,a'r)$.
\end{enumerate}
Note that $\sP[1,1](\Omega,K;\cU^*,\cU)=\sP(\Omega,K;\cU^*,\cU)$.

\begin{lemma}
\label{lem:aa'}
If Property $\sP[a,a'](\Omega,K;\cU^*,\cU)$ holds with parameters $(L,\rho_{\max},\kappa)$, then
\[
   \mbox{Property $\sP(\Omega,\frac{a}{a'}K;\cU^*,\cU)$ 
   holds with parameters $(L,a'\rho_{\max},\kappa)$.}
\]
\end{lemma}

\begin{proof}
Starting from the covering of $\cU^*\cap\overline\Omega$ by the sets $\overline\cD(\bx,r)$ and using condition (4), we can consider the covering of $\cU^*\cap\overline\Omega$ by the balls $\cB(\bx,a'r)$. Then $r':=a'r\in[\kappa a'\rho,a'\rho]=[\kappa\rho',\rho']$ with $\rho'<a'\rho_{\max}$. 

Concerning conditions (2) and (3), it suffices to note the inclusions
\[
   \cB(\bx,\frac{a}{a'}Kr') \subset \cD(\bx,\frac{1}{a'}r'K) = \cD(\bx,rK) \,.
\]
The lemma is proved.
\end{proof}

\begin{proof} {\em \!\!\!of Lemma \ref{lem:IMScov}.\ }
The principle of the proof is a recursion on the dimension $n$.

{\em Step 1.\ } Explicit construction when $n=1$.\\ 
The domain $\Omega$ and the localizing open sets $\cU^*$ and $\cU$ are then open intervals. Let us assume for example that $\cU^*=(-\ell,\ell)$, $\cU=(-\ell-\delta,\ell+\delta)$ and $\Omega=(0,\ell+\delta')$ with $\ell,\delta>0$ and $\delta'>\delta$. Let $K\ge1$. We can take
\[
   \rho_{\max} = \min\Big\{\frac{\ell}{K},\delta\Big\}
\]
and for any $\rho\le\rho_{\max}$ the following set of couples $(\bx_j,r_j)$, $j=0,1,\ldots,J$
\[
   \bx_0 = 0, \ r_0=\rho \quad\mbox{and}\quad
   \bx_j=\rho+\frac{2j-1}{K}\rho, \ r_j=\frac{\rho}{K}\quad\mbox{for}\quad j=1,\ldots,J
\]
with $J$ such that $\bx_J<\ell$ and $\rho+\frac{2J+1}{K}\rho\ge\ell$. If $\bx_J<\ell-\frac{\rho}{K}$, we add the point $\bx_{J+1}= \rho+\frac{2J}{K}\rho$. The covering condition (1) is obvious.

Concerning condition (2), we note that the bound $\rho_{\max}\le \frac{\ell}{K}$ implies that $[0,Kr_0)=[0,K\rho)$ is a map-neighborhood for the boundary of $\Omega$, and the bound $\rho_{\max}\le \delta$ implies that when $j\ge1$, the ``balls'' $(\bx_j-Kr_j,\bx_j+Kr_j)=(\bx_j-\rho,\bx_j+\rho)$ are map-neighborhoods for the interior of $\Omega$.

Concerning condition (3), we can check that $L=K+2$ is suitable.

{\em Step 2.\ } Localization.\\ 
Let $\Omega\in\gD(\R^n)$ or $\Omega\in\gD(\dS^n)$. For any $\bx\in\overline\Omega$, there exists a ball $\cB(\bx,r_x)$ with positive radius $r_\bx$ that is a map-neighborhood for $\Omega$. We extract a finite covering of $\overline\Omega$ by open sets $\cB(\bx^{(\ell)},\frac12 r^{(\ell)})$. We set
\[
   \cU^*_\ell = \cB(\bx^{(\ell)},\frac12 r^{(\ell)})\quad\mbox{and}\quad
   \cU_\ell = \cB(\bx^{(\ell)},r^{(\ell)}).
\]
The map $\diffeo^\ell:=\diffeo^{\bx^{(\ell)}}$ transforms $\cU^*_\ell$ and $\cU_\ell$ into neighborhoods $\cV^*_\ell$ and $\cV_\ell$ of $0$ in the tangent cone $\Pi_\ell:=\Pi_{\bx^{(\ell)}}$. Thus we are reduced to prove the local property $\sP(\Pi_\ell,K;\cV^*_\ell,\cV_\ell)$ for any $\ell$. Indeed
\begin{itemize}
\item The local diffeomorphism $\diffeo^\ell$ allows to deduce Property $\sP(\Omega,K;\cU^*_\ell,\cU_\ell)$ from Property $\sP(\Pi_\ell,K';\cV^*_\ell,\cV_\ell)$ for a ratio $K'/K$ that only depends on $\diffeo^\ell$ (this relies on Lemma \ref{lem:aa'}).
\item Properties $\sP(\Omega,K;\cU^*_\ell,\cU_\ell)$ imply Property $\sP(\Omega,K;\cup_\ell\,\cU^*_\ell,\cup_\ell\,\cU_\ell) = \sP(\Omega,K)$ (it suffices to merge the (finite) union of the sets $\sZ$ corresponding to each $\cU_\ell$).
\end{itemize}

{\em Step 3.\ } Core recursive argument: If $\Omega_0$ is the section of the cone $\Pi$, Property $\sP(\Omega_0,K)$ implies Property $\sP(\Pi,K';\cB(0,1),\cB(0,2))$ for a suitable ratio $K'/K$. We are going to prove this separately in several lemmas (\ref{lem:tens} to \ref{lem:cone}). Then the proof Lemma \ref{lem:IMScov} will be complete.
\end{proof}

\begin{lemma}
\label{lem:tens}
Let $\Gamma$ be a cone in $\gP^{n-1}$. For $\ell=1,2$, let $\cB_\ell$ and $\cI_\ell$ be the ball $\cB(0,\ell)$ of $\R^{n-1}$ and the interval $(-\ell,\ell)$, respectively. We assume that Property $\sP(\Gamma,K;\cB_1,\cB_2)$ holds (with parameters $(L,\rho_{\max},\kappa)$). Then Property $\sP[1,\sqrt2](\Gamma\times\R,K;\cB_1\times\cI_1,\cB_2\times\cI_2)$ holds.
\end{lemma}

\begin{proof}
Let us denote by $\by$ and $z$ coordinates in $\Gamma$ and $\R$, respectively. For $\rho\le\rho_{\max}$, let $\sZ_\Gamma$ be an associate set of couples $(\by,r_\by)$. For each $\by$ we consider the unique set of equidistant points $\sZ_\by=\{z_j\in[-1,1],\ j=1,\ldots,J_\by\}$ such that
\[
   z_j-z_{j-1}=2r_\by \quad\mbox{and}\quad z_1+1=1-z_{J_\by}<r_\by\,.
\]
Then we define
\begin{equation}
\label{eq:Zprod}
   \sZ^{(\rho)} = \big\{(\bx,r_\bx),\quad\mbox{for}\ \ 
   \bx=(\by,z) \ \mbox{with}\   
   (\by,r_\by)\in\sZ_\Gamma,\ z\in\sZ_\by \ \mbox{and}\ r_\bx=r_\by\big\}.
\end{equation}
The associate open set $\cD(\bx,r_\bx)$ is the product
\[
   \cD(\bx,r_\bx) = \cB(\by,r_\by)\times (z-r_\by,z+r_\by)\,.
\]
We have the inclusions $\cB(\bx,r_\bx)\subset\cD(\bx,r_\bx)\subset\cB(\bx,\sqrt2\, r_\bx)$ and it is easy to check that Property $\sP[1,\sqrt2](\Gamma\times\R,K;\cB(0,1)\times\cI_1,\cB(0,2)\times\cI_2)$ holds with parameters $(L',\rho_{\max},\kappa)$ with $L'=LK$.
\end{proof}

\begin{lemma}
\label{lem:annulus}
Let $\Omega$ be a section in $\gD(\dS^{n-1})$, let $\Pi$ be the corresponding cone, and let $\cI_\ell$ be the interval $(2^{-\ell},2^\ell)$ for $\ell=1,2$. We define the annuli 
\[
   \cA_\ell = \Big\{\bx\in\Pi, \quad 
   |\bx|\in\cI_\ell \ \ \mbox{and}\ \ \frac{\bx}{|\bx|}\in\Omega\Big\}.
\]
We assume that Property $\sP(\Omega,K)$ holds (with parameters $(L,\rho_{\max},\kappa)$). Then Property $\sP[a,a'](\Pi,K;\cA_1,\cA_2)$ holds for suitable constants $a$ and $a'$ (independent of $\Omega$ and $K$). 
\end{lemma}

\begin{proof}
Let us consider the diffeomorphism
\begin{equation}
\label{eq:diffeoT}
   \begin{aligned}
   \diffeoT : \ \ &\Omega\times(-2,2) &\longrightarrow\ \ \ &\cA_2\\
   &\bx = (\by,z) &\longmapsto \ \ \ &\breve\bx = 2^z\by
   \end{aligned}
\end{equation}
in view of proving Property $\sP[a,a'](\Pi,K;\cA_1,\cA_2)$, for a given $\rho\le\rho_{\max}$, we define a suitable set $\breve\sZ^{(\rho)}$ using the set $\sZ^{(\rho)}$ introduced in \eqref{eq:Zprod}
\begin{equation}
\label{eq:Zann}
   \breve\sZ^{(\rho)} = \big\{(\breve\bx,r_\bx),\quad\mbox{for}\ \ 
   \breve\bx=\diffeoT\bx  \ \ \mbox{with}\ \ 
   (\bx,r_\bx)\in\sZ^{(\rho)} \big\},
\end{equation} 
and the associated open sets
\[
   \breve\cD(\breve\bx,r_\bx) = \diffeoT \big(\cD(\bx,r_\bx)\big) .
\] 
We can check that
\[
   \cB(\breve\bx,ar_\bx) \subset \breve\cD(\breve\bx,r_\bx) \subset
   \cB(\breve\bx,a'r_\bx) 
\]
with $a = \frac18\log2$ and $a'=8\sqrt2\log2$ and that Property $\sP[a,a'](\Pi,K;\cA_1,\cA_2)$ holds with parameters $(L',\rho_{\max},\kappa)$ for $L'=NLK$ with an integer $N$ independent of $L$ and $K$.
\end{proof}

\begin{lemma}
\label{lem:cone}
Let $\Omega$ be a section in $\gD(\dS^{n-1})$, let $\Pi$ be the corresponding cone, and let $\cB_\ell$ be the balls $\cB(0,\ell)$ of $\R^{n}$ for $\ell=1,2$. We assume that Property $\sP(\Omega,K)$ holds with parameters $(L,\rho_{\max},\kappa)$ for a $\rho_{\max}\le1$. Then Property $\sP[a,a'](\Pi,K;\cB_1,\cB_2)$ holds for suitable constants $a$ and $a'$ (independent of $\Omega$ and $K$) and with parameters $(L',1,\kappa\rho_{\max})$. 
\end{lemma}

\begin{proof}
Let $\rho\le1$ and let $M$ be the natural number such that
\[
   2^{-M-1} <\rho\le 2^{-M}.
\]
On the model of \eqref{eq:diffeoT}-\eqref{eq:Zann}, we set
\[
   \breve\sZ^m = \big\{(2^{-m}\diffeoT\bx,2^{-m}r_\bx),\ \ \mbox{with}\ \
   (\bx,r_\bx)\in\sZ^{(2^m\rho_{\max}\rho)} \big\}, \quad
   m=0,\ldots,M,
\]
and the associated open sets are
\begin{equation}
\label{eq:Zm}
   2^{-m}\diffeoT \big(\cD(\bx,r_\bx)\big)\ \ \mbox{with}\ \
   (\bx,r_\bx)\in\sZ^{(2^m\rho_{\max}\rho)} .
\end{equation}
The set $\breve\sZ$ associated with the cone $\Pi$ in the ball $\cB_1$ is
\[
   \{(0,\rho)\} \cup \bigcup_{m=0}^M \breve\sZ^m
\]
and the associated open sets are the reunion of the sets \eqref{eq:Zm} for $m=0,\ldots,M$ and of the ball $\cB(0,\rho)$. As the radii $r_\bx$ belong to $[\kappa 2^m\rho_{\max}\rho,2^m\rho_{\max}\rho]$, we have $2^{-m}r_\bx\in[\kappa \rho_{\max}\rho,\rho_{\max}\rho]$. Since $\rho$ itself belongs to the full collection of radii $r$, we finally find $r\in[\kappa \rho_{\max}\rho,\rho]$. The finite covering holds with $L'=3NLK+1$ for the same integer $N$ appearing at the end of the proof of Lemma \ref{lem:annulus}.
\end{proof}

\begin{lemma}
\label{lem:IMSpart}
Let $\Omega\in\gD(\R^n)$.  Let $(L,\rho_{\max},\kappa)$ be the parameters provided by Lemma \ref{lem:IMScov}, for the Property $\sP(\Omega,2)$ to hold. For any $\rho\in(0,\rho_{\max}]$ let $\sZ\subset \overline\Omega \times [\kappa\rho,\rho]$ be an associate set of pairs (center, radius). Then there exists a collection of smooth functions $(\troncg)_{(\bx,r)\in\sZ}$ with $\troncg\in \sC^{\infty}_0(\cB(\bx,2r))$ satisfying the identity (partition of unity) 
$$
   \sum_{(\bx,r)\in\sZ}\troncg^2=1  \quad\mbox{on}\quad\overline\Omega
$$
and the uniform estimate of gradients 
\begin{equation*}
\exists C>0,\quad \forall (\bx,r)\in\sZ, \quad 
\|\nabla\troncg\|_{L^{\infty}(\Omega)} \leq C \rho^{-1} \, ,
\end{equation*}
where $C$ only depends on $\Omega$. By construction any ball $\cB(\bx,2r)$ is a map-neighborhood of $\bx$ for $\Omega$.
\end{lemma}

\begin{proof}
Let $\troncp\in \sC^{\infty}_0(\cB(\bx,2r))$, with the property that $\troncp\equiv1$ in $\cB(\bx,r)$, and satisfying
\[
   \|\nabla\troncp\|_{L^{\infty}(\R^3)} \leq C r^{-1}
\] 
where $C$ is a universal constant. Then we set for each $(\bx_0,r_0)\in\sZ$
\[
   \chi_{(\bx_{0},r_{0})} = \frac{\troncz}{(\sum_{(\bx,r)\in\sZ}\troncp^2)^{1/2}}\ .
\]
Due to property (1) in Lemma \ref{lem:IMScov}, $\sum_{(\bx,r)\in\sZ}\troncp^2\ge1$ and due to property (3), 
\[
   \|\sum_{(\bx,r)\in\sZ}\nabla\troncp^2\|_{L^{\infty}(\R^3)} \le CL_\Omega\,.
\]
We deduce the lemma.
\end{proof}

\end{document}